\documentclass[12pt,reqno]{amsart}
\usepackage{amsmath}
\usepackage{txfonts}
\usepackage{amsfonts}
\usepackage{a4wide}
\usepackage{amsmath,amsthm,amssymb,amscd}
\usepackage{latexsym}
\usepackage{hyperref}
\usepackage[numbers,sort&compress]{natbib}
\usepackage{hypernat}
\usepackage{color,enumitem,graphicx}
\usepackage{soul}
\allowdisplaybreaks
\numberwithin{equation}{section}

\newtheorem{theorem}{Theorem}[section]

\newtheorem{corollary}[theorem]{Corollary}
\newtheorem{lemma}[theorem]{Lemma}

\theoremstyle{definition}

\newtheorem{remark}[theorem]{Remark}

\begin{document}

\title[Existence and Uniqueness of Normalized Solutions for CNLS Systems]{Existence and Uniqueness of Normalized Multi-peak Solutions for Coupled Nonlinear Schr\"odinger Systems}
\author{Wenhao Hu \textsuperscript{1}, Benniao Li \textsuperscript{1}, Wei Long\textsuperscript{1}, Chunhua Wang\textsuperscript{2, }$^{\dagger}$ }
\thanks{ \textsuperscript{1}  \small  \it School of Mathematics and Statistics, and Jiangxi Provincial Center for Applied Mathematics,  Jiangxi Normal University, Nanchang, Jiangxi 330022, People' s Republic of China;}
\thanks{ \textsuperscript{2}  \small \it  School of Mathematics and Statistics, Key Lab NAA--MOE,
 Central China Normal University, Wuhan 430079, China}
\thanks{\small $^{\dagger}$Corresponding author. E-mail: chunhuawang@ccnu.edu.cn(C.~Wang);}
\thanks{\small Contributing authors:~huwenhao98@163.com(W.~Hu);~benniao\_li@jxnu.edu.cn(B.~Li); lwhope@jxnu.edu.cn(W.~Long).}

\maketitle

{\bf Abstract.}\ We consider the following two-component coupled nonlinear Schr\"odinger (CNLS) system:
\[
\begin{cases}
-\Delta u +(P(x) + \lambda ) u=\mu_1 u^3+\beta u v^2, & \text{in } \mathbb{R}^N,\\
-\Delta v +(Q(x) + \lambda ) v =\mu_2 v^3+\beta vu^2, & \text{in } \mathbb{R}^N
\end{cases}
\]
with the mass constraint $\int_{\mathbb{R}^N} (u^2+v^2)\,dx = \rho^2$ for $N=2,3$, where $\rho>0$ is a parameter. By employing the Lyapunov-Schmidt reduction and local Pohozaev identities, we establish the existence and local uniqueness of normalized multi-peak solutions: the result holds for sufficiently small $\rho$ when $N=3$, and for $\rho$ approaching a critical threshold when $N=2$. The main difficulty lies in that the mass constraint involves interactions among all concentration points, while a more refined characterization of such normalized solutions further requires sharp order estimates. In this work, we have discovered some new phenomena that differ from those of solutions without mass constraint and single-peak solutions.


\section{Introduction}

It is concerned with the problem arising from the time-dependent coupled Schr\"odinger system modeling two-component Bose-Einstein condensates (BECs)\cite{ morsch2006dynamics, bloch2022non} and nonlinear optical wave propagation \cite{marin2017microresonator, cao2023spatiotemporal}:
\begin{equation} \label{eq01}
\left\{
\begin{array}{ll}
- i \partial _ { t } \Psi _ { 1 } = \triangle \Psi _ { 1 } - P ( x ) \Psi _ { 1 } + \mu _ { 1 } | \Psi _ { 1 } | ^ { 2 } \Psi _ { 1 } + \beta | \Psi _ { 2 } | ^ { 2 } \Psi _ { 1 },\quad \text{in } \mathbb{R}^N,\\
- i \partial _ { t } \Psi _ { 2 } = \triangle \Psi _ { 2 } - Q ( x ) \Psi _ { 2 } + \mu _ { 2 } | \Psi _ { 2 } | ^ { 2 } \Psi _ { 2 } + \beta | \Psi _ { 1 } | ^ { 2 } \Psi _ { 2 },\quad \text{in } \mathbb{R}^N
\end{array}
\right.
\end{equation}
with $\int_{\mathbb{R}^N}(\Psi _ { 1 }^2+\Psi _ { 2}^2)=\rho^2$, where $N=2,3$, $\mu_1,\mu_2>0$ and $\beta\in\mathbb{R}$ describe the interaction among particles in the same component and between particles of different components, $\rho>0$ is a mass constraint. Via the standing wave ansatz $\Psi_1(t,x)=e^{i\lambda t}\bar{u}(x)$ and $\Psi_2(t,x)=e^{i\lambda t}\bar{v}(x)$,  $(\bar{u},\bar{v})$ solves the following two-component coupled nonlinear Schr\"odinger system:
\begin{equation} \label{eq0}
\left\{
\begin{array}{ll}
- \triangle \bar{u}+(P(x) + \lambda )\bar{u}=\mu_1 \bar{u}^3+\beta \bar{u}\bar{v}^2,\quad \text{in } \mathbb{R}^N,\\
- \triangle \bar{v}+(Q(x) + \lambda )\bar{v}=\mu_2 \bar{v}^3+\beta \bar{v}\bar{u}^2,\quad \text{in } \mathbb{R}^N
\end{array}
\right.
\end{equation}
with the global $L^2$ mass constraint
\begin{equation}\label{rho2}
\int_{\mathbb{R}^N} (\bar{u}^2+\bar{v}^2)=\rho^2.
\end{equation}

Normalized solutions to nonlinear Schr\"odinger (NLS) equations have been a central and highly active research topic in nonlinear analysis and mathematical physics. 
The variational method is a widely used approach to normalized solutions, which has been extensively applied to both single-component NLS equations and multi-component CNLS systems. For single-component NLS equations, via constrained variational methods and the mountain-pass lemma, a series of works have established a wealth of existence results for normalized solutions including ground states, multi-bump solutions and excited states, in settings of general potentials, various nonlinearities and both mass subcritical and supercritical cases \cite{soave2020normalized,ding2022normalized,zhang2022normalized,yang2022normalized,liu2024normalized}. For multi-component CNLS systems, the related study has also attracted wide attention in recent years, and numerous works based on this method have further obtained existence, non-existence and multiplicity results for normalized solutions, covering CNLS systems, mixed systems and Sobolev critical cases \cite{bartsch2016normalized,luo2022existence,li2026normalized}.

On the other hand, many mathematicians applied the Lyapunov-Schmidt reduction method to study the existence of solutions for the elliptic problems, see \cite{CMP, guo2023excited, guo2024normalized, guo2025segregated, huang2025normalized, pellacci2021normalized, PPVV, PV} and reference therein.  In \cite{pellacci2021normalized}, Pellacci et al. investigated the normalized NLS equation in $\mathbb{R}^N$:
\[
\begin{cases}
-\Delta v + (\lambda + V(x))v = v^p\quad \text{in } \mathbb{R}^N, \\
v>0, \quad \int_{\Omega} |v|^2 dx = \rho
\end{cases}
\]
and established the existence of concentrating solutions as $\rho$ is either small (when $p<1+\frac{4}{N}$) or large (when $p>1+\frac{4}{N}$) or it approaches some critical threshold (when $p=1+\frac{4}{N}$). In \cite{guo2024normalized}, Guo et al. investigated the NLS equation with prescribed mass constraint in $\mathbb{R}^N$:
\[
\begin{cases}
-\Delta u + (\lambda + V(x))u = u^{p_\varepsilon-1}, & \text{in } \mathbb{R}^N, \\
u>0, \quad \int_{\mathbb{R}^N} |u|^2 dx = a
\end{cases}
\]
with the exponent $p_\varepsilon$ near the $L^2$-critical value, and proved the existence and local uniqueness of multi-peak concentrating normalized solutions. For multi-component CNLS system
\begin{equation}\label{eq02}
\begin{cases}
-\Delta v_i + \lambda v_i + V_i(x) v_i = \sum_{j=1}^k \beta_{ij} v_i v_j^2 \quad \text{in } \mathbb{R}^N, \, i=1,\dots,k, \\
\int_{\mathbb{R}^N} \left(v_1^2 + \cdots + v_k^2\right) dx = \rho,
\end{cases}
\end{equation}
there is a rich literature on its normalized solutions. In the special case $k=2$, $N=2$ and $\rho=1$, Guo and Yang \cite{guo2023excited} proved the existence of normalized multi-peak solutions under $C^2$ potentials with common non-degenerate critical points . Further existence and local uniqueness results for this special case can be found in  \cite{guo2021existence,guo2025segregated}. For $N=1,2,3$, the existence of normalized single-peak solutions  was proved in \cite{huang2025normalized} when $\rho$ is regarded as the parameter, while the existence and local uniqueness of normalized multi-peak solutions was not discussed therein.

In this paper, we mainly investigate existence and local uniqueness of normalized multi-peak solutions for system \eqref{eq02} in $N=1, 2,3$.  In the following, we focus on the case $k=2$, since the proof for the general multi-component case is analogous to that for the two-component case. Set
\begin{equation*}
\varepsilon:=\lambda^{-\frac12},\ \ u:=\varepsilon \bar{u},\ \ v:=\varepsilon \bar{v},
\end{equation*}
then \eqref{eq0}-\eqref{rho2} is equivalent to the singularly perturbed system:
\begin{equation} \label{eq1}
\left\{
\begin{array}{ll}
-\varepsilon ^2 \triangle u+\big( \varepsilon ^2P(x) +1 \big) u=\mu_1 u^3+\beta uv^2,\quad \text{in } \mathbb{R}^N,\vspace{0.12cm}\\
-\varepsilon ^2 \triangle v+\big( \varepsilon ^2Q(x) +1\big) v=\mu_2 v^3+\beta vu^2,\quad \text{in } \mathbb{R}^N
\end{array}
\right.
\end{equation}
with the constraint:
\begin{equation*}
\varepsilon^{-2}\int_{\mathbb{R}^N} (u^2+v^2)=\rho^2.
\end{equation*}

We first recall the following key preliminary results and definitions. From \cite{Uniqueness}, we know that $w$ is the unique solution to
\begin{equation}\label{eqw}
\begin{cases}
 - \Delta w + w = w ^ { 3 } ,\ \ w > 0 \ \ \ \text{in}\ \ \mathbb{R} ^ { N } , \\
 w ( 0 ) = \max \limits_ { x \in \mathbb{R}  ^ { N } } w ( x ) ,\ \ w ( x ) \in H ^ { 1 } ( \mathbb{R} ^ { N } ) .
\end{cases}
\end{equation}
Moreover, $w(x)=w(|x|)$ satisfies that $w^{\prime}(r)<0$ and for any $r\in(0,+\infty)$,
\begin{align*}
 \lim _ { r \to \infty } w \left( r \right) e ^ { r } r^{\frac{N-1}{2}} = C > 0 ,\ \ \ \lim _ { r \rightarrow \infty } \frac { w ( r ) } { w ^ { \prime } ( r ) } = - 1 .
\end{align*}

Now, we consider the limiting problem with respect to \eqref{eq1}:
\begin{equation}\label{eq3}
\begin{cases}
-\Delta u + u = \mu_1  u^3 + \beta u v^2, & \text{in } \mathbb{R}^N, \\
-\Delta v + v = \mu_2  v^3 + \beta v u^2, & \text{in } \mathbb{R}^N, \\
u(y), v(y) > 0, & \text{in } \mathbb{R}^N, \\
u(0) = \max_{x \in \mathbb{R}^N} u(x), & v(0) = \max_{x \in \mathbb{R}^N} v(x).
\end{cases}
\end{equation}
It is easy to check that
\begin{align*}
\left(w^{*}(x),w^{\star}(x)\right):= \left( \sqrt { \frac { \beta - \mu_2 } { \beta ^ { 2 } - \mu_1 \mu_2 } } w(x) , \sqrt { \frac { \beta - \mu_1 } { \beta ^ { 2 } - \mu_1 \mu_2 } } w(x) \right) = \left(\sigma_{1} w(x), \sigma_{2} w(x)\right)
\end{align*}
solves \eqref{eq3}, provided $\beta \in \mathcal{M}:= (-\sqrt{\mu_1\mu_2}, \min\{\mu_1, \mu_2\}) \cup (\max\{\mu_1, \mu_2\}, +\infty)$.

To state our results, we here impose some assumptions on the potentials $P(x)$ and $Q(x)$:
\begin{itemize}
\item[$(H_{1})$]   $P(x)$ and $Q(x)$ belong to $C^{2}(\mathbb{R}^{N})$;
\item[$(H_{2})$] $P(x)$ and $Q(x)$ have $k$ distinct common non-degenerate critical points $\xi_l$ with $P(\xi_{l})=Q(\xi_{l})$, and for each $l=1,\cdots,k$, $\det\left( (\beta -\mu_2)\frac{\partial^2P(\xi_l)}{\partial y_n\partial y_m}+(\beta -\mu_1)\frac{\partial^2Q(\xi_l)}{\partial y_n\partial y_m} \right)_{1\leq n,m\leq N}\ne 0$.
\end{itemize}

Here we assume that $P(\xi_{l})=Q(\xi_{l})=0$ and in $B_{\tau}(\xi_{l})$, $P(x)$ and $Q(x)$ satisfy
\begin{equation*}
\begin{split}
&P(x)=  \sum_{i=1}^{N}p_{li}(x_{i}-\xi_{l, i})^{2}+ O(|x-\xi_{l}|^{3}), \\
&  Q(x) = \sum_{i=1}^{N} q_{li}(x_{i}-\xi_{l, i})^{2}+ O(|x-\xi_{l}|^{3}).
\end{split}
\end{equation*}

 \begin{itemize}
\item[$(H_{3})$] When $N=2$, $P(x)$ and $Q(x)$ satisfy $\sum_{l=1}^k \sum_{i=1}^{2} (\sigma_1^2 p_{li} + \sigma_2^2 q_{li}) \neq 0$.
\end{itemize}

\begin{remark}
We would like to point out that the assumption $(H_{2})$ coincides with
the assumption that the global potential is non-degenerate when $k=2$ in \cite{huang2025normalized}.
Also in Remark 1.2 in \cite{huang2025normalized}, Huang et al. mentioned that in the case of the single equation, once they fix the non-degenerate critical point $\xi_{0}$ of $V,$  they can assume (without loss of generality) that $V (\xi_{0}) = 0$ up to replacing $\lambda$ with the new parameter $\lambda-V (\xi_{0}).$ This no longer holds in the case of the system, because if the single parameter $\lambda$ is replaced
by the parameters $\lambda_{i}=\lambda -V_{i}(\xi_{0}),$ they are different unless all the potentials have the same value at the point $\xi_{0}$.
Hence, similarly we also assume that $P(\xi_{l})=Q(\xi_{l})=0.$
\end{remark}

 \begin{remark} The assumption $(H_2)$ indicates the requirement for the relation of both potentials at each critical point, while the assumption $(H_3)$ reveals the influence of the global interaction of all critical points, which comes from the global constraint.  It is quite different from the Schr\"odinger systems without the mass constraint\cite{pomponio2006coupled, li2024infinitely, LT, wang2025infinitely} and the issue about the single-peak solution \cite{huang2025normalized}.
\end{remark}

To prove the existence of normalized solutions for \eqref{eq0}, we need to get the solutions of \eqref{eq1} with the more specific form.  Let $ \big(W_l^{*}, W_l^{\star}\big)$ be the unique solution of the following linear system
\begin{equation}\label{eq2}
\begin{cases}
-\Delta u + u - 3\mu_1  (w^{*})^2 u - \beta (w^{\star})^2 u - 2\beta w^{*} w^{\star} v = - \sum_{i=1}^{N}p_{li}x^{2}_{i}w^{*}, \\
-\Delta v + v - 3\mu_2  (w^{\star})^2 v - \beta (w^{*})^2 v - 2\beta w^{*} w^{\star} u = - \sum_{i=1}^{N} q_{li}x^{2}_{i}w^{\star}, \\
u, v \in H^1(\mathbb{R}^N), \quad l = 1, \cdots, k.
\end{cases}
\end{equation}

With the scaling transformation and potential assumptions in place, we next introduce a key definition that underpins our subsequent Lyapunov-Schmidt reduction and uniqueness analysis:

Define $H:= \Big\{(u,v) : (u,v)\in H^1(\mathbb{R}^N)\times H^1(\mathbb{R}^N)\Big\}$, and endow it with the following norm:
\[\|(u,v)\|^2_H = ((u,v),(u,v))_{H} = \|u\|^2_{\varepsilon,P} + \|v\|^2_{\varepsilon,Q},\]
where
\[\|u\|^2_{\varepsilon,P} = (u,u)_{\varepsilon,P} := \int_{\mathbb{R}^N} (\varepsilon^2|\nabla u|^2 + (\varepsilon^2P(x)+1)u^2)\]
and
\[\|v\|^2_{\varepsilon,Q} = (v,v)_{\varepsilon,Q} := \int_{\mathbb{R}^N} (\varepsilon^2|\nabla v|^2 + (\varepsilon^2Q(x)+1)v^2).\]

Set
$$K_{\varepsilon} =span \left\{ \Big( \frac{\partial w_{\varepsilon, \xi_{\varepsilon,l}}^{*} \left( x \right)}{\partial x_j} , \frac{\partial w_{\varepsilon,  \xi_{\varepsilon,l}}^{\star} \left( x \right)}{\partial x_j} \Big), l=1, \cdots, k, j=1,\cdots,N \right\} $$
and
$$ E _ { \varepsilon } := K^{\bot}_{\varepsilon} = \left\{  ( \varphi , \psi ) \in H : \left( ( \varphi , \psi ) , \Big( \frac { \partial w _ { \varepsilon, \xi_{\varepsilon,l}} ^ { * } ( x ) } { \partial x _ { j } } , \frac { \partial w _ { \varepsilon , \xi_{\varepsilon,l} } ^ { \star } ( x ) } { \partial x _ { j } } \Big) \right) _ { H } = 0 , l = 1 , \cdots , k , j = 1 , \cdots , N \right\} ,$$
where
\[
w_{\varepsilon, \xi_{\varepsilon, l}}^{*}(x) := w^{*}\left(\frac{x - \xi_{\varepsilon, l}}{\varepsilon}\right), \quad w_{\varepsilon, \xi_{\varepsilon, l}}^{\star}(x) := w^{\star}\left(\frac{x - \xi_{\varepsilon, l}}{\varepsilon}\right).
\]

Since in \cite{huang2025normalized} Huang et al. only constructed the single-peak solutions and Zheng study the multi-peak solutions without mass constraint in \cite{1},
we revisit the existence of normalized multi-peak solutions to system \eqref{eq0}--\eqref{rho2}.
\begin{theorem} \label{pro11}
Suppose \( (H_1) \) and \( (H_2) \) hold.  For \( \mu_1, \mu_2 > 0 \), there exists a decreasing sequence \( \{\beta_n\}_{n=1}^{+\infty} \) (depending only on \( \mu_1 \) and \( \mu_2 \)) as well as constants \( \tilde{\rho}_0, \tilde{\delta}_0  > 0 \). If \( \beta \in \mathcal{M} \setminus \left( \{0\} \cup \{\beta_n\}_{n=1}^{+\infty} \right) \), then
\begin{itemize}
\item[(1)] \it{ Case \(N=3\)} \ \  For any \(
\rho \in (0,\tilde{\rho}_0]\);
\item [(2)]  \it{ Case \(N=2\)} \ \  If $\sum_{l=1}^k  \sum_{i=1}^{2} (\sigma_1^2 p_{li}+ \sigma_2^2 q_{li})>0$, for $\rho^{2} \in (\rho_{0}^{2}-\tilde\delta, \rho_{0}^{2})$;
If $\sum_{l=1}^k  \sum_{i=1}^{2} (\sigma_1^2 p_{li}+ \sigma_2^2 q_{li})< 0$, for $\rho^{2} \in (\rho_{0}^{2}, \rho_{0}^{2}+\tilde\delta),$ \end{itemize}
 the system \eqref{eq0}--\eqref{rho2} admits a solution of the form
\begin{align}\label{uv-}
\bigl(\bar{u}_{\lambda_{\rho}},\bar{v}_{\lambda_{\rho}}\bigr)
&=\sqrt{\lambda_{\rho}}\Bigg(
\sum_{l=1}^{k}\Bigl(w_{\lambda_{\rho},\xi_{\rho,l}}^{*}
      +\lambda_{\rho}^{-2}W_{\lambda_{\rho},\xi_{\rho,l}}^{*}\Bigr)+\varphi_{\rho},
\sum_{l=1}^{k}\Bigl(w_{\lambda_{\rho},\xi_{\rho,l}}^{\star}
      +\lambda_{\rho}^{-2}W_{\lambda_{\rho},\xi_{\rho,l}}^{\star}\Bigr)+\psi_{\rho}\Bigg),
\end{align}
where
\[
\rho_0^2 = k\int_{\mathbb{R}^2}\left( (w^*)^2 + (w^\star)^2 \right) \mathrm{d}x.
\]

Here, for each \(l=1,\dots,k\), \( \xi_{\rho,l} \in B_{\delta}(\xi_l) \)~for some sufficiently small \( \delta > 0 \), and \( (\varphi_\rho, \psi_\rho) \in E_{\lambda_\rho} \). The rescaled profiles appearing in the above expression are defined as
\[
\begin{array}{ll}
w_{\lambda_{\rho},\xi_{\rho,l}}^{*}(x):=w^{*}\bigl(\sqrt{\lambda_{\rho}}(x-\xi_{\rho,l})\bigr), & W_{\lambda_{\rho},\xi_{\rho,l}}^{*}(x):=W_{l}^{*}\bigl(\sqrt{\lambda_{\rho}}(x-\xi_{\rho,l})\bigr),\\[2mm]
w_{\lambda_{\rho},\xi_{\rho,l}}^{\star}(x):=w^{\star}\bigl(\sqrt{\lambda_{\rho}}(x-\xi_{\rho,l})\bigr), & W_{\lambda_{\rho},\xi_{\rho,l}}^{\star}(x):=W_{l}^{\star}\bigl(\sqrt{\lambda_{\rho}}(x-\xi_{\rho,l})\bigr).
\end{array}
\]
\end{theorem}

Based on the existence result, we present the local uniqueness of the solutions.

\begin{theorem}\label{main1}
Assume  \( (H_1) \)-\( (H_3) \) hold. Suppose
\(
(\lambda_{1,\rho},\bar{u}_{\lambda_{1,\rho}}^{(1)},\bar{v}_{\lambda_{1,\rho}}^{(1)})\) and \( (\lambda_{2,\rho},\bar{u}_{\lambda_{2,\rho}}^{(2)},\bar{v}_{\lambda_{2,\rho}}^{(2)})
\)
are two pairs of solutions of \eqref{eq0}--\eqref{rho2}, which satisfy \eqref{uv-}.
Then, these two solutions must coincide. More precisely:
    \[
    (\bar{u}_{\lambda_{1,\rho}}^{(1)},\bar{v}_{\lambda_{1,\rho}}^{(1)})=(\bar{u}_{\lambda_{2,\rho}}^{(2)},\bar{v}_{\lambda_{2,\rho}}^{(2)}).
    \]
    \end{theorem}

\begin{remark}
Since both $N=1$ and $N=3$ are not the mass-critical cases, the existence and local uniqueness of solutions for \eqref{eq0}--\eqref{rho2} for $N=1$ as $\rho\rightarrow \infty$ can be obtained via an analogous proof  to that for the $N=3$ case. Thus, we only consider the cases $N=2,3$.
\end{remark}

\begin{remark}
Just by the similar argument as that of Theorem \ref{main1} and make some minor modifications, we can prove the normalized single-peak solutions for the system with $k(k\geq 2)$ equations obtained by Huang et al. in \cite{huang2025normalized} are also local unique.
\end{remark}

To prove local uniqueness of solutions, we primarily apply some local Pohozaev identities combined with a contradiction argument (see \cite{DLY,LY,GPY,GILY}).
There exist two main technical difficulties in our analysis. First, in the unconstrained case, each concentration point can be analyzed independently with no need for global conditions such as $(H_3)$. However, under the global mass constraint, the interactions between concentration points cannot be neglected. Second, our solution expansion demands sharp order estimates---even a slight error in these estimates will invalidate the proof of local uniqueness. More precisely, for mass-critical case, we need a sharp characterization of the parameter $\lambda_\rho$ to prove the local uniqueness, and to this end, the non-vanishing condition $(H_3)$ is required.

This paper is organized as follows. Section 2 is devoted to establishing continuity of multi-peak solutions for system \eqref{eq1} without the mass constraint by means of local Pohozaev idnentities. We will give the detailed proofs of Theorems \ref{pro11} and \ref{main1}  in Section~3 and Section~4, respectively.

\section{Continuity of multi-peak solutions without constraint}

In this section, we establish the continuity of solutions for \eqref{eq0} in the unconstrained case with respect to $\lambda$, which brings from the uniqueness of these solutions. Thus, we aim to prove the uniqueness of solutions for \eqref{eq0} in the following. At first, let us introduce the existence result:

\begin{theorem}[See \cite{1}] \label{pro1}
Suppose that the functions \( P(x) \) and \( Q(x) \) satisfy conditions \( (H_1) \) and \( (H_2) \). For \( \mu_1, \mu_2 > 0 \), there exists a decreasing sequence \( \{\beta_n\}_{n=1}^{+\infty} \) (depending only on \( \mu_1 \) and \( \mu_2 \)) as well as a constant \( \varepsilon_0 > 0 \). If
\[
\beta \in \bigl( -\sqrt{\mu_1\mu_2}, 0 \bigr) \cup \bigl( 0, \min\{\mu_1, \mu_2\} \bigr) \cup \bigl( \max\{\mu_1, \mu_2\}, +\infty \bigr)
\]
and \( \beta \notin \{\beta_n\}_{n=1}^{+\infty} \), then for all $\varepsilon\in (0,\varepsilon_0]$, there exist $\xi_{\varepsilon,l}\in B_{\delta}(\xi_l)$, $l=1,\cdots,k$ and $(\varphi_{\varepsilon},\psi_{\varepsilon})\in E_{\varepsilon}$ so that \eqref{eq1} has a solution of the following form:
\begin{equation}\label{uv}
(u_{\varepsilon},v_{\varepsilon}) = \left( \sum_{l=1}^k \left( w_{\varepsilon, \xi_{\varepsilon,l}}^{*} + \varepsilon^4 W_{\varepsilon, \xi_{\varepsilon,l}}^{*} \right) + \varphi_{\varepsilon}, \sum_{l=1}^k \left( w_{\varepsilon, \xi_{\varepsilon,l}}^{\star} + \varepsilon^4 W_{\varepsilon, \xi_{\varepsilon,l}}^{\star} \right) + \psi_{\varepsilon} \right).
\end{equation}
Moreover, $(\varphi_{\varepsilon},\psi_{\varepsilon})$ satisfies $\|(\varphi_{\varepsilon},\psi_{\varepsilon})\|_{H} = O(\varepsilon^{5+\frac{N}{2}})$.
\end{theorem}

To obtain the existence of normalized solutions with constraints, we then proceed to establish the uniqueness results stated below.

\begin{theorem}\label{main}
Under the conditions of Theorem \ref{pro1}, suppose that \( (u^{(1)}_{\varepsilon},v^{(1)}_{\varepsilon}) \) and \( (u^{(2)}_{\varepsilon},v^{(2)}_{\varepsilon}) \) are two different solutions of \eqref{eq1}, as shown below:
\begin{equation*}
(u^{(i)}_{\varepsilon},v^{(i)}_{\varepsilon}) = \left( \sum_{l=1}^k \left( w_{\varepsilon, \xi^{(i)}_{\varepsilon,l}}^{*} + \varepsilon^4 W_{\varepsilon, \xi^{(i)}_{\varepsilon,l}}^{*} \right) + \varphi_{\varepsilon}^{(i)}, \sum_{l=1}^k \left( w_{\varepsilon, \xi^{(i)}_{\varepsilon,l}}^{\star} + \varepsilon^4 W_{\varepsilon, \xi^{(i)}_{\varepsilon,l}}^{\star} \right) + \psi_{\varepsilon}^{(i)} \right), \ \ \ i=1,2,
\end{equation*}
satisfying
\begin{equation*}
\xi^{(i)}_{\varepsilon,l} \rightarrow \xi_l, \ \ \ \  l = 1 , \cdots , k ,
\end{equation*}
as \( \varepsilon \rightarrow 0 \). If \( \xi_l \) is a common isolated non-degenerate critical point of \( P(x) \) and \( Q(x) \) for \( l=1,\cdots,k \), then there exists \( \varepsilon_0 > 0 \) such that for any \( \varepsilon \in (0,\varepsilon_0 ] \), the two solution pairs \( (u^{(1)}_{\varepsilon},v^{(1)}_{\varepsilon}) \) and \( (u^{(2)}_{\varepsilon},v^{(2)}_{\varepsilon}) \)  to system \eqref{eq1} coincide. This coincidence, in turn, implies the uniqueness of the solution to the system.
\end{theorem}

\subsection{Some necessary estimates}

We first establish the local Pohozaev identity in the following lemma.

\begin{lemma}
Let $(u_{\varepsilon}, v_{\varepsilon})$ be a nontrivial solution of \eqref{eq1}. Then for any bounded domain $\Omega \subset \mathbb{R}^N$ and for $ j = 1 , \cdots , N,$ we have,
\begin{align}\label{p4}
&\varepsilon^{2}\int _ { \Omega } \Bigg(\frac { \partial P(x) } { \partial x _ { j } } u _ {\varepsilon } ^ { 2 }+\frac { \partial Q(x) } { \partial x _ { j } } v _ {\varepsilon } ^ { 2 }\Bigg) \notag\\
=&-\beta\int _ { \partial \Omega } v _ { \varepsilon } ^ { 2 } u _ { \varepsilon } ^ { 2 } n _ { j } d S
- 2 \varepsilon ^ { 2 } \int _ { \partial \Omega } \Bigg(\frac { \partial u _ { \varepsilon } } { \partial x _ { j } } \frac { \partial u _ { \varepsilon } } { \partial n }+\frac { \partial v _ { \varepsilon } } { \partial x _ { j } } \frac { \partial v _ { \varepsilon } } { \partial n }\Bigg) d S + \varepsilon ^ { 2 } \int _ { \partial \Omega } (| \nabla u _ { \varepsilon } | ^ { 2 } n _ { j }+| \nabla v _ { \varepsilon } | ^ { 2 } n _ { j } )d S \notag\\
&+ \int _ { \partial \Omega }\left[(\varepsilon^{2}P(x)+1) u _ { \varepsilon } ^ { 2 } n _ { j }+(\varepsilon^{2}Q(x)+1)v _ { \varepsilon } ^ { 2 } n _ { j }\right] d S - \frac { 1 } { 2 } \Bigg(\mu_1\int _ { \partial \Omega } u _ { \varepsilon } ^ { 4 } n _ { j } d S +\mu_2\int _ { \partial \Omega } v _ { \varepsilon } ^ { 4 } n _ { j } d S\Bigg ),
\end{align}
where \( n_j \) (\( j=1,\dots,N \)) denotes the \( j \)-th component of the outward unit normal vector to \( \partial\Omega \).
\end{lemma}

\begin{proof}
Multiplying the first equation of \eqref{eq1} by $ \frac { \partial u _ { \varepsilon } } { \partial x _ { j } }$ and integrating over $\Omega$, we have
\begin{align}\label{p1}
-\varepsilon^{2}\int_{\Omega}\frac{\partial u_{\varepsilon}}{\partial x_{j}}\Delta u_{\varepsilon}+\int_{\Omega}(\varepsilon^{2}P(x)+1)\frac{\partial u_{\varepsilon}}{\partial x_{j }}u_{\varepsilon}
&=\mu_1 \int_{\Omega}\frac{\partial u_{\varepsilon}}{\partial x_{j}}u_{\varepsilon}^{3}+\beta\int_{\Omega}\frac{\partial u_{\varepsilon}}{\partial x _{j}}u_{\varepsilon}v_{\varepsilon}^{2}\notag\\
&=\frac{\mu_1 }{4}\int_{\partial \Omega} u_{\varepsilon}^{4}n_{j}dS + \frac{\beta}{2}\int_{\Omega}\frac{\partial u_{\varepsilon}^{2}}{\partial x_{j}}v_{\varepsilon}^{2}.
\end{align}
By Green's formulas and integration by parts formula, we get
\begin{align}\label{3111}
-\int_{\Omega}\frac{\partial u_{\varepsilon}}{\partial x_{j}}\Delta u_{\varepsilon}&=-\int_{\partial\Omega}\frac{\partial u_{\varepsilon}}{\partial x_{j}}\frac{\partial u_{\varepsilon}}{\partial n}dS+\int_{\Omega}\nabla \left(\frac{\partial u_{\varepsilon}}{\partial x_j}\right)\nabla u_{\varepsilon}\notag\\
&=-\int_{\partial\Omega}\frac{\partial u_{\varepsilon}}{\partial x_{j}}\frac{\partial u_{\varepsilon}}{\partial n}dS+\frac{1}{2}\int_{\partial\Omega}|\nabla u_{\varepsilon}|^{2}n_{j}dS
\end{align}
and
\begin{align}\label{3112}
\int _ { \Omega } (\varepsilon^{2}P(x)+1) \frac { \partial u _ {\varepsilon } } { \partial x _ { j } } u _ {\varepsilon } = \frac { 1 } { 2 } \int _ { \partial \Omega }(\varepsilon^{2}P(x)+1)u _ {\varepsilon } ^ { 2 } n _ { j } d S - \frac { \varepsilon^{2} } { 2 } \int _ { \Omega } \frac {\partial P(x)} { \partial x _ { j } } u _ {\varepsilon } ^ { 2 } .
\end{align}
According to \eqref{p1}, \eqref{3111} and \eqref{3112}, we derive that
\begin{align}\label{p2}
\varepsilon^{2}\int _ { \Omega } \frac { \partial P(x) } { \partial x _ { j } } u _ {\varepsilon } ^ { 2 } + \beta \int _ { \Omega } \frac { \partial u _ { \varepsilon } ^ { 2 } } { \partial x _ { j } } v _ {\varepsilon } ^ { 2 }
=& - 2 \varepsilon ^ { 2 } \int _ { \partial \Omega } \frac { \partial u _ { \varepsilon } } { \partial x _ { j } } \frac { \partial u _ { \varepsilon } } { \partial n } d S + \varepsilon ^ { 2 } \int _ { \partial \Omega } | \nabla u _ { \varepsilon } | ^ { 2 } n _ { j } d S \notag\\
&+ \int _ { \partial \Omega } (\varepsilon^{2}P(x)+1) u _ { \varepsilon } ^ { 2 } n _ { j } d S - \frac { \mu_1 } { 2 } \int _ { \partial \Omega } u _ { \varepsilon } ^ { 4 } n _ { j } d S .
\end{align}
Similarly, by using the second equation of \eqref{eq1}, we obtain
\begin{align}\label{p3}
\varepsilon^{2}\int _ { \Omega } \frac { \partial Q(x) } { \partial x _ { j } } v _ {\varepsilon } ^ { 2 } + \beta \int _ { \Omega } \frac { \partial v _ { \varepsilon } ^ { 2 } } { \partial x _ { j } } u _ {\varepsilon } ^ { 2 }
=& - 2 \varepsilon ^ { 2 } \int _ { \partial \Omega } \frac { \partial v _ { \varepsilon } } { \partial x _ { j } } \frac { \partial v _ { \varepsilon } } { \partial n } d S + \varepsilon ^ { 2 } \int _ { \partial \Omega } | \nabla v _ { \varepsilon } | ^ { 2 } n _ { j } d S \notag\\
&+ \int _ { \partial \Omega } (\varepsilon^{2}Q(x)+1) v _ { \varepsilon } ^ { 2 } n _ { j } d S - \frac { \mu_2 } { 2 } \int _ { \partial \Omega } v _ { \varepsilon } ^ { 4 } n _ { j } d S .
\end{align}
Note that
\begin{align}\label{p5}
\int _ { \Omega } \frac { \partial v _ { \varepsilon } ^ { 2 } } { \partial x _ { j } } u _ { \varepsilon } ^ { 2 }  = - \int _ { \Omega } \frac { \partial u _ { \varepsilon } ^ { 2 } } { \partial x _ { j } } v _ { \varepsilon } ^ { 2 }  + \int _ { \partial \Omega } v _ { \varepsilon } ^ { 2 } u _ { \varepsilon } ^ { 2 } n _ { j } d S .
\end{align}
Combining \eqref{p2}, \eqref{p3} and \eqref{p5}, we obtain \eqref{p4}.
\end{proof}

For the sake of subsequent proofs, we first present the decay estimates for $W_{l}^{*}$ and $W_{l}^{\star}$. Notably, Zheng in \cite{1} have rigorously proven the decay estimates of $W_{l}^{*}$ and $W_{l}^{\star}$ in $\mathbb{R}^3$. By extending their technical approach, we can analogously derive the corresponding decay estimates of $W_{l}^{*}$ and $W_{l}^{\star}$ in $\mathbb{R}^N.$
\begin{lemma}[Lemma 3.1.3, \cite{1}] \label{le3.1.3}
	For $( W_{l}^{*}\left(y\right), W_{l}^{\star}\left(y\right) ) \in C^2( \mathbb{R}^N ) \times C^2(\mathbb{R}^N)$ with $l=1, \cdots, k$ and for a fixed sufficiently small $\tau >0$, there exists a constant $C>0$ such that
\begin{equation*}
\left| W_{l}^{*}\left(y\right) \right| \leq C e^{-\left(1-\tau\right) \left|y\right|}, \,\,\,\,\,\, \left| W_{l}^{\star}\left(y\right) \right| \leq C e^{-\left(1-\tau\right) \left|y\right|},\,\, \text{in}\,\,\,\, \mathbb{R}^N.
\end{equation*}
\end{lemma}

Next, we will make some estimates for the solution shown in \eqref{uv}.
\begin{lemma}\label{lem32}
Let $(u_{\varepsilon},v_{\varepsilon})$ be a solution constructed in Theorem \ref{pro1}. Then the following statements hold.
\begin{itemize}
\item[(i)] For any given small $\tau>0$, there exists a suffiiciently large constant $R > 0$, such that
\[
|u_{\varepsilon}(x)|\leq\tau,\quad |v_{\varepsilon}(x)|\leq\tau,\quad\forall x\in\mathbb{R}^N\setminus\cup_{l = 1}^k B_{R\varepsilon}(\xi_{\varepsilon,l}).
\]
\item[(ii)] There exists $C>0$ such that
\[
\|(u_{\varepsilon},v_{\varepsilon})\|^2_H \leq C\varepsilon^N.
\]
\end{itemize}
\end{lemma}

\begin{proof}
To show $(i)$, we just need to show that as $ \varepsilon \rightarrow 0$, $| | \varphi_\varepsilon | |_{ L ^ \infty ( \mathbb{R}^N)}\rightarrow 0$, $| | \psi_\varepsilon | | _ { L ^ \infty ( \mathbb{R}^N)} \rightarrow 0$. For convenience, we let $\tilde{\varphi}_\varepsilon (x)=\varphi_\varepsilon (\varepsilon x)$, $\tilde{\psi}_\varepsilon (x)=\psi_\varepsilon (\varepsilon x)$ and write
\begin{align*}
(U_\varepsilon(x),V_\varepsilon(x))= \left(\sum _ { l = 1 } ^ { k } \left(w^{*}_{\varepsilon,\xi_{\varepsilon,l}}+\varepsilon^4 W^{*}_{\varepsilon,\xi_{\varepsilon,l}}\right),\sum _ { l = 1 } ^ { k } \left(w^{\star}_{\varepsilon,\xi_{\varepsilon,l}}+\varepsilon^4 W^{\star}_{\varepsilon,\xi_{\varepsilon,l}}\right)\right),
\end{align*}
in addition,
\begin{align*}
\tilde{U}_{\varepsilon}(x)=U_\varepsilon(\varepsilon x),\quad \tilde{V}_{\varepsilon}(x)=V_\varepsilon(\varepsilon x).
\end{align*}
Then by \eqref{eq1},  we know that
\begin{align} \label{-phi}
-\triangle \tilde{\varphi}_\varepsilon= &
-(\varepsilon^2 P(\varepsilon x)+1)(\tilde{U}_{\varepsilon}+\tilde{\varphi}_\varepsilon)+ \mu_1 (\tilde{U}_{\varepsilon}+\tilde{\varphi}_\varepsilon)^3 \notag\\
&+\beta (\tilde{U}_{\varepsilon}+\tilde{\varphi}_\varepsilon) (\tilde{V}_{\varepsilon}+\tilde{\psi}_\varepsilon)^2 + \triangle \tilde{U}_{\varepsilon}\notag\\
=&-(\varepsilon^2 P(\varepsilon x)+1)\tilde{\varphi}_\varepsilon+\mu_1\big[3(\tilde{U}_{\varepsilon})^2 \tilde{\varphi}_\varepsilon+3\tilde{U}_{\varepsilon}\tilde{\varphi}_\varepsilon^2+\tilde{\varphi}_\varepsilon^3\big]\notag\\
&+\beta \big(2\tilde{U}_{\varepsilon}\tilde{V}_{\varepsilon}\tilde{\psi}_\varepsilon+\tilde{U}_{\varepsilon}\tilde{\psi}_\varepsilon^2
+\tilde{\varphi}_\varepsilon (\tilde{V}_{\varepsilon})^2+2 \tilde{V}_{\varepsilon}\tilde{\varphi}_\varepsilon\tilde{\psi}_\varepsilon+ \tilde{\varphi}_\varepsilon\tilde{\psi}_\varepsilon^2\big)+Z_\varepsilon =:f_\varepsilon,
\end{align}
where
\begin{align*}
Z_\varepsilon=\triangle \tilde{U}_{\varepsilon} -(\varepsilon^2 P(\varepsilon x)+1)\tilde{U}_{\varepsilon}+\mu_1 \tilde{U}_{\varepsilon}^3 +\beta \tilde{U}_{\varepsilon} \tilde{V}_{\varepsilon}^2.
\end{align*}
By \eqref{eq3}, \eqref{eq2}, and Lemma \ref{le3.1.3}, we readily establish that $Z_\varepsilon=o(1)$. Additionally, Theorem \ref{pro1} yields that $||\varphi_{\varepsilon}||_{\varepsilon,P} = O(\varepsilon^{5+\frac{N}{2}})$ and $||\psi_{\varepsilon}||_{\varepsilon,Q} = O(\varepsilon^{5+\frac{N}{2}})$. Then by the Hypothesis $(H_1)$, we have
\begin{align*}
||\varphi_{\varepsilon}||^2_{\varepsilon,P}&=\int _ { \mathbb{R} ^ { N } } \left(\varepsilon ^ { 2 } | \nabla \varphi_{\varepsilon} | ^ { 2 } + \left( \varepsilon ^ { 2 } P \left( x \right) + 1 \right) \varphi_{\varepsilon} ^ { 2 }\right)\\
&=\varepsilon^N \int _ { \mathbb{R} ^ { N } } \left( | \nabla \tilde{\varphi}_\varepsilon | ^ { 2 } + \left( \varepsilon ^ { 2 } P \left( \varepsilon x \right) + 1 \right) \tilde{\varphi}_\varepsilon ^ { 2 }\right)\\
&\geq C\varepsilon^N ||\tilde{\varphi}_\varepsilon||_{H^1(\mathbb{R}^N)}^2.
\end{align*}
Therefore $||\tilde{\varphi}_\varepsilon||_{H^1(\mathbb{R}^N)}=O(\varepsilon^5)$ and similarly, $||\tilde{\psi}_\varepsilon||_{H^1(\mathbb{R}^N)}=O(\varepsilon^5)$. By virtue of the Sobolev inequalities, we can establish that
\begin{align} \label{Lp1}
||\tilde{\varphi}_\varepsilon||_{L^p(\mathbb{R}^N)}\leq C||\tilde{\varphi}_\varepsilon||_{H^{1}(\mathbb{R}^N)}=O(\varepsilon^5)\ \  \text{for}\ \ 2\leq p\leq 2^*
\end{align}
and
\begin{align} \label{Lp2}
||\tilde{\psi}_\varepsilon||_{L^p(\mathbb{R}^N)}\leq C||\tilde{\psi}_\varepsilon||_{H^{1}(\mathbb{R}^N)}=O(\varepsilon^5)\ \  \text{for}\ \ 2\leq p\leq 2^*,
\end{align}
where
\[
2^* =
\begin{cases}
6, & \text{ when }~ N=3, \\
\infty, & \text{ when }~ N=2.
\end{cases}
\]
Using Condition $(H_1)$ and \eqref{Lp1}, for any fixed $B_1 \subset \mathbb{R}^N$, we have
\begin{align}\label{lem14}
\int_{B_1} (\varepsilon^2 P(\varepsilon x)+1)^2\tilde{\varphi}_\varepsilon^2 \leq C\int_{B_1} \tilde{\varphi}_\varepsilon^2=O(\varepsilon^{10}).
\end{align}
Moreover, by Lemma \ref{le3.1.3}, we get
\begin{align}\label{lem15}
&\mu_1^2 \int_{B_1}\big(3\tilde{U}_{\varepsilon}^2 \tilde{\varphi}_\varepsilon+3\tilde{U}_{\varepsilon}\tilde{\varphi}_\varepsilon^2 +\tilde{\varphi}_\varepsilon^3\big)^2\notag\\
=&\mu_1^2 \int_{B_1}\big (9\tilde{U}_{\varepsilon}^4 \tilde{\varphi}_\varepsilon^2+15\tilde{U}_{\varepsilon}^2\tilde{\varphi}_\varepsilon^4 +\tilde{\varphi}_\varepsilon^6+18\tilde{U}_{\varepsilon}^3 \tilde{\varphi}_\varepsilon^3+6\tilde{U}_{\varepsilon}\tilde{\varphi}_\varepsilon^5\big)= O(\varepsilon^{10}).
\end{align}
Similarly, we can obtain that
\begin{align}\label{lem16}
\beta^2 \int_{B_1} \left(2\tilde{U}_{\varepsilon}\tilde{V}_{\varepsilon}\tilde{\psi}_\varepsilon+\tilde{U}_{\varepsilon}\tilde{\psi}_\varepsilon^2+\tilde{\varphi}_\varepsilon \tilde{V}_{\varepsilon}^2+2\tilde{V}_{\varepsilon}\tilde{\varphi}_\varepsilon\tilde{\psi}_\varepsilon+ \tilde{\varphi}_\varepsilon\tilde{\psi}_\varepsilon^2\right)^2=O(\varepsilon^{10}).
\end{align}
By \eqref{-phi} and \eqref{lem14}-\eqref{lem16}, we can know that $||f_\varepsilon||_{L^2(B_1)}=o(1)$. Therefore, according to the Moser iteration (Theorem 4.1, \cite{han2011elliptic}), we have
\begin{align*}
\sup_{B_{\theta}} \tilde{\varphi}_\varepsilon\leq C\left(\frac{1}{(1-\theta)^{\frac32}}||\tilde{\varphi}_\varepsilon||_{L^2(B_1)}+||f_\varepsilon||_{L^2(B_1)}\right)\leq C||\tilde{\varphi}_\varepsilon||_{L^2(\mathbb{R}^N)}+o(1)=o(1),\ \ \ \forall \theta\in(0,1).
\end{align*}
By the same argument, we can check that $\sup_{B_{\theta}} \tilde{\psi}_\varepsilon=o(1)$.

Next, we will prove that $(ii)$ holds. Since
\begin{align} \label{310}
\|u_\varepsilon\|^2_{\varepsilon,P} &= \|U_\varepsilon +\varphi_\varepsilon \|^2_{\varepsilon,P} \notag\\
&= \int_{\mathbb{R}^N} \left(\varepsilon^2|\nabla (U_\varepsilon +\varphi_\varepsilon)|^2 + (\varepsilon^2P(x)+1)(U_\varepsilon +\varphi_\varepsilon)^2\right) \notag\\
&= \int_{\mathbb{R}^N} \left(\varepsilon^2|\nabla U_\varepsilon|^2 + (\varepsilon^2P(x)+1)U_\varepsilon^2\right)+\int_{\mathbb{R}^N} \left(\varepsilon^2|\nabla \varphi_\varepsilon|^2 + (\varepsilon^2P(x)+1)\varphi_\varepsilon^2\right) \notag \\
&\ \ \ \ +2\int_{\mathbb{R}^N} \left(\varepsilon^2 \nabla U_\varepsilon \nabla \varphi_\varepsilon + (\varepsilon^2P(x)+1)U_\varepsilon \varphi_\varepsilon\right)
\end{align}
and by Theorem \ref{pro1}, we have
\begin{align} \label{311}
\int_{\mathbb{R}^N} \left(\varepsilon^2|\nabla \varphi_\varepsilon|^2 + (\varepsilon^2P(x)+1)\varphi_\varepsilon^2\right) \leq C\varepsilon^{10+N}.
\end{align}
Meanwhile,
\begin{align} \label{312}
&\int_{\mathbb{R}^N} \left(\varepsilon^2|\nabla U_\varepsilon (x)|^2 + (\varepsilon^2P(x)+1)U_\varepsilon^2 (x)\right) dx \notag\\
= &\varepsilon^N\int_{\mathbb{R}^N} \left(|\nabla U_\varepsilon (\varepsilon z+\xi_{\varepsilon,l})|^2 + (\varepsilon^2P(\varepsilon z+\xi_{\varepsilon,l})+1)U_\varepsilon^2 (\varepsilon z+\xi_{\varepsilon,l})\right) dz \notag\\
\leq & C \varepsilon^N\int_{\mathbb{R}^N} \left(|\nabla U_\varepsilon (\varepsilon z+\xi_{\varepsilon,l})|^2 + U_\varepsilon^2 (\varepsilon z+\xi_{\varepsilon,l})\right) dz\leq  C \varepsilon^N.
\end{align}
From Young's inequality and \eqref{311}, \eqref{312}, we know that
\begin{align} \label{313}
&2\int_{\mathbb{R}^N} \left(\varepsilon^2 \nabla U_\varepsilon \nabla \varphi_\varepsilon + (\varepsilon^2P(x)+1)U_\varepsilon \varphi_\varepsilon\right)\notag\\
\leq & \int_{\mathbb{R}^N} \left(\varepsilon^2|\nabla U_\varepsilon|^2 + (\varepsilon^2P(x)+1)U_\varepsilon^2\right)+\int_{\mathbb{R}^N} \left(\varepsilon^2|\nabla \varphi_\varepsilon|^2 + (\varepsilon^2P(x)+1)\varphi_\varepsilon^2\right) \notag\\
\leq & C (\varepsilon^N + \varepsilon^{10+N})\leq C \varepsilon^N.
\end{align}
Combining \eqref{310}-\eqref{313}, we obtain that $\|u_\varepsilon\|^2_{\varepsilon,P} \leq C \varepsilon^N$. And similarly, $\|v_\varepsilon\|^2_{\varepsilon,P} \leq C \varepsilon^N$. Therefore, $(ii)$ holds.

\end{proof}

\begin{lemma}\label{lem33}
Suppose that $(u_{\varepsilon},v_{\varepsilon})$ be a solution constructed in Theorem \ref{pro1}.  For any $\alpha\in(0,1)$, there exists a constant $C>0$ such that
\begin{equation}\label{lem11}
u _ { \varepsilon } ( x ) \leq C \sum _ { l = 1 } ^ { k } e ^ { - \frac { \sqrt { \alpha } | x - \xi _ { \varepsilon,l } | } { \varepsilon } } ,\ \ \ v _ { \varepsilon } ( x ) \leq C \sum _ { l = 1 } ^ { k } e ^ { - \frac { \sqrt { \alpha } | x - \xi _ { \varepsilon,l } | } { \varepsilon } } ,\ \ \ \ \forall x \in \mathbb{R } ^ { N }.
\end{equation}
Moreover, we have
\begin{equation}\label{lem12}
 | \nabla u _ { \varepsilon } ( x ) | \leq C e ^ { - \frac { \sqrt { \alpha } } { 4 \varepsilon } } ,\ \ \ | \nabla v _ { \varepsilon } ( x ) | \leq C e ^ { - \frac { \sqrt { \alpha } } { 4 \varepsilon } } ,\ \ \ \ \forall x \in \partial B _ { \delta } ( \xi _ { \varepsilon , l } ) , l = 1 , \cdots , k .
 \end{equation}
\end{lemma}

\begin{proof}
Write
\begin{equation*}
  \left\{\begin{array}{ll}
		-\varepsilon ^2 \triangle u_{\varepsilon}+\left(\big( \varepsilon ^2P\left( x \right) +1 \big) -\mu_1  u_{\varepsilon}^2-\beta v_{\varepsilon}^2\right)u_{\varepsilon}=0,\,\,\,\,x\in \mathbb{R } ^ { N },\vspace{0.12cm}\\
		-\varepsilon ^2 \triangle v_{\varepsilon}+\left(\big( \varepsilon ^2Q\left( x \right) +1 \big) -\mu_2  v_{\varepsilon}^2-\beta u_{\varepsilon}^2\right)v_{\varepsilon}=0,\,\,\,\,x\in\mathbb{R } ^ { N }.\\
	\end{array} \right.
\end{equation*}
According to the condition $(H_1)$ and conclusion $(i)$ in Lemma \ref{lem32}, for any $\alpha\in(0,1)$, there exists $R>0$ such that
\begin{equation*}
\left( \varepsilon ^2 P\left( x \right) +1 \right) -\mu_1  u_{\varepsilon}^2-\beta v_{\varepsilon}^2\geq\alpha,\ \ \left( \varepsilon ^2Q\left( x \right) +1 \right) -\mu_2  v_{\varepsilon}^2-\beta u_{\varepsilon}^2\geq\alpha,
\end{equation*}
for $ x \in \mathbb{R}^{N} \backslash \cup _ { l = 1 } ^ { k } B _ { R\varepsilon } ( \xi _ { \varepsilon , l } ) $. Therefore, we have
\begin{equation}
  \left\{\begin{array}{ll}
		-\varepsilon ^2 \triangle u_{\varepsilon}+\alpha u_{\varepsilon}\leq0,\,\, \text{in} \,\, \mathbb{R } ^ { N } \backslash \cup _ { j = 1 } ^ { k } B _ { R\varepsilon } ( \xi _ { \varepsilon , l } ),\\
		-\varepsilon ^2 \triangle v_{\varepsilon}+\alpha v_{\varepsilon}\leq0,\,\, \text{in} \,\, \mathbb{R } ^ { N } \backslash \cup _ { j = 1 } ^ { k } B _ { R\varepsilon } ( \xi _ { \varepsilon , l } ).\\
	\end{array} \right.
\label{<=}
\end{equation}
Letting
\begin{equation*}
L_\varepsilon \vartheta =-\varepsilon ^2 \triangle \vartheta+\alpha \vartheta,
\end{equation*}
then it follows that
\begin{align}\label{L}
L _ { \varepsilon } e ^ { - \frac { \sqrt { \alpha } r } { \varepsilon } } &= - \varepsilon ^ { 2 } \left( \frac { \alpha } { \varepsilon ^ { 2 } } - \frac { ( N - 1 ) } { r } \frac { \sqrt { \alpha } } { \varepsilon } \right) e ^ { - \frac { \sqrt { \alpha } r } { \varepsilon } } + \alpha e ^ { - \frac { \sqrt { \alpha } r } { \varepsilon } } \notag \\
&=\varepsilon\frac{N-1}{r}\sqrt{\alpha}e ^ { - \frac { \sqrt { \alpha } r } { \varepsilon } }> 0 ,
\end{align}
where $r=|x-\xi_{\varepsilon,l}|$. Set
\begin{align*}
\tilde{u}_\varepsilon (x)= u_\varepsilon(\varepsilon x), \ \ \ \tilde{v}_\varepsilon (x)= v_\varepsilon(\varepsilon x).
 \end{align*}
Subsequently, based on the Conclusion $(ii)$ in Lemma \ref{lem32}, it can be deduced that
\begin{align*}
\int_{\mathbb{R}^N} (|\nabla \tilde{u}_\varepsilon|^2 + \tilde{u}_\varepsilon^2)\leq C,\ \ \ \int_{\mathbb{R}^N} (|\nabla \tilde{v}_\varepsilon|^2 + \tilde{v}_\varepsilon^2)\leq C.
 \end{align*}
Moreover, $\tilde{u}_\varepsilon$,  $\tilde{v}_\varepsilon$ satisfies
\begin{equation*}
\begin{cases}
-\triangle \tilde{u}_\varepsilon + \left[(\varepsilon^2 P(\varepsilon x)+1)-\mu_1 \tilde{u}_\varepsilon^2 -\beta  \tilde{v}_\varepsilon^2\right]\tilde{u}_\varepsilon=0,\,\,\,\,x\in \mathbb{R } ^ { N },\vspace{0.12cm}\\
-\triangle \tilde{v}_\varepsilon + \left[(\varepsilon^2 Q(\varepsilon x)+1)-\mu_2 \tilde{v}_\varepsilon^2 -\beta  \tilde{u}_\varepsilon^2\right]\tilde{v}_\varepsilon=0\,\,\,\,\,x\in \mathbb{R } ^ { N }.
\end{cases}
\end{equation*}
 Let
 \begin{align*}
\tilde{f}_\varepsilon:=(\varepsilon^2 P(\varepsilon x)+1)-\mu_1 \tilde{u}_\varepsilon^2 -\beta  \tilde{v}_\varepsilon^2,\\
\tilde{g}_\varepsilon:=(\varepsilon^2 Q(\varepsilon x)+1)-\mu_2 \tilde{v}_\varepsilon^2 -\beta  \tilde{u}_\varepsilon^2.
 \end{align*}
By Sobolev embedding theorem, we know that
\begin{align*}
||\tilde{u}_\varepsilon||_{L^p(\mathbb{R}^N)}\leq C||\tilde{u}_\varepsilon||_{H^1(\mathbb{R}^N)}\leq C, \ \ \ ||\tilde{v}_\varepsilon||_{L^p(\mathbb{R}^N)}\leq C||\tilde{v}_\varepsilon||_{H^1(\mathbb{R}^N)}\leq C, \ \ \text{for}\ \ 2\leq p\leq 2^*.
 \end{align*}
Therefore, through H\"older inequality, we have
 \begin{align*}
&\int_{B_1} \left((\varepsilon^2 P(\varepsilon x)+1)-\mu_1 \tilde{u}_\varepsilon^2 -\beta  \tilde{v}_\varepsilon^2\right)^2\\
\leq&C \int_{B_1} \left(\mu_1^2 \tilde{u}_\varepsilon^4+ |\beta|^2  \tilde{v}_\varepsilon^4+  \mu_1 \tilde{u}_\varepsilon^2 + |\beta|  \tilde{v}_\varepsilon^2 +\mu_1  |\beta| \tilde{u}_\varepsilon^2  \tilde{v}_\varepsilon^2 \right)+C\\
\leq&C \left(\int_{B_1} \left(\mu_1^2 \tilde{u}_\varepsilon^4+ |\beta|^2  \tilde{v}_\varepsilon^4+  \mu_1 \tilde{u}_\varepsilon^2 + |\beta|  \tilde{v}_\varepsilon^2\right)+\mu_1  |\beta|\left(\int_{B_1} \tilde{u}_\varepsilon^4\right)^{\frac12}\left(\int_{B_1} \tilde{v}_\varepsilon^4\right)^{\frac12}\right)+C\leq C,
 \end{align*}
where $B_1$ is an arbitrary unit ball in $\mathbb{R}^N$. This means that $||\tilde{f}_\varepsilon||_{L^2(B_1)}\leq C$ and similarly, $||\tilde{g}_\varepsilon||_{L^2(B_1)}\leq C$. Using Moser iteration argument, we get $|\tilde{u}_\varepsilon|\leq C$, $|\tilde{v}_\varepsilon|\leq C$ for some constant $C$ independent of $\varepsilon$. Therefore, we can see that $|u_\varepsilon|$, $|v_\varepsilon|\leq M$ for some $M>0$. Let
\begin{align*}
 p _ { \varepsilon } ( x ) = M e ^ { \sqrt { \alpha } R } \sum _ { j = 1 } ^ { k } e ^ { - \frac { \sqrt { \alpha } | x - \xi _ { \varepsilon, l } | } { \varepsilon } } - u _ { \varepsilon } ( x ) ,\ \ \ \ q _ { \varepsilon } ( x ) = M e ^ { \sqrt { \alpha } R } \sum _ { j = 1 } ^ { k } e ^ { - \frac { \sqrt { \alpha } | x - \xi _ { \varepsilon, l } | } { \varepsilon } } - v_ { \varepsilon } ( x ).
 \end{align*}
Combining with \eqref{<=} and \eqref{L}, we have
\begin{align*}
L_\varepsilon p_\varepsilon\geq 0,\ \ \ \ L_\varepsilon q_\varepsilon\geq 0,\ \ \ \ \text{in}\ \ \mathbb{R } ^ { N } \backslash \cup _ { l = 1 } ^ { k } B _ { R\varepsilon } ( \xi _ { \varepsilon , l} ) .
 \end{align*}
Particularly, on $ \partial B _ {R\varepsilon   } ( \xi _ { \varepsilon , l } )$ , we conclude that
\begin{align*}
p _ { \varepsilon } \geq M - \max_{x\in\mathbb{R}^N} u _ { \varepsilon } ( x ) \geq 0 ,\ \ \ \ q _ { \varepsilon } \geq M - \max_{x\in\mathbb{R}^N} v _ { \varepsilon } ( x ) \geq 0.
\end{align*}
Hence, by the comparison theorem, we get
\begin{align*}
p _ { \varepsilon } \geq 0,\ \ \ \ q _ { \varepsilon }\geq 0,\ \ \ \ \text{in}\ \ \mathbb{R } ^ { N } \backslash \cup _ { l = 1 } ^ { k } B _ { R\varepsilon} ( \xi _ { \varepsilon , l } ) ,
\end{align*}
while in $ B _ { R\varepsilon } ( \xi _ { \varepsilon , l } )$, we have the estimate
\begin{align*}
 M e ^ { \sqrt { \alpha } R } \sum _ { j = 1 } ^ { k } e ^ { - \frac { \sqrt { \alpha } | x - \xi _ { \varepsilon, l } | } { \varepsilon } } \geq M\geq u_\varepsilon(x),\ \ \ \ M e ^ { \sqrt { \alpha } R } \sum _ { j = 1 } ^ { k } e ^ { - \frac { \sqrt { \alpha } | x - \xi _ { \varepsilon, l } | } { \varepsilon } } \geq M\geq v_\varepsilon(x).
\end{align*}
This completes the proof of \eqref{lem11}.

 On the other hand, from $L^p$ estimate that whenever $z\in\partial B_{\delta}(\xi_{\varepsilon,l})$, we have
\begin{align*}
|| u_\varepsilon ||_{ W ^ { 2,p } ( B_{\frac{1}{4}\delta}(z) )}
\leq& C  || u_\varepsilon ||_{ L ^ { p } ( B_{\frac{1}{2}\delta}(z) )} + \frac{C}{\varepsilon^2}|| (\varepsilon^2 P(x)+1)u_\varepsilon-\mu_1 u_\varepsilon^3 -\beta u_\varepsilon v_\varepsilon^2 || _ { L ^ { p } ( B_{\frac{1}{2}\delta}(z) ) } \\
 \leq& \frac{C}{\varepsilon^2} e ^ { - \frac { \sqrt { \alpha } \delta } { 2 \varepsilon } }\leq e ^ { - \frac { \sqrt { \alpha } \delta } { 4 \varepsilon } }
\end{align*}
and
\begin{align*}
|| v_\varepsilon ||_{ W ^ { 2,p }  ( B_{\frac{1}{4}\delta}(z) )} \leq& C  || v_\varepsilon ||_{ L ^ { p } ( B_{\frac{1}{2}\delta}(z) )} + \frac{C}{\varepsilon^2}|| (\varepsilon^2 P(x)+1)v_\varepsilon-\mu_1 v_\varepsilon^3 -\beta v_\varepsilon u_\varepsilon^2 || _ { L ^ { p } ( B_{\frac{1}{2}\delta}(z) ) } \\
 \leq& \frac{C}{\varepsilon^2} e ^ { - \frac { \sqrt { \alpha } \delta } { 2 \varepsilon } }\leq e ^ { - \frac { \sqrt { \alpha } \delta } { 4 \varepsilon } }.
\end{align*}
 Taking $p>N$, \eqref{lem12} then follows from the above equations, and the Sobolev embedding theorem.
\end{proof}

Similarly, we can obtain some estimates of $(\varphi_{\varepsilon},\psi_{\varepsilon})$ in \eqref{uv}.

\begin{lemma}\label{lem34}
For $(\varphi_{\varepsilon},\psi_{\varepsilon})$ in \eqref{uv}, there exist constants $C>0$ large and $\tau>0$ small, such that
$$|\varphi_{\varepsilon}|\leq C\sum_{l=1}^k e^{-\frac{\tau|x-\xi_{\varepsilon,l}|}{\varepsilon}},\ \ \ \ |\psi_{\varepsilon}|\leq C\sum_{l=1}^k e^{-\frac{\tau|x-\xi_{\varepsilon,l}|}{\varepsilon}}, \ \ \ \ \forall x\in \mathbb{R}^N$$
and
$$|\nabla\varphi_{\varepsilon}|\leq C\sum_{l=1}^k e^{-\frac{\tau}{\varepsilon}},\ \ \ \ |\nabla\psi_{\varepsilon}|\leq C\sum_{l=1}^k e^{-\frac{\tau}{\varepsilon}}, \ \ \ \ \forall x\in \partial B_{\delta}(\xi_{\varepsilon,l}),\ \ l=1,\cdots,k. $$

\end{lemma}

To establish a local uniqueness result, we need to derive an estimate for $|\xi_{\varepsilon,l}-\xi_l|$.

\begin{lemma}\label{lem35}
If $(u_{\varepsilon},v_{\varepsilon}) $ is a solution to \eqref{eq1}, then
\begin{equation}\label{xx}
|\xi_{\varepsilon,l}-\xi_l| = O(\varepsilon^2).
\end{equation}
\end{lemma}

\begin{proof}
On the one hand , by Lemma \ref{lem33}, we know that there exists $\theta>0$, such that
$$ | u _ { \varepsilon } ( x ) | + | \nabla u _ { \varepsilon } ( x ) | \leq C e ^ { - \frac { \theta } { \varepsilon } } ,\ \ \  \forall x \in \partial B _ { \delta } ( \xi _ { \varepsilon , l } ) ,\ \ l = 1 , \cdots , k .$$
Therefore, using the local Pohozaev identity \eqref{p4}, we have
\begin{align}\label{353}
\int _ { B_{\delta}(\xi_{\varepsilon,l}) } \left(\frac { \partial P(x) } { \partial x _ { j } } u _ {\varepsilon } ^ { 2 }+\frac { \partial Q(x) } { \partial x _ { j } } v _ {\varepsilon } ^ { 2 }\right)=O\left(\varepsilon^{-2} e ^ { - \frac { \tau } { \varepsilon } }\right),
\end{align}
where $\tau>0$ is a small constant.

On the other hand, according to Lemma \ref{lem34}, we can see that
\begin{align}\label{354}
 &\int _ { B_{\delta}(\xi_{\varepsilon,l}) } \left(\frac { \partial P(x) } { \partial x _ { j } }-\frac { \partial P(\xi_{\varepsilon,l}) } { \partial x _ { j } }\right) u _ {\varepsilon } ^ { 2 }\notag\\
=&\int_{B_\delta(\xi_{\varepsilon,l})} \sum_{i=1}^N \frac{\partial^2 P(\xi_{\varepsilon,l})}{\partial x_i \partial x_j} (x_i - (\xi_{\varepsilon,l})_i) u_\varepsilon^2  + O\left( \int_{B_\delta(\xi_{\varepsilon,l})} |x - \xi_{\varepsilon,l}|^2 u_\varepsilon^2  \right)\notag\\
=&\int_{B_\delta(\xi_{\varepsilon,l})} \sum_{i=1}^N \frac{\partial^2 P(\xi_{\varepsilon,l})}{\partial x_i \partial x_j} (x_i - (\xi_{\varepsilon,l})_i) (U_{\varepsilon,l}^2+2U_{\varepsilon,l}\varphi_{\varepsilon}+\varphi_{\varepsilon}^2) +O(\varepsilon^{N+2})\notag\\
&+O\Big(\int_{B_\delta(\xi_{\varepsilon,l})} |x - \xi_{\varepsilon,l}| \Big( \sum _ { h \neq l }^k U _ { \varepsilon ,  h  } ^ { 2 } + 2 \sum _ { h \neq l}^k U _ { \varepsilon , h } U _ { \varepsilon ,  l  } + \sum _ { h \neq l }^k U _ { \varepsilon ,h  } \varphi _ { \varepsilon } \Big)\Big)\notag\\
=&\int_{B_\delta(\xi_{\varepsilon,l})} \sum_{i=1}^N \frac{\partial^2 P(\xi_{\varepsilon,l}) }{\partial x_i \partial x_j}(x_i - (\xi_{\varepsilon,l})_i) (U_{\varepsilon,l}^2+2U_{\varepsilon,l}\varphi_{\varepsilon}+\varphi_{\varepsilon}^2) +O(\varepsilon^{N+2}).
\end{align}
Since $U_{\varepsilon,l}$ is an even function with respect to $\xi_{\varepsilon,l}$, we have
\begin{align}\label{355}
\int_{B_\delta(\xi_{\varepsilon,l})} \sum_{i=1}^N \frac{\partial^2 P}{\partial x_i \partial x_j}(\xi_{\varepsilon,l}) (x_i - (\xi_{\varepsilon,l})_i) U_{\varepsilon,l}^2=0.
\end{align}
By H\"older inequality, we have
\begin{align}\label{356}
&\left|\int_{B_\delta(\xi_{\varepsilon,l})} \sum_{i=1}^N \frac{\partial^2 P}{\partial x_i \partial x_j}(\xi_{\varepsilon,l}) (x_i - (\xi_{\varepsilon,l})_i) U_{\varepsilon,l}\varphi_{\varepsilon} \right|\notag\\
\leq&C\left(\int_{B_\delta(\xi_{\varepsilon,l})} \left(\sum_{i=1}^N \frac{\partial^2 P}{\partial x_i \partial x_j}(\xi_{\varepsilon,l}) (x_i - (\xi_{\varepsilon,l})_i)\right)^2 U_{\varepsilon,l}^2\right)^{\frac{1}{2}}||\varphi_{\varepsilon}||_{L^2(\mathbb{R}^N)}= O\big(\varepsilon^{N+6}\big)
\end{align}
and similarly,
\begin{align}\label{357}
&\left|\int_{B_\delta(\xi_{\varepsilon,l})} \sum_{i=1}^N \frac{\partial^2 P}{\partial x_i \partial x_j}(\xi_{\varepsilon,l}) (x_i - (\xi_{\varepsilon,l})_i) \varphi_{\varepsilon}^2 \right|= O\big(\varepsilon^{N+10}\big).
\end{align}
Combining \eqref{354}-\eqref{357}, we get
\begin{align}\label{358}
 \int _ { B_{\delta}(\xi_{\varepsilon,l}) } \left(\frac { \partial P(x) } { \partial x _ { j } }-\frac { \partial P(\xi_{\varepsilon,l}) } { \partial x _ { j } }\right) u _ {\varepsilon } ^ { 2 }=O\big(\varepsilon^{N+2}\big).
\end{align}
Similar to \eqref{358}, we can also deduce that
\begin{align}\label{359}
 \int _ { B_{\delta}(\xi_{\varepsilon,l}) } \left(\frac { \partial Q(x) } { \partial x _ { j } }-\frac { \partial Q(\xi_{\varepsilon,l}) } { \partial x _ { j } }\right) v _ {\varepsilon } ^ { 2 }=O\big(\varepsilon^{N+2}\big).
\end{align}
By combining \eqref{353}, \eqref{358} and \eqref{359}, we can obtain
\begin{align*}
\int _ { B_{\delta}(\xi_{\varepsilon,l}) } \left(\frac { \partial P(\xi_{\varepsilon,l}) } { \partial x _ { j } } u _ {\varepsilon } ^ { 2 }+ \frac { \partial Q(\xi_{\varepsilon,l}) } { \partial x _ { j } } v _ {\varepsilon } ^ { 2 }\right)=O\big(\varepsilon^{N+2}\big).
\end{align*}
Substituting the expression of \( (u_\varepsilon, v_\varepsilon) \), we obtain
\begin{align*}
\int _ { B_{\delta}(\xi_{\varepsilon,l}) } \left(\frac { \partial P(\xi_{\varepsilon,l}) } { \partial x _ { j } } (w_{\varepsilon, \xi_{\varepsilon, l}}^{*})^2 + \frac { \partial Q(\xi_{\varepsilon,l}) } { \partial x _ { j } }( w_{\varepsilon, \xi_{\varepsilon, l}}^{\star})^2 \right)=O\big(\varepsilon^{N+2}\big),
\end{align*}
which means that
\begin{align*}
 \int _ { B_{\delta \varepsilon^{-1}}(0) }\left(\frac{\beta-\mu_2} { \beta ^ { 2 } - \mu_1 \mu_2 } \frac { \partial P(\xi_{\varepsilon,l}) } { \partial x _ { j } } + \frac{\beta-\mu_1} { \beta ^ { 2 } - \mu_1 \mu_2 } \frac { \partial Q(\xi_{\varepsilon,l}) } { \partial x _ { j } }  \right) w^2=O\big(\varepsilon^2\big).
\end{align*}
Therefore, we obtain
\begin{align*}
(\beta-\mu_2) \frac { \partial P(\xi_{\varepsilon,l}) } { \partial x _ { j } } + (\beta-\mu_1) \frac { \partial Q(\xi_{\varepsilon,l}) } { \partial x _ { j } }=O(\varepsilon^2),
\end{align*}
from which it follows that
\begin{equation}\label{26-1}
 (\beta-\mu_2) \nabla P( \xi_{\varepsilon,l} ) +  (\beta-\mu_1) \nabla Q( \xi_{\varepsilon,l} ) = O(\varepsilon^2).
\end{equation}
Based on \eqref{26-1} and combined with condition $(H_2)$, we can obtain
\begin{equation*}
|\xi_{\varepsilon,l}-\xi_l| = O(\varepsilon^2).
\end{equation*}

\end{proof}

In the subsequent uniqueness proof of the normalized solution, the conclusion of Lemma \ref{lem35} is insufficient to support the involved calculation process. Hence, we present a more precise estimate.

\begin{lemma}\label{lem355}
It follows that
\begin{equation}\label{xxx}
\xi_{\varepsilon,l}-\xi_l =C_{l} \varepsilon^2+O(\varepsilon^{5}) \quad \text{for}\quad l=1,2,\cdots,k,
\end{equation}
where each \( C_l \) is a fixed constant vector.
\end{lemma}
\begin{proof}
By using the local Pohozaev identity \eqref{p4}, we have
\begin{align}\label{3541-0}
\int _ { B_{\delta}(\xi_{\varepsilon,l}) } \left(\frac { \partial P(x) } { \partial x _ { j } } u _ {\varepsilon } ^ { 2 }+\frac { \partial Q(x) } { \partial x _ { j } } v _ {\varepsilon } ^ { 2 }\right)=O\big(\varepsilon^{-2} e ^ { - \frac { \tau } { \varepsilon } }\big),
\end{align}
where $\tau>0$ is a small constant. Noting that $\|(\varphi_{\varepsilon}, \psi_{\varepsilon})\|_{\varepsilon} =O(\varepsilon^{\frac{N}{2}+5}),$ by H\"older inequality and Sobolev imbedding theorem, it follows that
\begin{align}\label{3541-1}
&\int _ { B_{\delta}(\xi_{\varepsilon,l}) } \left(\frac { \partial P(x) } { \partial x _ { j } }(2U_\varepsilon\varphi_\varepsilon+ \varphi_\varepsilon ^2)+
\frac { \partial Q(x) } { \partial x _ { j } } (2V_\varepsilon\psi_\varepsilon+\psi_\varepsilon ^2)\right)
\notag\\
&\leq  \left(\int _ { B_{\delta}(\xi_{\varepsilon,l}) }\Big|\frac { \partial P(x) } { \partial x _ { j } }\Big| U_\varepsilon ^2\right)^{\frac{1}{2}}
\left(\int _ { B_{\delta}(\xi_{\varepsilon,l}) }\Big|\frac { \partial P(x) } { \partial x _ { j } }\Big| \varphi_\varepsilon ^2\right)^{\frac{1}{2}}
+\left(\int _ { B_{\delta}(\xi_{\varepsilon,l}) }\Big|\frac { \partial Q(x) } { \partial x _ { j } }\Big| V_\varepsilon ^2\right)^{\frac{1}{2}}
\left(\int _ { B_{\delta}(\xi_{\varepsilon,l}) }\Big|\frac { \partial Q(x) } { \partial x _ { j } }\Big| \psi_\varepsilon ^2\right)^{\frac{1}{2}}+
\notag\\
&+C\|(\varphi_{\varepsilon}, \psi_{\varepsilon})\|_{\varepsilon} ^{2}\notag\\
&\leq C\varepsilon^{\frac{N}{2}}\|(\varphi_{\varepsilon}, \psi_{\varepsilon})\|_{\varepsilon}++C\|(\varphi_{\varepsilon}, \psi_{\varepsilon})\|_{\varepsilon} ^{2}\notag\\
&
=O(\varepsilon^{N+5}).
\end{align}
Thus, combining \eqref{3541-0} with \eqref{3541-1}, we have

\begin{align}\label{3541}
	\int _ { B_{\delta}(\xi_{\varepsilon,l}) } \left(\frac { \partial P(x) } { \partial x _ { j } } U_\varepsilon ^2+\frac { \partial Q(x) } { \partial x _ { j } } V_\varepsilon ^2\right)=O(\varepsilon^{N+5}).
	\end{align}
Applying Taylor expansion and Lemma \ref{lem35}, we obtain
\begin{align}\label{3542}
&\int _ { B_{\delta}(\xi_{\varepsilon,l}) } \frac { \partial P(x) } { \partial x _ { j } } U_\varepsilon ^2 \notag\\
=&2p_{lj} ( \xi _ { \varepsilon , l , j} -\xi _ {l ,j } ) \varepsilon^N \int _ { \mathbb{R}^N } (w^*+\varepsilon^4 W_l^*)^2 + \frac{1}{2N}\triangle \frac { \partial P(\xi_{l}) } { \partial x _ { j } } \varepsilon^{N+2} \int _ { \mathbb{R}^N } |x|^2 (w^*+\varepsilon^4 W_l^*)^2+O(\varepsilon^{N+5})\notag\\
=&2\varepsilon^N  p_{lj} ( \xi _ { \varepsilon , l , j } -\xi _ {l ,j } )  \int _ { \mathbb{R}^N } (w^*)^2 +  \frac{\varepsilon^{N+2} }{2N}\triangle \frac { \partial P(\xi_{l}) } { \partial x _ { j } } \int _ { \mathbb{R}^N } |x|^2 (w^*)^2 \notag\\
&+ 4 \varepsilon^{N+4} p_{lj} ( \xi _ { \varepsilon , l , j } -\xi _ {l ,j } ) \int _ { \mathbb{R}^N } w^*W_l^*+O(\varepsilon^{N+5}).
\end{align}
Similarly, we also have
\begin{align}\label{3543}
\int _ { B_{\delta}(\xi_{\varepsilon,l}) } \frac { \partial Q(x) } { \partial x _ { j } } V_\varepsilon ^2
=&2\varepsilon^N  q_{lj} ( \xi _ { \varepsilon , l , j} -\xi _ {l ,j } )  \int _ { \mathbb{R}^N } (w^\star)^2 +  \frac{\varepsilon^{N+2} }{2N}\triangle \frac { \partial Q(\xi_{l}) } { \partial x _ { j } } \int _ { \mathbb{R}^N } |x|^2 (w^\star)^2 \notag\\
&+ 4 \varepsilon^{N+4}  q_{lj} ( \xi _ { \varepsilon , l ,j } -\xi _ {l ,j } ) \int _ { \mathbb{R}^N } w^\star W_l^\star+O(\varepsilon^{N+5}).
\end{align}
Substituting \eqref{3542} and \eqref{3543} into \eqref{3541}, and summing over \(j = 1, \dots, N\), it follows that \eqref{xxx} holds.
\end{proof}

\subsection{Proof of Theorem \ref{main}}
In this subsection, we will prove Theorem \ref{main} by contradiction, employing the Pohozaev identity combined with blow-up analysis.

 Suppose that there are two different solutions $(u^{(1)}_{\varepsilon},v^{(1)}_{\varepsilon})$ and $(u^{(2)}_{\varepsilon},v^{(2)}_{\varepsilon})$ to \eqref{eq1}. We define
\begin{equation*}
(\eta^{(1)}_\varepsilon,  \eta^{(2)}_\varepsilon) =\left(\frac{u^{(1)}_{\varepsilon}-u^{(2)}_{\varepsilon}}{||u^{(1)}_{\varepsilon}-u^{(2)}_{\varepsilon}||_{L^{\infty}(\mathbb{R}^N)}+||v^{(1)}_{\varepsilon}-v^{(2)}_{\varepsilon}||_{L^{\infty}(\mathbb{R}^N)}}, \frac{v^{(1)}_{\varepsilon}-v^{(2)}_{\varepsilon}}{||u^{(1)}_{\varepsilon}-u^{(2)}_{\varepsilon}||_{L^{\infty}(\mathbb{R}^N)}+||v^{(1)}_{\varepsilon}-v^{(2)}_{\varepsilon}||_{L^{\infty}(\mathbb{R}^N)}} \right).
\end{equation*}
Then, $(\eta^{(1)}_\varepsilon, \eta^{(2)}_\varepsilon)$ satisfies
 \begin{equation} \label{eqeta}
  \left\{\begin{array}{ll}
		-\varepsilon ^2 \triangle \eta^{(1)}_\varepsilon+\left( \varepsilon ^2P\left( x \right) +1 -\mu_1  ((u^{(1)}_{\varepsilon})^2+u^{(1)}_{\varepsilon}u^{(2)}_{\varepsilon}+(u^{(2)}_{\varepsilon})^2)\right)\eta^{(1)}_\varepsilon\\
		\ \ \ \ \ \ \ \ \ \ \ \ \ \ \ \ \ \ \ \ \ \ \ \ \ \ \ \ \ \ \ \ \ \ \ \ \ \  -\beta\left[\eta^{(1)}_\varepsilon(v_{\varepsilon}^{(2)})^2
+u^{(1)}_{\varepsilon}\eta^{(2)}_\varepsilon(v_{\varepsilon}^{(1)}+v_{\varepsilon}^{(2)})\right]=0,\,\,\,\,x\in\mathbb{R}^N,\\
-\varepsilon ^2 \triangle \eta^{(2)}_\varepsilon+\left( \varepsilon ^2Q\left( x \right) +1 -\mu_2 ((v^{(1)}_{\varepsilon})^2+v^{(1)}_{\varepsilon}v^{(2)}_{\varepsilon}+(v^{(2)}_{\varepsilon})^2)\right) \eta^{(2)}_\varepsilon\\
		\ \ \ \ \ \ \ \ \ \ \ \ \ \ \ \ \ \ \ \ \ \ \ \ \ \ \ \ \ \ \ \ \ \ \ \ \ \
- \beta \left[(u_\varepsilon^{(2)})^2 \eta_\varepsilon^{(2)}+ v_\varepsilon^{(1)} \eta_\varepsilon^{(1)} (u_\varepsilon^{(1)} + u_\varepsilon^{(2)}) \right] =0 \,\,\,\,x\in\mathbb{R}^N.
	\end{array} \right.
\end{equation}
Since
$$\|\eta^{(1)}_\varepsilon\|_{L^{\infty}(\mathbb{R}^N)}+ \|\eta^{(2)}_\varepsilon\|_{L^{\infty}(\mathbb{R}^N)}=1,$$
a contradiction arises if we can show that $\|\eta^{(1)}_\varepsilon\|_{L^{\infty}(\mathbb{R}^N)}\to 0$ and $\|\eta^{(2)}_\varepsilon\|_{L^{\infty}(\mathbb{R}^N)}\to 0$ as $\varepsilon\to 0$. To establish this, we decompose the proof into two parts, we first carry out the estimation of $\eta^{(1)}_\varepsilon$ and $\eta^{(2)}_\varepsilon$ in $\mathbb{R}^N \setminus \bigcup^k_{l=1} B_{R\varepsilon}(\xi^{(1)}_{\varepsilon,l})$. For this part of the estimation, we appeal to the following lemma.

\begin{lemma}\label{lem36}
There exist constants $C>0$ and small $\tau>0$, such that
\begin{align}\label{361}
| \eta^{(i )}_\varepsilon ( x ) | \leq C \sum _ { l = 1 } ^ { k } e ^ { -  \frac { \tau| x - \xi^{(1)}_{\varepsilon,l} | } { \varepsilon } },\ \ \ \forall x \in \mathbb{R}^N \backslash \bigcup^k_{l=1} B_{R\varepsilon}(\xi^{(1)}_{\varepsilon,l}),\ \ \ i=1,2,
\end{align}
and
\begin{align}\label{362}
| \nabla \eta^{(i )}_\varepsilon| \leq C e ^ { - \frac { \tau } { \varepsilon } },\ \ \ \forall x \in \partial B_{R\varepsilon}(\xi^{(1)}_{\varepsilon,l}),\ \ \ i=1,2,\ \ \ l=1,\cdots,k.
\end{align}
\end{lemma}

\begin{proof}
By Lemma \ref{lem35}, we have $|\xi_{\varepsilon,l}^{(1)}-\xi_{\varepsilon,l}^{(2)}|=O(\varepsilon^2)$. Since
\begin{align*}
|y-\xi_{\varepsilon,l}^{(2)}|\geq |y-\xi_{\varepsilon,l}^{(1)}|-|\xi_{\varepsilon,l}^{(1)}-\xi_{\varepsilon,l}^{(2)}|\geq R\varepsilon+O(\varepsilon^2),\quad \forall y \in \mathbb{R}^N \backslash \bigcup^k_{l=1} B_{R\varepsilon}(\xi^{(1)}_{\varepsilon,l}),
\end{align*}
 we obtain
\begin{align*}
\big( \varepsilon ^2P\left( x \right) +1 \big)-\left[\mu_1  ((u^{(1)}_{\varepsilon})^2+u^{(1)}_{\varepsilon}u^{(2)}_{\varepsilon}+(u^{(2)}_{\varepsilon})^2)+\beta(v_{\varepsilon}^{(2)})^2\right]\geq \frac12
\end{align*}
and similarly,
\begin{align*}
\big( \varepsilon ^2Q\left( x \right) +1 \big) -\left[\mu_2 ((v^{(1)}_{\varepsilon})^2+v^{(1)}_{\varepsilon}v^{(2)}_{\varepsilon}+(v^{(2)}_{\varepsilon})^2) + \beta (u_\varepsilon^{(2)})^2 \right]\geq \frac12.
\end{align*}
For any small $\bar{\tau}>0$, based on Lemma \ref{lem33} and \eqref{xx}, we can conclude that there is a large $R>0$ such that
\begin{align*}
\left|\mu_1  ((u^{(1)}_{\varepsilon})^2+u^{(1)}_{\varepsilon}u^{(2)}_{\varepsilon}+(u^{(2)}_{\varepsilon})^2)+\beta v_{\varepsilon}^{(2)})^2\right|\leq \bar{\tau},\ \ \ \forall x \in \mathbb{R}^N \backslash \bigcup^k_{l=1} B_{R\varepsilon}(\xi^{(1)}_{\varepsilon,l})
\end{align*}
and
\begin{align*}
\big|\beta u^{(1)}_{\varepsilon}\eta^{(2)}_\varepsilon(v_{\varepsilon}^{(1)}+v_{\varepsilon}^{(2)})\big|\leq C \sum _ { l = 1 } ^ { k } e ^ { -  \theta R } .
\end{align*}
 With the application of \eqref{eqeta} together with the comparison theorem, the proof of \eqref{361} can be accomplished.

By employing the $L^p$ estimate for elliptic equations, \eqref{362} can be validated.
\end{proof}

Next, we estimate $\eta^{(1)}_\varepsilon$, $\eta^{(2)}_\varepsilon$ in $B_{R\varepsilon}(\xi^{(1)}_{\varepsilon,l})$, $l=1,\cdots,k$. We aim to obtain $\eta^{(1)}_\varepsilon=o(1)$ and $\eta^{(2)}_\varepsilon=o(1)$ in $B_{R\varepsilon}(\xi^{(1)}_{\varepsilon,l})$, $l=1,\cdots,k$.  Let
\begin{align*}
\tilde{\eta}^{(i)}_{\varepsilon,l}(x)=\eta^{(i)}_\varepsilon\big(\varepsilon x + \xi_{\varepsilon,l}^{(1)}\big),\ \ \ l=1,\cdots,k.
\end{align*}

We know that
\begin{align}\label{3113}
u_{\varepsilon}^{(1)}\big(\varepsilon x + \xi_{\varepsilon,l}^{(1)}\big)\rightarrow w^{*}(x)
\end{align}
uniformly in $B_{R}(0)$ for any $R>0$.

On the other hand, from \eqref{xx}, we have
\begin{align}\label{3114}
u_{\varepsilon}^{(2)}\big(\varepsilon x + \xi_{\varepsilon,l}^{(1)}\big)\rightarrow w^{*}(x)
\end{align}
uniformly in $B_{R}(0)$ for any $R>0$. Similarly, we can also conclude that
\begin{align}\label{31115}
v_{\varepsilon}^{(1)}\big(\varepsilon x + \xi_{\varepsilon,l}^{(1)}\big)\rightarrow w^{\star}(x),\ \ \ v_{\varepsilon}^{(2)}\big(\varepsilon x + \xi_{\varepsilon,l}^{(1)}\big)\rightarrow w^{\star}(x)
\end{align}
uniformly in $B_{R}(0)$ for any $R>0$.

 Based on the $L^p$ estimate, the Sobolev embedding theorem and the Schauder estimate, we get that $\tilde{\eta}^{(i)}_{\varepsilon,l}\rightarrow \eta^{(i)}_{l}$, $i=1,2$, in $C^{2}(B_R (0))$ for any $R>0$. Combining \eqref{3113}, \eqref{3114} and \eqref{31115}, $(\eta^{(1)}_{l}, \eta^{(2)}_{l})$ satisfies
\begin{equation*}
  \left\{\begin{array}{ll}
		- \Delta \eta^{(1)}_l  + \eta^{(1)}_l =3\mu_1 (w^{*})^2 \eta^{(1)}_l +\beta\left( (w^{\star})^2 \eta^{(1)}_l+2w^{*}w^{\star} \eta^{(2)}_l\right),\,\,\,\,x\in \mathbb{R}^N,\vspace{0.12cm}\\
		-\Delta \eta^{(2)}_l  + \eta^{(2)}_l =3\mu_2 (w^{\star})^2 \eta^{(2)}_l +\beta \left( (w^{*})^2 \eta^{(2)}_l+2w^{*}w^{\star} \eta^{(1)}_l\right),\,\,\,\,x\in \mathbb{R}^N.
	\end{array} \right.
\end{equation*}
Therefore, according to the non-degeneracy of $(w ^ { * },w ^ { \star })$, there exist constants $a_{l,j}$, $l=1,\cdots,k$, $j=1,\cdots,N$, such that
\begin{align}\label{3116}
(\eta^{(1)}_{l},\eta^{(2)}_{l})=\sum_{h=1}^N a_{l,h}  \left( \frac { \partial w ^ { * } } { \partial x _ { h } } , \frac { \partial w ^ { \star } } { \partial x _ { h } } \right) .
\end{align}
Applying \eqref{p4} to $(u_\varepsilon^{(i)},v_\varepsilon^{(i)})$, we have
\begin{align}\label{p3117}
&\varepsilon^{2}\int _ {  B_{\delta}(\xi_{\varepsilon,l}^{(1)}) } \left(\frac { \partial P(x) } { \partial x _ { j } }\left( u _ {\varepsilon }^{(1)}+u _ {\varepsilon }^{(2)}\right)\eta^{(1)}_\varepsilon +\frac { \partial Q(x) } { \partial x _ { j } } \left( v _ {\varepsilon }^{(1)}+v _ {\varepsilon }^{(2)}\right)\eta^{(2)}_\varepsilon\right) \notag\\
=&-\beta\int _ { \partial  B_{\delta}(\xi_{\varepsilon,l}^{(1)}) }  \left( \eta _ { \varepsilon } ^ { ( 2 ) } ( v _ { \varepsilon } ^ { ( 1 ) } + v _ { \varepsilon } ^ { ( 2 ) } ) ( u _ { \varepsilon } ^ { ( 1 ) } ) ^ { 2 } + ( v _ { \varepsilon } ^ { ( 2 ) } ) ^ { 2 } \eta _ { \varepsilon } ^ { ( 1 ) } ( u _ { \varepsilon } ^ { ( 1 ) } + u _ { \varepsilon } ^ { ( 2 ) } ) \right) n _ { j } d S \notag\\
&- 2 \varepsilon ^ { 2 } \int _ { \partial  B_{\delta}(\xi_{\varepsilon,l}^{(1)}) }  \Big( \frac { \partial \eta _ { \varepsilon } ^ { ( 1 ) } } { \partial x _ { j } } \frac { \partial u _ { \varepsilon } ^ { ( 1 ) } } { \partial n } + \frac { \partial u _ { \varepsilon } ^ { ( 2 ) } } { \partial x _ { j } } \frac { \partial \eta _ { \varepsilon } ^ { ( 1 ) } } { \partial n } + \frac { \partial \eta _ { \varepsilon } ^ { ( 2 ) } }{ \partial x _ { j } } \frac { \partial v _ { \varepsilon } ^ { ( 1 ) } }{ \partial n } +\frac { \partial v _ { \varepsilon } ^ { ( 2 ) } } { \partial x _ { j } } \frac { \partial \eta _ { \varepsilon } ^ { ( 2 ) } } { \partial n } \Big) d S \notag\\
&+ \varepsilon ^ { 2 } \int _ { \partial  B_{\delta}(\xi_{\varepsilon,l}^{(1)}) } \left(  \nabla \eta _ { \varepsilon } ^ { ( 1 ) }  \nabla ( u _ { \varepsilon } ^ { ( 1 ) } + u _ { \varepsilon } ^ { ( 2 ) } )  n _ { j } +  \nabla \eta _ { \varepsilon } ^ { ( 2 ) }  \nabla ( v _ { \varepsilon } ^ { ( 1 ) } + v _ { \varepsilon } ^ { ( 2 ) } )  n _ { j } \right)d S \notag\\
&+ \int _ { \partial  B_{\delta}(\xi_{\varepsilon,l}^{(1)}) }\left((\varepsilon^{2}P(x)+1) \left( u _ {\varepsilon }^{(1)}+u _ {\varepsilon }^{(2)}\right)\eta^{(1)}_\varepsilon n _ { j }+(\varepsilon^{2}Q(x)+1) \left( v _ {\varepsilon }^{(1)}+v _ {\varepsilon }^{(2)}\right)\eta^{(2)}_\varepsilon n _ { j }\right) d S\notag\\
& - \frac { 1 } { 2 } \Bigg(\mu_1\int _ { \partial  B_{\delta}(\xi_{\varepsilon,l}^{(1)}) }  \eta _ { \varepsilon } ^ { ( 1 ) } \left( ( u _ { \varepsilon } ^ { ( 1 ) } ) ^ { 3 } + ( u _ { \varepsilon } ^ { ( 1 ) } ) ^ { 2 } u _ { \varepsilon } ^ { ( 2 ) } + u _ { \varepsilon } ^ { ( 1 ) } ( u _ { \varepsilon } ^ { ( 2 ) } ) ^ { 2 } + ( u _ { \varepsilon } ^ { ( 2 ) } ) ^ { 3 } \right) n _ { j } d S \notag\\
&+\mu_2\int _ { \partial  B_{\delta}(\xi_{\varepsilon,l}^{(1)}) } \eta _ { \varepsilon } ^ { ( 2 ) } \left( ( v _ { \varepsilon } ^ { ( 1 ) } ) ^ { 3 } + ( v _ { \varepsilon } ^ { ( 1 ) } ) ^ { 2 } v _ { \varepsilon } ^ { ( 2 ) } + v _ { \varepsilon } ^ { ( 1 ) } ( v _ { \varepsilon } ^ { ( 2 ) } ) ^ { 2 } + ( v _ { \varepsilon } ^ { ( 2 ) } ) ^ { 3 } \right) n _ { j } d S \Bigg).
\end{align}

Based on the above analysis, we proceed to present the proof of Theorem \ref{main}.

\begin{proof}[Proof of Theorem \ref{main}]
According to Lemma \ref{lem36} and \eqref{p3117}, we conclude that
\begin{align*}
\int _ {  B_{\delta}(\xi_{\varepsilon,l}^{(1)}) } \left(\frac { \partial P(x) } { \partial x _ { j } }\left( u _ {\varepsilon }^{(1)}+u _ {\varepsilon }^{(2)}\right)\eta^{(1)}_\varepsilon +\frac { \partial Q(x) } { \partial x _ { j } } \left( v _ {\varepsilon }^{(1)}+v _ {\varepsilon }^{(2)}\right)\eta^{(2)}_\varepsilon\right)=O(\varepsilon^{-2}e^{-\frac{\tau}{\varepsilon}}),\ \ \ j=1,\cdots,N,
\end{align*}
where $\tau>0$ is a small constant. Therefore, we have
\begin{align}\label{3701}
\int _ {  B_{\delta \varepsilon^{-1}}(0) } \Bigg(&\frac { \partial P(\varepsilon x+\xi_{\varepsilon,l}^{(1)}) } { \partial x _ { j } }\Big( u _ {\varepsilon }^{(1)}(\varepsilon x+\xi_{\varepsilon,l}^{(1)})+u _ {\varepsilon }^{(2)}(\varepsilon x+\xi_{\varepsilon,l}^{(1)})\Big)\tilde{\eta}^{(1)}_{\varepsilon,l}\notag\\
&+\frac { \partial Q(\varepsilon x+\xi_{\varepsilon,l}^{(1)}) } { \partial x _ { j } } \left( v _ {\varepsilon }^{(1)}(\varepsilon x+\xi_{\varepsilon,l}^{(1)})+v _ {\varepsilon }^{(2)}(\varepsilon x+\xi_{\varepsilon,l}^{(1)})\right)\tilde{\eta}^{(2)}_{\varepsilon,l}\Bigg)=O\big(\varepsilon^{-(N+2)}e^{-\frac{\tau}{\varepsilon}}\big).
\end{align}
Since $|\xi_{\varepsilon,l}^{(i)}-\xi_l| = O(\varepsilon^2)$, we can obtain that
\begin{align*}
\frac { \partial P(\xi_{\varepsilon,l}^{(1)}) } { \partial x _ { j } }=\frac { \partial P(x_{l}) } { \partial x _ { j } }+O(|\xi_{\varepsilon,l}^{(1)}-\xi_l|) = O(\varepsilon^2)
\end{align*}
and
\begin{align*}
\frac { \partial Q(\xi_{\varepsilon,l}^{(1)}) } { \partial x _ { j } }=\frac { \partial Q(x_{l}) } { \partial x _ { j } }+O(|\xi_{\varepsilon,l}^{(1)}-\xi_l|) = O(\varepsilon^2).
\end{align*}
This implies
\begin{align}\label{tho49}
\int _ {  B_{\delta \varepsilon^{-1}}(0) } \Bigg(&\frac { \partial P(\xi_{\varepsilon,l}^{(1)}) } { \partial x _ { j } }\Big( u _ {\varepsilon }^{(1)}(\varepsilon x+\xi_{\varepsilon,l}^{(1)})+u _ {\varepsilon }^{(2)}(\varepsilon x+\xi_{\varepsilon,l}^{(1)})\Big)\tilde{\eta}^{(1)}_{\varepsilon,l} \notag\\
&+\frac { \partial Q(\xi_{\varepsilon,l}^{(1)}) } { \partial x _ { j } } \left( v _ {\varepsilon }^{(1)}(\varepsilon x+\xi_{\varepsilon,l}^{(1)})+v _ {\varepsilon }^{(2)}(\varepsilon x+\xi_{\varepsilon,l}^{(1)})\right)\tilde{\eta}^{(2)}_{\varepsilon,l}\Bigg)=O(\varepsilon^{2}).
\end{align}
Combining \eqref{3701} and \eqref{tho49}, we have
\begin{align*}
\int _ {  B_{\delta \varepsilon^{-1}}(0) } \Bigg(&\Bigg(\frac { \partial P(\varepsilon x+\xi_{\varepsilon,l}^{(1)}) } { \partial x _ { j } }-\frac { \partial P(\xi_{\varepsilon,l}^{(1)}) } { \partial x _ { j } }\Bigg)\Big( u _ {\varepsilon }^{(1)}(\varepsilon x+\xi_{\varepsilon,l}^{(1)})+u _ {\varepsilon }^{(2)}(\varepsilon x+\xi_{\varepsilon,l}^{(1)})\Big)\tilde{\eta}^{(1)}_{\varepsilon,l}\\
&+\Bigg(\frac { \partial Q(\varepsilon x+\xi_{\varepsilon,l}^{(1)}) } { \partial x _ { j } }-\frac { \partial Q(\xi_{\varepsilon,l}^{(1)}) } { \partial x _ { j } }\Bigg) \left( v _ {\varepsilon }^{(1)}(\varepsilon x+\xi_{\varepsilon,l}^{(1)})+v _ {\varepsilon }^{(2)}(\varepsilon x+\xi_{\varepsilon,l}^{(1)})\right)\tilde{\eta}^{(2)}_{\varepsilon,l}\Bigg)=O(\varepsilon^{2}).
\end{align*}
By using the Taylor expansion, we obtain that
\begin{align}\label{3702}
\varepsilon\sum_{h=1}^N \int _ {  B_{\delta \varepsilon^{-1}}(0) } \Bigg(&\frac { \partial^2 P(\xi_{\varepsilon,l}^{(1)}) } { \partial x _ { j }\partial x _ { h } } x _ { h } \Big( u _ {\varepsilon }^{(1)}(\varepsilon x+\xi_{\varepsilon,l}^{(1)})+u _ {\varepsilon }^{(2)}(\varepsilon x+\xi_{\varepsilon,l}^{(1)})\Big)\tilde{\eta}^{(1)}_{\varepsilon,l}\notag\\
&+\frac { \partial^2 Q(\xi_{\varepsilon,l}^{(1)}) } { \partial x _ { j }\partial x _ { h } } x _ { h } \left( v _ {\varepsilon }^{(1)}(\varepsilon x+\xi_{\varepsilon,l}^{(1)})+v _ {\varepsilon }^{(2)}(\varepsilon x+\xi_{\varepsilon,l}^{(1)})\right)\tilde{\eta}^{(2)}_{\varepsilon,l}\Bigg)=O(\varepsilon^{2}).
\end{align}
Letting $\varepsilon\rightarrow 0$ in \eqref{3702}, by \eqref{3113}-\eqref{3116}, we get
\begin{align*}
 \int _ {  \mathbb{R}^N } \Bigg(p_{lj} x _ { j } w^{*} \sum_{i=1}^3 a_{l,i}  \frac { \partial w^{*} } { \partial x _ { i } }+q_{lj} x _ { j } w^{\star} \sum_{i=1}^3 a_{l,i}  \frac { \partial w^{\star} } { \partial x _ { i } } \Bigg)=0.
\end{align*}
Therefore,
\begin{align}\label{3703}
\frac{1} { \beta ^ { 2 } - \mu_1 \mu_2 } \left(( \beta - \mu_2 )p_{lj} +( \beta - \mu_1 )q_{lj} \right) a_{l,j} \int _ {  \mathbb{R}^N } x _ { j } w   \frac { \partial w } { \partial x _ { j } }=0,\ \ j=1,\cdots,N.
\end{align}
But
\begin{align*}
\int _ {  \mathbb{R}^N } x _ { j } w   \frac { \partial w } { \partial x _ { j } }=\frac12\int _ {  \mathbb{R}^N } x _ { j }   \frac { \partial w^2 } { \partial x _ { j } }=-\frac12\int _ {  \mathbb{R}^N } w^2<0,
\end{align*}
and according to \eqref{3703}, we have
\begin{align*}
\left(( \beta - \mu_2 )p_{lj} +( \beta - \mu_1) q_{lj} \right) a_{l,j}=0.
\end{align*}
 By condition $(H_2)$, we obtain that $a_{l,j}=0$, $j=1,\cdots,N$, $l=1,\cdots,k$.
\end{proof}
To summarize the foregoing, we have demonstrated that $\eta^{(i)}_\varepsilon=o(1)$ in $B_{R\varepsilon}(\xi_{\varepsilon,l}^{(1)})$, $l=1,\cdots,N$. This combining with Lemma \ref{lem36}, we have $||\eta^{(1)}_\varepsilon||_{L^{\infty}(\mathbb{R}^N)}+||\eta^{(2)}_\varepsilon||_{L^{\infty}(\mathbb{R}^N)}=o(1)$. This contradicts $||\eta^{(1)}_\varepsilon||_{L^{\infty}(\mathbb{R}^N)}+||\eta^{(2)}_\varepsilon||_{L^{\infty}(\mathbb{R}^N)}=1$.

We now proceed to prove the continuity of $(u,v)$ with respect to $\varepsilon$ in the space \(L^2 (\mathbb{R}^N)\times L^2(\mathbb{R}^N)\).

\begin{corollary}\label{con1}
Let $(u_{\varepsilon}, v_{\varepsilon}) \in L^2(\mathbb{R}^N) \times L^2(\mathbb{R}^N)$ denote the solutions to system \eqref{eq1} corresponding to the parameter $\varepsilon > 0$. For any fixed $\hat{\varepsilon} > 0$, the solutions satisfy the following convergence:
\begin{equation*}
(u_{\varepsilon}, v_{\varepsilon}) \to (u_{\hat{\varepsilon}}, v_{\hat{\varepsilon}}) \quad \text{in}\ \ L^2(\mathbb{R}^N) \times L^2(\mathbb{R}^N), \quad \text{as}\ \ \varepsilon \to \hat{\varepsilon}.
\end{equation*}
This convergence directly implies the continuity of the solution pair $(u_{\varepsilon}, v_{\varepsilon})$ with respect to the parameter $\varepsilon$ in $L^2(\mathbb{R}^N) \times L^2(\mathbb{R}^N)$.
\end{corollary}

\begin{proof}
 Suppose that $(u_{\varepsilon_n},v_{\varepsilon_n})$ does not converge to $(u_{\hat{\varepsilon}},v_{\hat{\varepsilon}})$ as $\varepsilon_n \rightarrow\hat{\varepsilon}$. According to Theorem \ref{pro1}, we have
\begin{equation*}
\left\| \left( \varphi_{\varepsilon_n}, \psi_{\varepsilon_n} \right) \right\|_{H} \leq C\varepsilon_n ^{5+\frac{N}{2}}.
\end{equation*}
Therefore, there exist subsequence of $\{( \varphi_{\varepsilon_n}, \psi_{\varepsilon_n})\}$ (we still denote as $\{( \varphi_{\varepsilon_n}, \psi_{\varepsilon_n})\}$), and $\{( \hat{\varphi}, \hat{\psi})\}\in H^1 (\mathbb{R}^N)\times H^1 (\mathbb{R}^N)$, such that
\begin{align}\label{4140}
&\varphi_{\varepsilon_n} \rightharpoonup \hat{\varphi}, \ \ \text{in} \ \ H^1(\mathbb{R}^N), \ \ \ \psi_{\varepsilon_n}\rightharpoonup \hat{\psi}, \ \ \text{in} \ \ H^1(\mathbb{R}^N),\notag\\
&\varphi_{\varepsilon_n} \rightarrow \hat{\varphi}, \ \ \text{in} \ \ L^q_{loc}(\mathbb{R}^N), \ \ \ \psi_{\varepsilon_n}\rightarrow \hat{\psi}, \ \ \text{in} \ \ L^q_{loc}(\mathbb{R}^N),\ \ q\in[2,2^*)
\end{align}
and
\begin{equation}\label{4141}
||( \hat{\varphi}, \hat{\psi})||_{H}\leq C\hat{\varepsilon} ^{5+\frac{N}{2}}.
\end{equation}
Noting that $|\xi_{\varepsilon,l}-\xi_l|=O(\varepsilon^2)$, there exists $\hat{x}_l\in B_\delta (\xi_l)$ such that
\begin{equation}\label{4142}
\xi_{\varepsilon_n,l}\rightarrow \hat{x}_l,\ \ \text{as} \ \ n\rightarrow \infty.
\end{equation}

On the one hand, by \eqref{4140} and \eqref{4142}, letting $n\rightarrow \infty$, we get
 \begin{equation}\label{eq41} \left\{\begin{array}{ll}
		-\hat{\varepsilon} ^2 \triangle \hat{u}+\left( \hat{\varepsilon}^2 P\left( x \right) +1 \right) \hat{u}=\mu_1  \hat{u}^3+\beta \hat{u}\hat{v}^2,\,\, \text{in} \,\,\mathbb{R}^N,\\
		-\hat{\varepsilon} ^2 \triangle \hat{v}+\left( \hat{\varepsilon}^2 Q\left( x \right) +1 \right) \hat{v}=\mu_2  \hat{v}^3+\beta \hat{v}\hat{u}^2,\,\, \text{in} \,\,\mathbb{R}^N,\\
	\end{array} \right.
\end{equation}
where
\begin{align}
\hat{u}=\sum_{l=1}^k \left( w^{*}\left(\frac{x - \hat{x}_l}{\hat{\varepsilon}}\right) + \hat{\varepsilon}^4 W_l^{*}\left(\frac{x - \hat{x}_l}{\hat{\varepsilon}}\right) \right) + \hat{\varphi},\\
\hat{v}=\sum_{l=1}^k \left( w^{\star}\left(\frac{x - \hat{x}_l}{\hat{\varepsilon}}\right) + \hat{\varepsilon}^4 W_l^{\star}\left(\frac{x - \hat{x}_l}{\hat{\varepsilon}}\right) \right) + \hat{\psi}.
\end{align}
Additionally, it is further crucial to note that \( (\hat{u},\hat{v})\neq (u_{\hat{\varepsilon}},v_{\hat{\varepsilon}}) \), for if these two pairs were identical, this would immediately contradict our initial assumption that $(u_{\varepsilon_n},v_{\varepsilon_n})$ does not converge to \( (u_{\hat{\varepsilon}},v_{\hat{\varepsilon}}) \).
On the other hand, by Lemma \ref{lem34} and \eqref{4142}, we obtain that
\begin{align*}
|\varphi_{\varepsilon_n}(x)|\leq C\sum_{l=1}^k e^{-\frac{\tau|x-\hat{x}_l+(\hat{x}_l-\xi_{\varepsilon_n,l})|}{\hat{\varepsilon}+(\varepsilon_n-\hat{\varepsilon})}}\leq C\sum_{l=1}^k e^{-\frac{\tau|x-\hat{x}_l|}{\hat{\varepsilon}}}\quad \text{and} \quad |\psi_{\varepsilon_n}(x)|\leq C\sum_{l=1}^k e^{-\frac{\tau|x-\hat{x}_l|}{\hat{\varepsilon}}},
\end{align*}
 which implies that
\begin{align}\label{4147}
|\hat{\varphi}(x)|\leq C\sum_{l=1}^k e^{-\frac{\tau|x-\hat{x}_l|}{\hat{\varepsilon}}}\ \ \ \text{and}\ \ \ |\hat{\psi}(x)|\leq C\sum_{l=1}^k e^{-\frac{\tau|x-\hat{x}_l|}{\hat{\varepsilon}}}.
\end{align}
By combining \eqref{4141} and \eqref{eq41}-\eqref{4147}, and invoking Theorem \ref{main}, we conclude that $(\hat{u},\hat{v})=(u_{\hat{\varepsilon}},v_{\hat{\varepsilon}})$.  This leads to a contradiction.
\end{proof}

\section{Existence and uniqueness of normalized solutions}
In this section, we proceed to prove Theorem \ref{pro11} and Theorem \ref{main1}.
\subsection{Existence of normalized solutions}
\subsubsection{The non-critical case (N=3)}
To ensure that the solution obtained in Theorem \ref{pro1} satisfies \eqref{rho2}, we make the following setting:
\begin{equation*}
(\tilde{u}_\varepsilon, \tilde{v}_\varepsilon):=\left(\frac{\varepsilon \rho u_\varepsilon}{\left(||u_{\varepsilon}||_{L^2(\mathbb{R}^3)}^2+ ||v_{\varepsilon}||_{L^2(\mathbb{R}^3)}^2\right)^{\frac12}}, \frac{\varepsilon \rho v_\varepsilon}{\left(||u_{\varepsilon}||_{L^2(\mathbb{R}^3)}^2+||v_{\varepsilon}||_{L^2(\mathbb{R}^3)}^2\right)^{\frac12}}\right),
\end{equation*}
then $(\tilde{u}_\varepsilon,\tilde{v}_\varepsilon)$ satisfies the system
 \begin{equation*}
  \left\{\begin{array}{ll}
		-\varepsilon ^2 \triangle \tilde{u}_{\varepsilon}+\left( \varepsilon ^2P\left( x \right) +1 \right) \tilde{u}_{\varepsilon}=\frac{ ||u_{\varepsilon}||_{L^2(\mathbb{R}^3)}^2+ ||v_{\varepsilon}||_{L^2(\mathbb{R}^3)}^2}{\varepsilon^2 \rho^2} (\mu_1  \tilde{u}_{\varepsilon}^3+\beta \tilde{u}_{\varepsilon}\tilde{v}_{\varepsilon}^2),\,\, \text{in} \,\,\mathbb{R}^3,\\
		-\varepsilon ^2 \triangle \tilde{v}_{\varepsilon}+\left( \varepsilon ^2Q\left( x \right) +1 \right) \tilde{v}_{\varepsilon}=\frac{||u_{\varepsilon}||_{L^2(\mathbb{R}^3)}^2 +||v_{\varepsilon}||_{L^2(\mathbb{R}^3)}^2}{\varepsilon^2 \rho^2}  (\mu_2  \tilde{v}_{\varepsilon}^3+\beta \tilde{v}_{\varepsilon}\tilde{u}_{\varepsilon}^2),\,\, \text{in} \,\,\mathbb{R}^3,\\
\varepsilon^{-2}\int_{\mathbb{R}^3} (\tilde{u}_{\varepsilon}^2+\tilde{v}_{\varepsilon}^2)=\rho^2.\\
	\end{array} \right.
\end{equation*}
Our goal is to prove that there exists $\varepsilon$, such that
\begin{equation*}
\frac{||u_{\varepsilon}||_{L^2(\mathbb{R}^3)}^2+||v_{\varepsilon}||_{L^2(\mathbb{R}^3)}^2} {\varepsilon^2 \rho^2}=1
\end{equation*}
 Denote
\begin{equation*}
F(\varepsilon):=\frac{||u_{\varepsilon}||_{L^2(\mathbb{R}^3)}^2+||v_{\varepsilon}||_{L^2(\mathbb{R}^3)}^2} {\varepsilon^2 \rho^2}.
\end{equation*}
By Corollary \ref{con1}, $F(\varepsilon)$ is a continuous function in $(0,\varepsilon_0)$. Through the previous estimates, we obtain that
\begin{align}\label{Fvarepslion}
F(\varepsilon)&=\varepsilon^{-2}\rho^{-2}\int_{\mathbb{R}^3}(u_\varepsilon^2+v_\varepsilon^2)\notag\\
&=\varepsilon^{-2}\rho^{-2}\int_{\mathbb{R}^3}\left(\left(\sum_{l=1}^k \left( w_{\varepsilon, \xi_{\varepsilon,l}}^{*} + \varepsilon^4 W_{\varepsilon, \xi_{\varepsilon,l}}^{*} \right) + \varphi_{\varepsilon}\right)^2+ \left(\sum_{l=1}^k \left( w_{\varepsilon, \xi_{\varepsilon,l}}^{\star} + \varepsilon^4 W_{\varepsilon, \xi_{\varepsilon,l}}^{\star} \right) + \psi_{\varepsilon}\right)^2 \right)\notag\\
&=\varepsilon^{-2}\rho^{-2}\int_{\mathbb{R}^3}\left(\sum_{l=1}^k\left( w_{\varepsilon, \xi_{\varepsilon,l}}^{*}\right)^2+\sum_{l=1}^k\left( w_{\varepsilon, \xi_{\varepsilon,l}}^{\star}\right)^2\right)+O(\rho^{-2}\varepsilon^5)\notag\\
&=\varepsilon \rho^{-2}\rho_0^2+O(\rho^{-2}\varepsilon^5).
\end{align}
Then we have
\begin{align*}
F\left(\frac{\rho^2}{2\rho_0^2}\right)&=\frac{1}{2}+O(\rho^{-2}\varepsilon^5)<1
\end{align*}
and
\begin{align*}
F\left(\frac{3\rho^2}{2\rho_0^2}\right)&=\frac{3}{2}+O(\rho^{-2}\varepsilon^5)>1.
\end{align*}
It follows from Corollary \ref{con1} that $F(\varepsilon)$ is continuous. By intermediate value theorem, we know that there exists $\tilde{\varepsilon}\in\left(\frac{\rho^2}{2\rho_0^2},\frac{3\rho^2}{2\rho_0^2}\right)$ such that $F(\tilde{\varepsilon})=1$.

\begin{remark}
From the expression of \( F(\varepsilon) \) in \eqref{Fvarepslion}, we note the following key observation for \( N=2 \): the leading-order estimate of \( F(\varepsilon) \) is a constant (i.e., independent of \( \varepsilon \)) at the leading order. This, in turn, precludes the possibility of comparing \( F(\varepsilon) \) with 1 by selecting a sufficiently small \( \varepsilon \). Moreover, this property further implies that for the system \eqref{eq0}-\eqref{rho2} to admit a solution, \( \rho \) must be close to a specific critical threshold.
\end{remark}
\subsubsection{The critical case (N=2)}
In this case, we set
\begin{align*}
\bar{F}(\varepsilon)&:=\varepsilon^{-2}\int_{\mathbb{R}^2} (u_{\varepsilon}^2+v_{\varepsilon}^2) = \rho_0^2 +2\varepsilon^4 \sum_{l=1}^k \int_{\mathbb{R}^2} \left(w^{*}W_l^{*}+w^{\star}W_l^{\star}\right) +O(\varepsilon^5).
\end{align*}
For convenience, we write
\begin{align*}
\sigma_1=\sqrt { \frac { \beta - \mu_2 } { \beta ^ { 2 } - \mu_1 \mu_2 } },\quad \sigma_2=\sqrt { \frac { \beta - \mu_1 } { \beta ^ { 2 } - \mu_1 \mu_2 } }.
\end{align*}
By \eqref{eqw} and \eqref{eq3}, we have
 \begin{equation*} \left\{\begin{array}{ll}
\mu_1 \sigma_1^2+\beta\sigma_2^2=1,\\
\mu_2 \sigma_2^2+\beta\sigma_1^2=1.
	\end{array} \right.
\end{equation*}
Multiplying the first equation in \eqref{eq2} by $\sigma_1$, the second by $\sigma_2$, and adding the resulting equations yields:
\begin{equation*}
-\triangle z + z - 3w^2 z =-\sum_{i=1}^{2} (\sigma_1^2 p_{li}x^{2}_{i}+ \sigma_2^2 q_{li}x^{2}_{i}) w,
\end{equation*}
where
\begin{equation*}
z=\sigma_1 W_l^{*}+\sigma_2 W_l^{\star}.
\end{equation*}

Assume that $z_{0}$ is the solution of
\begin{equation*}
-\triangle z + z - 3w^2 z =-|x|^{2}w,
\end{equation*}
then
\begin{align*}
&\quad\int_{\mathbb{R}^2} \left(w^{*}W_l^{*}+w^{\star}W_l^{\star}\right)
= \int_{\mathbb{R}^2} w z = \frac{1}{2} \sum_{i=1}^{2} (\sigma_1^2 p_{li}+ \sigma_2^2 q_{li}) \int_{\mathbb{R}^{2}} wz_{0}.
\end{align*}

Therefore,
\begin{align*}
\bar{F}(\varepsilon) = \rho_0^2 +\varepsilon^4 \sum_{l=1}^k  \sum_{i=1}^{2} (\sigma_1^2 p_{li}+ \sigma_2^2 q_{li}) \int_{\mathbb{R}^{2}} wz_{0}+O(\varepsilon^5).
\end{align*}

As mentioned by the work of \cite{huang2025normalized}, we know that
\[\int_{\mathbb{R}^{2}} wz_{0} <0. \]

If $\sum_{l=1}^k  \sum_{i=1}^{2} (\sigma_1^2 p_{li}+ \sigma_2^2 q_{li})>0$, for $\rho^{2} \in (\rho_{0}^{2}-\tilde\delta, \rho_{0}^{2})$, there exists $ \tilde\varepsilon$, such that $\bar{F}(\tilde\varepsilon)= \rho^{2}$.

If $\sum_{l=1}^k  \sum_{i=1}^{2} (\sigma_1^2 p_{li}+ \sigma_2^2 q_{li})< 0$, for $\rho^{2} \in (\rho_{0}^{2}, \rho_{0}^{2}+\tilde\delta)$, there exists $ \tilde\varepsilon$, such that $\bar{F}(\tilde\varepsilon)= \rho^{2}$.

To establish the uniqueness of the solution, it is necessary to derive a more detailed estimate of \( \lambda_\rho \) for \( N=2 \). To this end, we firstly state two integral identities.
\begin{lemma}\label{lemww}
  It holds that
  \begin{align*}
  2\int_{\mathbb{R}^2} \Big( (w^{*})^2 +(w^{\star})^2  \Big) = \int_{\mathbb{R}^2} \Big( \mu_1(w^{*})^4 +\mu_2 (w^{\star})^4 +2\beta (w^{*})^2 (w^{\star})^2 \Big),
  \end{align*}
  and
    \begin{align*}
  \int_{\mathbb{R}^2} \Big( |\nabla w^{*}|^2 +|\nabla w^{\star}|^2  \Big) = \int_{\mathbb{R}^2} \Big( (w^{*})^2 +(w^{\star})^2  \Big) .
  \end{align*}
\end{lemma}
\begin{proof}
Multiplying the first equation of system \eqref{eq3} by \( w^* \), the second equation by \( w^\star \), integrating each product over \( \mathbb{R}^N \), and summing the resulting equations,  this yields
  \begin{align*}
  \int_{\mathbb{R}^N} \Big(  |\nabla w^{*}|^2 +|\nabla w^{\star}|^2+(w^{*})^2 +(w^{\star})^2  \Big) = \int_{\mathbb{R}^N} \Big( \mu_1(w^{*})^4 +\mu_2 (w^{\star})^4 +2\beta (w^{*})^2 (w^{\star})^2 \Big).
  \end{align*}
The second integral relation is obtained by: multiplying the first equation of system \eqref{eq3} by \( x \cdot \nabla w^* \), multiplying the second equation by \( x \cdot \nabla w^\star \), integrating each product over \( \mathbb{R}^N \), summing the two integrated equations. This gives:
  \begin{align*}
  (N-2) \int_{\mathbb{R}^2} \Big(  |\nabla w^{*}|^2 +|\nabla w^{\star}|^2\Big) + N\int_{\mathbb{R}^2} \Big( (w^{*})^2 +(w^{\star})^2  \Big) =\frac N 2 \int_{\mathbb{R}^2} \Big( \mu_1(w^{*})^4 +\mu_2 (w^{\star})^4 +2\beta (w^{*})^2 (w^{\star})^2 \Big).
  \end{align*}

  This result follows directly from these two integral relations (with \( N=2 \)).
\end{proof}

In addition, we have the following key lemma.

\begin{lemma}\label{lemwW}
It holds that
\begin{align}\label{eqwW}
\sum_{l=1}^k\int_{\mathbb{R}^2} \bigl( \nabla w^{*} \cdot \nabla W_{l}^{*}+\nabla w^{\star} \cdot \nabla W_{l}^{\star} +  w^{*} W_{l}^{*}+ w^{\star} W_{l}^{\star}   \bigr) = \frac14 \sum_{l=1}^k \sum_{i=1}^{2}\big(\sigma_1^2 p_{li} + \sigma_2^2  q_{li}\big) \int_{\mathbb{R}^2} |x|^2 w^2 .
\end{align}
\end{lemma}

\begin{proof}
For each $l = 1, \dots, k$, define
\[
I_l := \int_{\mathbb{R}^2} \left( \nabla w^* \cdot \nabla W_l^* + \nabla w^\star \cdot \nabla W_l^\star + w^* W_l^* + w^\star W_l^\star \right) dx.
\]
Integration by parts gives
\[
I_l = \int_{\mathbb{R}^2} w^* (-\Delta W_l^* + W_l^*) dx + \int_{\mathbb{R}^2} w^\star (-\Delta W_l^\star + W_l^\star) dx.
\]
Substituting the system \eqref{eq2}, we obtain
\[
I_l = J_l - \int_{\mathbb{R}^2} \sum_{i=1}^2 p_{li} x_i^2 (w^*)^2 dx - \int_{\mathbb{R}^2} \sum_{i=1}^2 q_{li} x_i^2 (w^\star)^2 dx,
\]
where
\[
\begin{aligned}
J_l &= \int_{\mathbb{R}^2} w^* \left[ 3\mu_1 (w^*)^2 W_l^* + \beta (w^\star)^2 W_l^* + 2\beta w^* w^\star W_l^\star \right] dx \\
&\quad + \int_{\mathbb{R}^2} w^\star \left[ 3\mu_2 (w^\star)^2 W_l^\star + \beta (w^*)^2 W_l^\star + 2\beta w^* w^\star W_l^* \right] dx\\
&=\int_{\mathbb{R}^2} w^* \left[ 3\mu_1 (w^*)^2 W_l^* + 3\beta (w^\star)^2 W_l^*  \right] dx
+ \int_{\mathbb{R}^2} w^\star \left[ 3\mu_2 (w^\star)^2 W_l^\star +3 \beta (w^*)^2 W_l^\star \right] dx.
\end{aligned}
\]
Using the system \eqref{eq3}, we have
\[
w^*(\mu_1 (w^*)^2 + \beta (w^\star)^2) = -\Delta w^* + w^*, \quad w^\star(\mu_2 (w^\star)^2 + \beta (w^*)^2) = -\Delta w^\star + w^\star
\]
and
\[
J_l = 3\int_{\mathbb{R}^2} (-\Delta w^* + w^*) W_l^* dx + 3\int_{\mathbb{R}^2} (-\Delta w^\star + w^\star) W_l^\star dx = 3I_l.
\]
Thus,
\[
I_l = 3I_l - \int_{\mathbb{R}^2} \sum_{i=1}^2 p_{li} x_i^2 (w^*)^2 dx - \int_{\mathbb{R}^2} \sum_{i=1}^2 q_{li} x_i^2 (w^\star)^2 dx,
\]
which implies
\[
2I_l = \sigma_1^2 \int_{\mathbb{R}^2} \sum_{i=1}^2 p_{li} x_i^2 w^2 dx + \sigma_2^2 \int_{\mathbb{R}^2} \sum_{i=1}^2 q_{li} x_i^2 w^2 dx.
\]
By radial symmetry of $w$, $\int_{\mathbb{R}^2} x_1^2 w^2 dx = \int_{\mathbb{R}^2} x_2^2 w^2 dx = \frac{1}{2}\int_{\mathbb{R}^2} |x|^2 w^2 dx$, so
\[
I_l = \frac{1}{4} \left( \sigma_1^2 \sum_{i=1}^2 p_{li} + \sigma_2^2 \sum_{i=1}^2 q_{li} \right) \int_{\mathbb{R}^2} |x|^2 w^2 dx,
\]
and therefore, we get \eqref{eqwW}.
\end{proof}
By Lemma \ref{lemww} and Lemma \ref{lemwW}, we derive a more refined estimate of \( \lambda_\rho \).

\begin{lemma}\label{lem37}
For \( N=2 \), it follows that
\begin{align}\label{32-0}
 \lambda_\rho \Big(\rho^2 - \rho_0^2\Big)
  =& A  \lambda_\rho^{-1}+C \lambda_\rho^{-2} +O(\lambda_\rho^{-3}),
\end{align}
where
\begin{align}\label{A}
A:= -\frac34 \sum_{l=1}^k \sum_{i=1}^{2}\big(\sigma_1^2 p_{li} + \sigma_2^2  q_{li}\big) \int_{\mathbb{R}^2} |x|^2 w^2.
\end{align}
\end{lemma}
\begin{proof}
Let \( (\bar{u}_{\lambda_{\rho}}, \bar{v}_{\lambda_{\rho}}) \) denote a solution to \eqref{eq0}. For each \( l = 1, \dots, k \), multiplying the first equation in \eqref{eq0} by \( \langle x - \xi_{\rho,l}, \nabla\bar{u}_{\lambda_{\rho}} \rangle \), the second equation by \( \langle x - \xi_{\rho,l}, \nabla\bar{v}_{\lambda_{\rho}} \rangle \), summing these two products, and integrating the resulting equality over \( B_{\delta}(\xi_{\rho,l}) \), we then sum the equations obtained via this procedure over \( l = 1, \dots, k \),
\begin{align}\label{32-1}
&\sum_{l=1}^k \int_{B_{\delta} (\xi_{\rho,l} )} \Bigg[\left((P(x)+\lambda_{\rho})+\frac12\langle \nabla P(x),x - \xi_{\rho,l} \rangle \right)(\bar{u}_{\lambda_{\rho}} )^2- \frac12 \mu_1 (\bar{u}_{\lambda_{\rho}} )^4 - \beta (\bar{u}_{\lambda_{\rho}} )^2 (\bar{v}_{\lambda_{\rho}} )^2 \notag\\
&\quad\quad\quad+\left((Q(x)+\lambda_{\rho})+\frac12\langle \nabla Q(x),x - \xi_{\rho,l} \rangle \right)(\bar{v}_{\lambda_{\rho}} )^2- \frac12 \mu_2 (\bar{v}_{\lambda_{\rho}} )^4 \Bigg] \notag\\
=&\sum_{l=1}^k \int_{\partial B_{\delta} (\xi_{\rho,l} )} \Bigg[-\frac{\partial \bar{u}_{\lambda_{\rho}} }{\partial n}\langle x - \xi_{\rho,l} , \nabla \bar{u}_{\lambda_{\rho}}  \rangle -\frac{\partial \bar{v}_{\lambda_{\rho}} }{\partial n}\langle x - \xi_{\rho,l} , \nabla \bar{v}_{\lambda_{\rho}}  \rangle+\frac12\Bigg(|\nabla \bar{u}_{\lambda_{\rho}} |^2+ (P(x)+\lambda_{\rho})(\bar{u}_{\lambda_{\rho}} )^2 \notag\\
&\quad\quad\quad-\frac{\mu_1}{2}( \bar{u}_{\lambda_{\rho}} )^4 +|\nabla \bar{v}_{\lambda_{\rho}} |^2+ (Q(x)+\lambda_{\rho})(\bar{v}_{\lambda_{\rho}} )^2-\frac{\mu_2}{2}( \bar{v}_{\lambda_{\rho}} )^4 - \beta ( \bar{u}_{\lambda_{\rho}} )^2 ( \bar{v}_{\lambda_{\rho}} )^2\Bigg)\langle x - \xi_{\rho,l} , n\rangle\Bigg]dS,
\end{align}
We now carry out a more precise calculation as follows
\begin{align}\label{32-2}
&\sum_{l=1}^k \int_{B_{\delta} (\xi_{\rho,l} )} \Big(\frac12 \mu_1 (\bar{u}_{\lambda_{\rho}} )^4 + \beta (\bar{u}_{\lambda_{\rho}} )^2 (\bar{v}_{\lambda_{\rho}} )^2 + \frac12 \mu_2 (\bar{v}_{\lambda_{\rho}} )^4 \Big) \notag\\
=&\lambda_{\rho}^2 \sum_{l=1}^k\int_{B_{\delta} (\xi_{\rho,l} )} \Big(\frac12 \mu_1 \big(w_{\lambda_{\rho},\xi_{\rho,l}}^{*}
      +\lambda_{\rho}^{-2}W_{\lambda_{\rho},\xi_{\rho,l}}^{*}\big)^4 + \beta \big(w_{\lambda_{\rho},\xi_{\rho,l}}^{*}
      +\lambda_{\rho}^{-2}W_{\lambda_{\rho},\xi_{\rho,l}}^{*}\big)^2 \big(w_{\lambda_{\rho},\xi_{\rho,l}}^{\star}
      +\lambda_{\rho}^{-2}W_{\lambda_{\rho},\xi_{\rho,l}}^{\star}\big)^2 \notag\\
      &\qquad\qquad+ \frac12 \mu_2
\big(w_{\lambda_{\rho},\xi_{\rho,l}}^{\star}
      +\lambda_{\rho}^{-2}W_{\lambda_{\rho},\xi_{\rho,l}}^{\star}\big)^4 \Big)+O(\lambda_\rho^{-\frac{7}{2}})  \notag\\
=&\lambda_\rho k \int_{\mathbb{R}^2} \Big( (w^{*})^2 +(w^{\star})^2  \Big) \notag\\
&+\lambda_\rho^{-1}\sum_{l=1}^k \int_{\mathbb{R}^2}  \Big(2 \mu_1 (w ^{*})^3 W_l^{*} + 2 \beta (w ^{*}  (w^{\star})^2 W_l^{*} +  w ^{\star}  (w^{*})^2 W_l^{\star} )+ 2 \mu_2 (w ^{\star})^3 W_l^{\star}  \Big)\notag\\
& +\lambda_\rho^{-3} \sum_{l=1}^k\int_{\mathbb{R}^2} \Big(3 \mu_1 (w ^{*})^2 (W_l^*)^2 + \beta \big((w ^{*})^2 (W_l^\star)^2+ (w^{\star})^2 (W_l^*)^2 + 4 w ^{*}w^{\star}W_l^* W_l^\star\big)+ 3 \mu_2 (w ^{\star})^2 (W_l^\star)^2\Big) \notag\\
& + O(\lambda_\rho^{-\frac{7}{2}}) \notag\\
=& \lambda_\rho \rho_0^2 +\lambda_{\rho}^{-1} \sum_{l=1}^k \int_{\mathbb{R}^{2}} \bigl( \nabla w^{*} \nabla W_{l}^{*}+\nabla w^{\star} \nabla W_{l}^{\star} +  w^{*} W_{l}^{*}+ w^{\star} W_{l}^{\star}   \bigr)+ C_1 \lambda_\rho^{-3} +O(\lambda_\rho^{-\frac{7}{2}}),
\end{align}
and, together with \eqref{rho2}, it follows that
\begin{align}\label{32-3}
&\sum_{l=1}^k \int_{B_{\delta} (\xi_{\rho,l} )} \Big[(P(x)+\frac12\langle \nabla P(x),x - \xi_{\rho,l} \rangle )(\bar{u}_{\lambda_{\rho}} )^2+((Q(x)+\frac12\langle \nabla Q(x),x - \xi_{\rho,l} \rangle )(\bar{v}_{\lambda_{\rho}} )^2\Big] \notag\\
=&\lambda_\rho^{-1} \sum_{l=1}^k\sum_{i=1}^{2} \Big(\sigma_1^2 p_{li} + \sigma_2^2 q_{li}\Big) \int_{\mathbb{R}^2} |x|^2 w^2 + C_3 \lambda_\rho^{-2} +O(\lambda_\rho^{-3}).
\end{align}
By substituting \eqref{32-2} and \eqref{32-3} into \eqref{32-1}, it follows that \eqref{32-0} holds.
\end{proof}

Moreover, we establish the specific relationship between the Lagrange multiplier \(\lambda_\rho\) and \(\rho\).

\begin{lemma}\label{lem51}
It follows that
\begin{align*}
\lambda_{\rho} =
\begin{cases}
\rho_0^4 \rho^{-4}+O(\rho^{-2}), &N = 3, \\
\left(\frac{A}{\rho^2 - \rho_0^2}\right)^{\frac12} \left( 1 + \frac { C } { 2 A  } \left(\frac{\rho^2 - \rho_0^2}{A}\right)^{\frac12}  - \frac { 3 C ^ { 2 } } { 8 A ^ { 3 } } (\rho^2 - \rho_0^2) + O \Big( \Big(\frac{\rho^2 - \rho_0^2}{A}\Big)  ^ { \frac32 }\Big) \right), &N = 2,
\end{cases}
\end{align*}
where the constant $A$ is the same as that of \eqref{A}.
\end{lemma}
\begin{proof}
When $N=3$, similar to \eqref{Fvarepslion}, we conclude that
\begin{align*}
\rho^2=\int_{\mathbb{R}^3}(\bar{u}_{\lambda_{\rho}}^2+\bar{v}_{\lambda_{\rho}}^2)=\lambda_{\rho}^{-\frac{1}{2}} \rho_0^2 +O(\rho^{-2}\lambda_{\rho}^{-\frac{3}{2}}).
\end{align*}
Moreover, combining with $\lambda_{\rho}^{-1/2}\in\left(\frac12 \rho_0^{-2} \rho^2 ,\frac32 \rho_0^{-2} \rho^2 \right)$, we get
\begin{align*}
\lambda_{\rho}=\rho_0^4 \rho^{-4}+O(\rho^{-2}).
\end{align*}
When $N=2$, by Lemma \ref{lem37}, we have
\begin{align}\label{38-0}
\rho^2 - \rho_0^2
  =& A \lambda_\rho^{-2} +C \lambda_\rho^{-3} +O(\lambda_\rho^{-4}).
\end{align}
Let \( t = \lambda_\rho^{-1} \), so that \( \lambda_\rho = t^{-1} \). Substituting this into \eqref{38-0}, we obtain
\begin{align}\label{38-1}
\frac{\rho^2 - \rho_0^2}{A}=  t^2 +\frac{C}{A} t^3 +O(t^4).
\end{align}
By Puiseux' Theorem (Theorem 2.1.1 in \cite{wall2004singular}),  we can expand \( t \) as a power series in \( \left(\frac{\rho^2 - \rho_0^2}{A}\right)^{\frac12} \):
\begin{align}\label{38-2}
t=\alpha_1 \left(\frac{\rho^2 - \rho_0^2}{A}\right)^{\frac12}+\alpha_2 \left(\frac{\rho^2 - \rho_0^2}{A}\right)+\alpha_3 \left(\frac{\rho^2 - \rho_0^2}{A}\right)^{\frac32}+\alpha_4 \left(\frac{\rho^2 - \rho_0^2}{A}\right)^2+\cdots.
\end{align}
Substituting \eqref{38-2} into \eqref{38-1} and equating the coefficients of like powers, we can conclude that
\begin{align}\label{38-3}
t= \left(\frac{\rho^2 - \rho_0^2}{A}\right)^{\frac12}\Big(1-\frac{C}{2A} \left(\frac{\rho^2 - \rho_0^2}{A}\right)^{\frac12}+\frac{5C^2}{8A^3} (\rho^2 - \rho_0^2)-\frac{C^3}{A^{3}} \left(\frac{\rho^2 - \rho_0^2}{A}\right)^{\frac32}+O((\rho^2 - \rho_0^2)^2)\Big).
\end{align}
Observing that \( \lambda_\rho = t^{-1} \), we use the binomial expansion
\[
(1 + D)^{-1} = 1 - D + D^2 - D^3 + \cdots,
\]
where
\[
D = -\frac{C}{2A} \left(\frac{\rho^2 - \rho_0^2}{A}\right)^{\frac12} + \frac{5C^2}{8A^3} (\rho^2 - \rho_0^2) + O\left(\left(\frac{\rho^2 - \rho_0^2}{A}\right)^{\frac32}\right),
\]
then we get
\begin{align}\label{38-4}
&\left(1-\frac{C}{2A} \left(\frac{\rho^2 - \rho_0^2}{A}\right)^{\frac12} + \frac{5C^2}{8A^3} (\rho^2 - \rho_0^2) + O\left(\left(\frac{\rho^2 - \rho_0^2}{A}\right)^{\frac32}\right)\right)^{-1}\notag\\
= & 1 + \frac { C } { 2 A  } \left(\frac{\rho^2 - \rho_0^2}{A}\right) ^ { \frac12 } - \frac { 3 C ^ { 2 } } { 8 A ^ { 3 } } (\rho^2 - \rho_0^2) + O \left( \left(\frac{\rho^2 - \rho_0^2}{A}\right) ^ {\frac32 } \right).
\end{align}
By combining \eqref{38-3} with \eqref{38-4}, we have
\begin{align*}
\lambda _ { \rho } = \left(\frac{A}{\rho^2 - \rho_0^2}\right)^{\frac12} \left( 1 + \frac { C } { 2 A  } \left(\frac{\rho^2 - \rho_0^2}{A}\right)^{\frac12}  - \frac { 3 C ^ { 2 } } { 8 A ^ { 3 } } (\rho^2 - \rho_0^2) + O \Big( \Big(\frac{\rho^2 - \rho_0^2}{A}\Big)  ^ { \frac32 }\Big) \right)  .
\end{align*}
\end{proof}

\section{Uniqueness of normalized multi-peak solutions}

In this section, we aim to prove Theorem \ref{main1} via a contradiction argument. By Theorem \ref{pro11}, we assume there exist two distinct normalized solutions \( (\bar{u}_{\lambda_{1,\rho}}^{(1)}, \bar{v}_{\lambda_{1,\rho}}^{(1)}) \) and \( (\bar{u}_{\lambda_{2,\rho}}^{(2)}, \bar{v}_{\lambda_{2,\rho}}^{(2)}) \) of the following form
\begin{align*}
(\bar{u}_{\lambda_{1,\rho}}^{(1)},\bar{v}_{\lambda_{1,\rho}}^{(1)})=&\sqrt{\lambda_{1,\rho}}\left( \sum_{l=1}^k \left( w_{\lambda_{1,\rho}, \xi_{\rho,l}^{(1)}}^{*} + \lambda_{1,\rho}^{-2} W_{\lambda_{1,\rho}, \xi_{\rho,l}^{(1)}}^{*} \right) + \varphi_{\rho}^{(1)},\sum_{l=1}^k \left( w_{\lambda_{1,\rho}, \xi_{\rho,l}^{(1)}}^{\star} + \lambda_{1,\rho}^{-2} W_{\lambda_{1,\rho}, \xi_{\rho,l}^{(1)}}^{\star} \right) + \psi_{\rho}^{(1)} \right)
 \end{align*}
and
 \begin{align*}
 (\bar{u}_{\lambda_{2,\rho}}^{(2)},\bar{v}_{\lambda_{2,\rho}}^{(2)})=&\sqrt{\lambda_{2,\rho}}\left( \sum_{l=1}^k \left( w_{\lambda_{2,\rho}, \xi_{\rho,l}^{(2)}}^{*} + \lambda_{2,\rho}^{-2} W_{\lambda_{2,\rho}, \xi_{\rho,l}^{(2)}}^{*} \right) + \varphi_{\rho}^{(2)},\sum_{l=1}^k \left( w_{\lambda_{2,\rho}, \xi_{\rho,l}^{(2)}}^{\star} + \lambda_{2,\rho}^{-2} W_{\lambda_{2,\rho}, \xi_{\rho,l}^{(2)}}^{\star} \right) + \psi_{\rho}^{(2)} \right),
\end{align*}
where
 \begin{align*}
 w_{\lambda_{i,\rho},\xi_{\rho,l}^{(i)}}^{*}(x) :=w^{*}(\sqrt{\lambda_{i,\rho}}(x-\xi_{\rho,l}^{(i)})), \quad W_{\lambda_{i,\rho},\xi_{\rho,l}^{(i)}}^{*}(x) :=W^{*}_l(\sqrt{\lambda_{i,\rho}}(x-\xi_{\rho,l}^{(i)})),\quad i=1,2,\quad l=1,2,\cdots,k
\end{align*}
and
 \begin{align*}
 w_{\lambda_{i,\rho},\xi_{\rho,l}^{(i)}}^{\star}(x) :=w^{\star}(\sqrt{\lambda_{i,\rho}}(x-\xi_{\rho,l}^{(i)})), \quad W_{\lambda_{i,\rho},\xi_{\rho,l}^{(i)}}^{\star}(x) :=W^{\star}_l(\sqrt{\lambda_{i,\rho}}(x-\xi_{\rho,l}^{(i)})),\quad i=1,2\quad l=1,2,\cdots,k.
\end{align*}

  We write
\begin{align*}
(\kappa_{\rho}^{(1)}, \kappa_{\rho}^{(2)}) =\left(\frac{\bar{u}^{(1)}_{\lambda_{1,\rho}}-\bar{u}^{(2)}_{\lambda_{2,\rho}}}{||\bar{u}^{(1)}_{\lambda_{1,\rho}} -\bar{u}^{(2)}_{\lambda_{2,\rho}}||_{L^{\infty}(\mathbb{R}^N)}+ ||\bar{v}^{(1)}_{\lambda_{1,\rho}}-\bar{v}^{(2)}_{\lambda_{2,\rho}}||_{L^{\infty}(\mathbb{R}^N)}}, \frac{\bar{v}^{(1)}_{\lambda_{1,\rho}}-\bar{v}^{(2)}_{\lambda_{2,\rho}}}{||\bar{u}^{(1)}_{\lambda_{1,\rho}}-\bar{u}^{(2)}_{\lambda_{2,\rho}}||_{L^{\infty}(\mathbb{R}^N)}+ ||\bar{v}^{(1)}_{\lambda_{1,\rho}}-\bar{v}^{(2)}_{\lambda_{2,\rho}}||_{L^{\infty}(\mathbb{R}^N)}}\right).
\end{align*}
According to system \eqref{eq0}, $\kappa_{\rho}^{(1)} $ and $\kappa_{\rho}^{(2)} $ satisfy
 \begin{equation} \label{eq-} \left\{\begin{array}{ll}
		- \triangle \kappa_{\rho}^{(1)} +(P\left( x \right) + \lambda_{1,\rho} )\kappa_{\rho}^{(1)} =\frac{(\lambda_{2,\rho}-\lambda_{1,\rho})\bar{u}_{\lambda_{2,\rho}}^{(2)}}{||\bar{u}^{(1)}_{\lambda_{1,\rho}} -\bar{u}^{(2)}_{\lambda_{2,\rho}}||_{L^{\infty}(\mathbb{R}^N)} +||\bar{v}^{(1)}_{\lambda_{1,\rho}}-\bar{v}^{(2)}_{\lambda_{2,\rho}}||_{L^{\infty}(\mathbb{R}^N)}}+ \mu_1  ((\bar{u}^{(1)}_{\lambda_{1,\rho}})^2+\bar{u}^{(1)}_{\lambda_{1,\rho}}\bar{u}^{(2)}_{\lambda_{2,\rho}}+(\bar{u}^{(2)}_{\lambda_{2,\rho}})^2)\kappa_{\rho}^{(1)} \vspace{0.12cm}\\
	\ \ \ \ 	\ \ \ \ \ \ \ \ \ \ \ \ \ \ \ \ \ \ \ \ \ \ \ \ \ \ \ \ \ \ \ \ \ \ \ \ \ \  +\beta\left[\kappa_{\rho}^{(1)} (\bar{v}^{(2)}_{\lambda_{2,\rho}})^2+\bar{u}^{(1)}_{\lambda_{1,\rho}}\kappa_{\rho}^{(2)} (\bar{v}^{(1)}_{\lambda_{1,\rho}}+\bar{v}^{(2)}_{\lambda_{2,\rho}})\right],\,\,\,\,x\in \mathbb{R}^N, \\
		- \triangle \kappa_{\rho}^{(2)} +(Q\left( x \right) + \lambda_{1,\rho} )\kappa_{\rho}^{(2)} =\frac{(\lambda_{2,\rho}-\lambda_{1,\rho})\bar{v}_{\lambda_{2,\rho}}^{(2)}}{||\bar{u}^{(1)}_{\lambda_{1,\rho}}-\bar{u}^{(2)}_{\lambda_{2,\rho}}||_{L^{\infty}(\mathbb{R}^N)}+||\bar{v}^{(1)}_{\lambda_{1,\rho}}-\bar{v}^{(2)}_{\lambda_{2,\rho}}||_{L^{\infty}(\mathbb{R}^N)}}+ \mu_2  ((\bar{v}^{(1)}_{\lambda_{1,\rho}})^2+\bar{v}^{(1)}_{\lambda_{1,\rho}}\bar{v}^{(2)}_{\lambda_{2,\rho}}+(\bar{v}^{(2)}_{\lambda_{2,\rho}})^2)\kappa_{\rho}^{(2)} \vspace{0.12cm} \\
	\ \ \ \ 	\ \ \ \ \ \ \ \ \ \ \ \ \ \ \ \ \ \ \ \ \ \ \ \ \ \ \ \ \ \ \ \ \ \ \ \ \ \  +\beta\left[\kappa_{\rho}^{(2)} (\bar{u}^{(2)}_{\lambda_{2,\rho}})^2+\bar{v}^{(1)}_{\lambda_{1,\rho}}\kappa_{\rho}^{(1)} (\bar{u}^{(1)}_{\lambda_{1,\rho}}+\bar{u}^{(2)}_{\lambda_{2,\rho}})\right],\,\,\,\,x\in \mathbb{R}^N. \\
	\end{array} \right.
\end{equation}
Moreover, by local Pohozaev identity, we have
\begin{align}\label{poho2}
&\int_{\Omega}\left(\frac{\partial P(x)}{\partial x_j}(\bar{u}_{\lambda_{1,\rho}}^{(1)}+\bar{u}_{\lambda_{2,\rho}}^{(2)})\kappa_{\rho}^{(1)} +\frac{\partial Q(x)}{\partial x_j}(\bar{v}_{\lambda_{1,\rho}}^{(1)}+\bar{v}_{\lambda_{2,\rho}}^{(2)})\kappa_{\rho}^{(2)}  \right) \notag\\
=&-2\int_{\partial\Omega}\left(\frac{\partial \bar{u}_{\lambda_{1,\rho}}^{(1)}}{\partial x_j}\frac{\partial\kappa_{\rho}^{(1)} }{\partial n}+\frac{\partial \bar{u}_{\lambda_{2,\rho}}^{(2)}}{\partial n}\frac{\partial\kappa_{\rho}^{(1)} }{\partial x_j}+\frac{\partial \bar{v}_{\lambda_{1,\rho}}^{(1)}}{\partial x_j}\frac{\partial\kappa_{\rho}^{(2)} }{\partial n}+\frac{\partial \bar{v}_{\lambda_{2,\rho}}^{(2)}}{\partial n}\frac{\partial\kappa_{\rho}^{(2)} }{\partial x_j}\right)dS \notag\\
&+\int_{\partial\Omega}\left(\nabla(\bar{u}_{\lambda_{1,\rho}}^{(1)}+\bar{u}_{\lambda_{2,\rho}}^{(2)})\nabla\kappa_{\rho}^{(1)}  n_j +\nabla(\bar{v}_{\lambda_{1,\rho}}^{(1)}+\bar{v}_{\lambda_{2,\rho}}^{(2)})\nabla\kappa_{\rho}^{(2)}  n_j\right)dS \notag\\
&+\int_{\partial\Omega}\left((P(x)+\lambda_{1,\rho})(\bar{u}_{\lambda_{1,\rho}}^{(1)}+\bar{u}_{\lambda_{2,\rho}}^{(2)})\kappa_{\rho}^{(1)} n_j+(Q(x)+\lambda_{1,\rho})(\bar{v}_{\lambda_{1,\rho}}^{(1)}+\bar{v}_{\lambda_{2,\rho}}^{(2)})\kappa_{\rho}^{(2)} n_j\right)dS \notag\\
&+\int_{\partial\Omega}\left(\frac{(\lambda_{1,\rho}-\lambda_{2,\rho})((\bar{u}_{\lambda_{2,\rho}}^{(2)})^2 +(\bar{v}_{\lambda_{2,\rho}}^{(2)})^2) n_j}{||\bar{u}^{(1)}_{\lambda_{1,\rho}}-\bar{u}^{(2)}_{\lambda_{2,\rho}}||_{L^{\infty}(\mathbb{R}^N)}+||\bar{v}^{(1)}_{\lambda_{1,\rho}}-\bar{v}^{(2)}_{\lambda_{2,\rho}}||_{L^{\infty}(\mathbb{R}^N)}}\right)dS \notag\\
&-\frac{1}{2}\int_{\partial\Omega}\left(\mu_1((\bar{u}^{(1)}_{\lambda_{1,\rho}})^3+\bar{u}^{(1)}_{\lambda_{1,\rho}}(\bar{u}^{(2)}_{\lambda_{2,\rho}})^2+(\bar{u}^{(1)}_{\lambda_{1,\rho}})^2\bar{u}^{(2)}_{\lambda_{2,\rho}} +(\bar{u}^{(2)}_{\lambda_{2,\rho}} )^3)\kappa_{\rho}^{(1)} n_j \right)dS \notag\\
&-\frac{1}{2}\int_{\partial\Omega}\left( \mu_2((\bar{v}^{(1)}_{\lambda_{1,\rho}})^3+\bar{v}^{(1)}_{\lambda_{1,\rho}}(\bar{v}^{(2)}_{\lambda_{2,\rho}})^2+(\bar{v}^{(1)}_{\lambda_{1,\rho}})^2\bar{v}^{(2)}_{\lambda_{2,\rho}} +(\bar{v}^{(2)}_{\lambda_{2,\rho}} )^3)\kappa_{\rho}^{(2)} n_j\right)dS.
\end{align}

Denoting
\begin{align*}
(\tilde{\kappa}^{(1)}_{l,\rho} (x), \tilde{\kappa}^{(2)}_{l,\rho} (x)):=\left(\kappa_{\rho}^{(1)} (\frac{x}{\sqrt{\lambda_{1,\rho}}}+\xi_{\rho,l}^{(1)}),  \kappa_{\rho}^{(2)} (\frac{x}{\sqrt{\lambda_{1,\rho}}}+\xi_{\rho,l}^{(1)})\right), \quad l=1,2,\cdots,k,
\end{align*}
then
 \begin{equation*}
  \left\{\begin{array}{ll}
- \Delta_x \tilde{\kappa}_{l,\rho}^{(1)}(x) + \Big(\frac{1}{\lambda_{1,\rho}} P( \frac{x}{\sqrt{\lambda_{1,\rho}}} + \xi_{\rho,l}^{(1)} ) + 1 \Big) \tilde{\kappa}_{l,\rho}^{(1)}(x)
= \frac{(\frac{\lambda_{2,\rho}}{\lambda_{1,\rho}} - 1) \bar{u}_{\lambda_{2,\rho}}^{(2)}( \frac{x}{\sqrt{\lambda_{1,\rho}}} + \xi_{\rho,l}^{(1)} )}{\|\bar{u}_{\lambda_{1,\rho}}^{(1)} - \bar{u}_{\lambda_{2,\rho}}^{(2)}\|_{L^\infty(\mathbb{R}^N)} + \|\bar{v}_{\lambda_{1,\rho}}^{(1)} - \bar{v}_{\lambda_{2,\rho}}^{(2)}\|_{L^\infty(\mathbb{R}^N)}} \\
\quad + \frac{\mu_1}{\lambda_{1,\rho}} \Big( \big(\bar{u}_{\lambda_{1,\rho}}^{(1)}( \frac{x}{\sqrt{\lambda_{1,\rho}}} + \xi_{\rho,l}^{(1)}) \big)^2 + \bar{u}_{\lambda_{1,\rho}}^{(1)}( \frac{x}{\sqrt{\lambda_{1,\rho}}} + \xi_{\rho,l}^{(1)} ) \bar{u}_{\lambda_{2,\rho}}^{(2)}( \frac{x}{\sqrt{\lambda_{1,\rho}}} + \xi_{\rho,l}^{(1)} ) + \big(\bar{u}_{\lambda_{2,\rho}}^{(2)}( \frac{x}{\sqrt{\lambda_{1,\rho}}} + \xi_{\rho,l}^{(1)}) \big)^2 \Big) \tilde{\kappa}_{l,\rho}^{(1)}(x) \\
 \quad + \frac{\beta}{\lambda_{1,\rho}} \Big( \tilde{\kappa}_{l,\rho}^{(1)}(x) \big(\bar{v}_{\lambda_{2,\rho}}^{(2)}( \frac{x}{\sqrt{\lambda_{1,\rho}}} + \xi_{\rho,l}^{(1)} ) \big)^2 + \bar{u}_{\lambda_{1,\rho}}^{(1)}( \frac{x}{\sqrt{\lambda_{1,\rho}}} + \xi_{\rho,l}^{(1)} ) \tilde{\kappa}_{l,\rho}^{(2)}(x) \big( \bar{v}_{\lambda_{1,\rho}}^{(1)}( \frac{x}{\sqrt{\lambda_{1,\rho}}} + \xi_{\rho,l}^{(1)} ) + \bar{v}_{\lambda_{2,\rho}}^{(2)}( \frac{x}{\sqrt{\lambda_{1,\rho}}} + \xi_{\rho,l}^{(1)} ) \big) \Big), \\
- \Delta_x \tilde{\kappa}_{l,\rho}^{(2)}(x) + \Big(\frac{1}{\lambda_{1,\rho}} Q( \frac{x}{\sqrt{\lambda_{1,\rho}}} + \xi_{\rho,l}^{(1)} ) + 1 \Big) \tilde{\kappa}_{l,\rho}^{(2)}(x) = \frac{(\frac{\lambda_{2,\rho}}{\lambda_{1,\rho}} - 1) \bar{v}_{\lambda_{2,\rho}}^{(2)}( \frac{x}{\sqrt{\lambda_{1,\rho}}} + \xi_{\rho,l}^{(1)} )}{\|\bar{u}_{\lambda_{1,\rho}}^{(1)} - \bar{u}_{\lambda_{2,\rho}}^{(2)}\|_{L^\infty(\mathbb{R}^N)} + \|\bar{v}_{\lambda_{1,\rho}}^{(1)} - \bar{v}_{\lambda_{2,\rho}}^{(2)}\|_{L^\infty(\mathbb{R}^N)}} \\
\quad + \frac{\mu_2}{\lambda_{1,\rho}} \Big( \big(\bar{v}_{\lambda_{1,\rho}}^{(1)}( \frac{x}{\sqrt{\lambda_{1,\rho}}} + \xi_{\rho,l}^{(1)} )\big)^2 + \bar{v}_{\lambda_{1,\rho}}^{(1)}( \frac{x}{\sqrt{\lambda_{1,\rho}}} + \xi_{\rho,l}^{(1)} ) \bar{v}_{\lambda_{2,\rho}}^{(2)}( \frac{x}{\sqrt{\lambda_{1,\rho}}} + \xi_{\rho,l}^{(1)} ) + \big(\bar{v}_{\lambda_{2,\rho}}^{(2)}( \frac{x}{\sqrt{\lambda_{1,\rho}}} + \xi_{\rho,l}^{(1)} )\big)^2 \Big) \tilde{\kappa}_{l,\rho}^{(2)}(x) \\
\quad + \frac{\beta}{\lambda_{1,\rho}} \Big( \tilde{\kappa}_{l,\rho}^{(2)}(x) \big( \bar{u}_{\lambda_{2,\rho}}^{(2)}( \frac{x}{\sqrt{\lambda_{1,\rho}}} + \xi_{\rho,l}^{(1)} )\big)^2 + \bar{v}_{\lambda_{1,\rho}}^{(1)}( \frac{x}{\sqrt{\lambda_{1,\rho}}} + \xi_{\rho,l}^{(1)} ) \tilde{\kappa}_{l,\rho}^{(1)}(x) ( \bar{u}_{\lambda_{1,\rho}}^{(1)}( \frac{x}{\sqrt{\lambda_{1,\rho}}} + \xi_{\rho,l}^{(1)} ) + \bar{u}_{\lambda_{2,\rho}}^{(2)}( \frac{x}{\sqrt{\lambda_{1,\rho}}} + \xi_{\rho,l}^{(1)} ) ) \Big).
	\end{array} \right.
\end{equation*}

In what follows, we estimate each term in the above expressions.

\begin{lemma}\label{lem52}
Let $\delta>0$ be a small constant. Then for any $x\in B_{\sqrt{\lambda_{1,\rho}}\delta}(0)$, it holds that
\begin{align*}
&\frac{\mu_1}{\lambda_{1,\rho}}  \Big( \big(\bar{u}_{\lambda_{1,\rho}}^{(1)}( \frac{x}{\sqrt{\lambda_{1,\rho}}} + \xi_{\rho,l}^{(1)} )\big)^2 + \bar{u}_{\lambda_{1,\rho}}^{(1)}( \frac{x}{\sqrt{\lambda_{1,\rho}}} + \xi_{\rho,l}^{(1)} ) \bar{u}_{\lambda_{2,\rho}}^{(2)}( \frac{x}{\sqrt{\lambda_{1,\rho}}} + \xi_{\rho,l}^{(1)} ) + \big( \bar{u}_{\lambda_{2,\rho}}^{(2)}( \frac{x}{\sqrt{\lambda_{1,\rho}}} + \xi_{\rho,l}^{(1)} )\big)^2 \Big)\\
=& 3\mu_1(w^{*}(x))^2+O(\lambda_{1,\rho}^{-\frac{1}{2}}w(x)x\cdot \nabla w(x)), \\
&\frac{\mu_2}{\lambda_{1,\rho}} \Big(\big( \bar{v}_{\lambda_{1,\rho}}^{(1)}( \frac{x}{\sqrt{\lambda_{1,\rho}}} + \xi_{\rho,l}^{(1)} ) \big)^2 + \bar{v}_{\lambda_{1,\rho}}^{(1)}( \frac{x}{\sqrt{\lambda_{1,\rho}}} + \xi_{\rho,l}^{(1)} ) \bar{v}_{\lambda_{2,\rho}}^{(2)}( \frac{x}{\sqrt{\lambda_{1,\rho}}} + \xi_{\rho,l}^{(1)} ) + \big( \bar{v}_{\lambda_{2,\rho}}^{(2)}( \frac{x}{\sqrt{\lambda_{1,\rho}}} + \xi_{\rho,l}^{(1)} ) \big)^2 \Big)\\
=& 3\mu_2(w^{\star}(x))^2+O(\lambda_{1,\rho}^{-\frac{1}{2}}w(x)x\cdot \nabla w(x)),
\end{align*}
furthermore,
\begin{align*}
\frac{\beta}{\lambda_{1,\rho}} \big(\bar{v}_{\lambda_{2,\rho}}^{(2)}( \frac{x}{\sqrt{\lambda_{1,\rho}}} + \xi_{\rho,l}^{(1)} )\big)^2=&  \beta (w^{\star}(x))^2+O(\lambda_{1,\rho}^{-\frac{1}{2}}w(x)x\cdot \nabla w(x)), \\
\frac{\beta}{\lambda_{1,\rho}} \big(\bar{u}_{\lambda_{2,\rho}}^{(2)}( \frac{x}{\sqrt{\lambda_{1,\rho}}} + \xi_{\rho,l}^{(1)} )\big)^2=&
   \beta (w^{*}(x))^2+O(\lambda_{1,\rho}^{-\frac{1}{2}}w(x)x\cdot \nabla w(x)) ,
\end{align*}
in addition,
\begin{align*}
&\frac{\beta}{\lambda_{1,\rho}}  \bar{u}_{\lambda_{1,\rho}}^{(1)}( \frac{x}{\sqrt{\lambda_{1,\rho}}} + \xi_{\rho,l}^{(1)} ) ( \bar{v}_{\lambda_{1,\rho}}^{(1)}( \frac{x}{\sqrt{\lambda_{1,\rho}}} + \xi_{\rho,l}^{(1)} ) + \bar{v}_{\lambda_{2,\rho}}^{(2)}( \frac{x}{\sqrt{\lambda_{1,\rho}}} + \xi_{\rho,l}^{(1)} ) ) \\
=&  2\beta w^{*}(x)w^{\star}(x)+O(\lambda_{1,\rho}^{-\frac{1}{2}}w(x)x\cdot \nabla w(x)),\\
&\frac{\beta}{\lambda_{1,\rho}}  \bar{v}_{\lambda_{1,\rho}}^{(1)}( \frac{x}{\sqrt{\lambda_{1,\rho}}} + \xi_{\rho,l}^{(1)} ) ( \bar{u}_{\lambda_{1,\rho}}^{(1)}( \frac{x}{\sqrt{\lambda_{1,\rho}}} + \xi_{\rho,l}^{(1)} ) + \bar{u}_{\lambda_{2,\rho}}^{(2)}( \frac{x}{\sqrt{\lambda_{1,\rho}}} + \xi_{\rho,l}^{(1)} ) ) \\
=& 2\beta w^{*}(x)w^{\star}(x)+O(\lambda_{1,\rho}^{-\frac{1}{2}}w(x)x\cdot \nabla w(x)).
\end{align*}
\end{lemma}

\begin{proof}
By Lemma \ref{lem355} and Lemma \ref{lem51}, we have
\begin{align*}
\frac{\lambda_{2,\rho}}{\lambda_{1,\rho}}&=1+O(\lambda_{1,\rho}^{-\frac{1}{2}})
\end{align*}
and
\begin{align*}
|\xi_{\rho,i}^{(1)}-\xi_{\rho,i}^{(2)}|=|(\xi_{\rho,i}^{(1)}-\xi_i) -(\xi_{\rho,i}^{(2)}-\xi_i)|=O(\lambda_{1,\rho}^{-2}),\quad \text{for } i=1,2,\cdots,k.
\end{align*}
Hence, for any $x\in B_{\sqrt{\lambda_{1,\rho}}\delta}(0)$, we have
\begin{align}\label{509}
&\lambda_{1,\rho}^{-\frac12}\bar{u}_{\lambda_{1,\rho}}^{(1)} \Big(\frac{x}{\sqrt{\lambda_{1,\rho}}}+\xi_{\rho,i}^{(1)}\Big)\notag\\
=&\sum_{l=1}^k \left( w_{\lambda_{1,\rho}, \xi_{\rho,l}^{(1)}}^{*}\Big(\frac{x}{\sqrt{\lambda_{1,\rho}}}+\xi_{\rho,i}^{(1)}\Big) + \lambda_{1,\rho}^{-1} W_{\lambda_{1,\rho}, \xi_{\rho,l}^{(1)}}^{*}\Big(\frac{x}{\sqrt{\lambda_{1,\rho}}}+\xi_{\rho,i}^{(1)}\Big) \right) + \varphi_{\rho}^{(1)}\Big(\frac{x}{\sqrt{\lambda_{1,\rho}}}+\xi_{\rho,i}^{(1)}\Big)\notag\\
=&w^{*}(x)+\lambda_{1,\rho}^{-1}W_l^{*}(x)+ \varphi_{\rho}^{(1)}\Big(\frac{x}{\sqrt{\lambda_{1,\rho}}}+\xi_{\rho,i}^{(1)}\Big)\notag\\
&+\sum_{l\neq i}^k \left( w^{*}(\sqrt{\lambda_{1,\rho}}\Big(\frac{x}{\sqrt{\lambda_{1,\rho}}}+\xi_{\rho,i}^{(1)}-\xi_{\rho,l}^{(1)})\Big) + \lambda_{1,\rho}^{-1} W_l^{*}\Big(\sqrt{\lambda_{1,\rho}}(\frac{x}{\sqrt{\lambda_{1,\rho}}}+\xi_{\rho,i}^{(1)}-\xi_{\rho,l}^{(1)}\Big) \right) \notag\\
=& w^{*}(x)+\lambda_{1,\rho}^{-1}W_l^{*}(x)+ \varphi_{\rho}^{(1)}\Big(\frac{x}{\sqrt{\lambda_{1,\rho}}}+\xi_{\rho,i}^{(1)}\Big)+O\Big(e^{-\theta \sqrt{\lambda_{1,\rho}}}e^{-\delta^2 |y|}\Big).
\end{align}
Similarly, we conclude that
\begin{align}\label{510}
&\lambda_{1,\rho}^{-\frac12}\bar{v}_{\lambda_{1,\rho}}^{(1)}\Big(\frac{x}{\sqrt{\lambda_{1,\rho}}}+\xi_{\rho,i}^{(1)}\Big)\notag\\
=&w^{\star}(x)+\lambda_{1,\rho}^{-1}W_l^{\star}(x)+ \psi_{\rho}^{(1)}\Big(\frac{x}{\sqrt{\lambda_{1,\rho}}}+\xi_{\rho,i}^{(1)}\Big)\notag\\
&+\sum_{l\neq i}^k \left( w^{\star}\Big(\sqrt{\lambda_{1,\rho}}(\frac{x}{\sqrt{\lambda_{1,\rho}}}+\xi_{\rho,i}^{(1)}-\xi_{\rho,l}^{(1)})\Big) + \lambda_{1,\rho}^{-1} W_l^{\star}\Big(\sqrt{\lambda_{1,\rho}}\Big(\frac{x}{\sqrt{\lambda_{1,\rho}}}+\xi_{\rho,i}^{(1)}-\xi_{\rho,l}^{(1)}\Big) \right) \notag\\
=& w^{\star}(x)+\lambda_{1,\rho}^{-1}W_l^{\star}(x)+ \psi_{\rho}^{(1)}\Big(\frac{x}{\sqrt{\lambda_{1,\rho}}}+\xi_{\rho,i}^{(1)}\Big)+O\Big(e^{-\theta \sqrt{\lambda_{1,\rho}}}e^{-\delta^2 |y|}\Big),
\end{align}
in addition,
\begin{align}\label{511}
&\lambda_{2,\rho}^{-\frac12}\bar{u}_{\lambda_{2,\rho}}^{(2)}(\frac{x}{\sqrt{\lambda_{1,\rho}}} +\xi_{\rho,i}^{(1)}) \notag\\
=&\sum_{l=1}^k \left( w_{\lambda_{2,\rho}, \xi_{\rho,l}^{(2)}}^{*}(\frac{x}{\sqrt{\lambda_{1,\rho}}}+\xi_{\rho,i}^{(1)}) + \lambda_{2,\rho}^{-1} W_{\lambda_{2,\rho}, \xi_{\rho,l}^{(2)}}^{*}(\frac{x}{\sqrt{\lambda_{1,\rho}}}+\xi_{\rho,i}^{(1)}) \right) + \varphi_{\rho}^{(2)}(\frac{x}{\sqrt{\lambda_{1,\rho}}}+\xi_{\rho,i}^{(1)})\notag\\
=&w^{*}(\sqrt{\lambda_{2,\rho}}(\frac{x}{\sqrt{\lambda_{1,\rho}}}+\xi_{\rho,i}^{(1)} -\xi_{\rho,i}^{(2)}))+\lambda_{2,\rho}^{-1}W_l^{*}(\sqrt{\lambda_{2,\rho}}(\frac{x}{\sqrt{\lambda_{1,\rho}}} +\xi_{\rho,i}^{(1)}-\xi_{\rho,i}^{(2)}))\notag\\
&+ \varphi_{\rho}^{(2)}(\frac{x}{\sqrt{\lambda_{1,\rho}}}+\xi_{\rho,i}^{(1)}) \notag\\
&+\sum_{l\neq i}^k \left( w^{*}(\sqrt{\lambda_{2,\rho}}(\frac{x}{\sqrt{\lambda_{1,\rho}}}+\xi_{\rho,i}^{(1)}-\xi_{\rho,l}^{(2)})) + \lambda_{2,\rho}^{-1} W_l^{*}(\sqrt{\lambda_{2,\rho}}(\frac{x}{\sqrt{\lambda_{1,\rho}}}+\xi_{\rho,i}^{(1)}-\xi_{\rho,l}^{(2)}) \right) \notag\\
=&    w^{*}(x)+ \Big((\frac{\sqrt{\lambda_{2,\rho}}}{\sqrt{\lambda_{1,\rho}}}-1) x +\sqrt{\lambda_{2,\rho}}(\xi_{\rho,i}^{(1)}-\xi_{\rho,l}^{(2)})\Big)\cdot \nabla w^{*}(x)+ \lambda_{1,\rho}^{-1} W_l^{*}(x) \notag\\
   &\quad+ \Big((\frac{\sqrt{\lambda_{2,\rho}}}{\sqrt{\lambda_{1,\rho}}}-1) x +\sqrt{\lambda_{2,\rho}}(\xi_{\rho,i}^{(1)}-\xi_{\rho,l}^{(2)})\Big)\cdot \nabla^2 w^{*}(x) \Big((\frac{\sqrt{\lambda_{2,\rho}}}{\sqrt{\lambda_{1,\rho}}}-1) x +\sqrt{\lambda_{2,\rho}}(\xi_{\rho,i}^{(1)}-\xi_{\rho,l}^{(2)})\Big) +O(\lambda_{2,\rho}^{-\frac32})
\end{align}
and
\begin{align}\label{512}
&\lambda_{2,\rho}^{-\frac12}\bar{v}_{\lambda_{2,\rho}}^{(2)}(\frac{x}{\sqrt{\lambda_{1,\rho}}}+\xi_{\rho,i}^{(1)})\notag\\
=&w^{\star}(\sqrt{\lambda_{2,\rho}}(\frac{x}{\sqrt{\lambda_{1,\rho}}}+\xi_{\rho,i}^{(1)}-\xi_{\rho,i}^{(2)}))+ \lambda_{2,\rho}^{-1}W_l^{\star}(\sqrt{\lambda_{2,\rho}}(\frac{x}{\sqrt{\lambda_{1,\rho}}}+ \xi_{\rho,i}^{(1)}-\xi_{\rho,i}^{(2)}))\notag\\
&+ \psi_{\rho}^{(2)}(\frac{x}{\sqrt{\lambda_{1,\rho}}}+\xi_{\rho,i}^{(1)})\notag\\
&+\sum_{l\neq i}^k \left( w^{\star}(\sqrt{\lambda_{2,\rho}}(\frac{x}{\sqrt{\lambda_{1,\rho}}}+\xi_{\rho,i}^{(1)}-\xi_{\rho,l}^{(2)})) + \lambda_{2,\rho}^{-1} W_l^{\star}(\sqrt{\lambda_{2,\rho}}(\frac{x}{\sqrt{\lambda_{1,\rho}}}+\xi_{\rho,i}^{(1)}-\xi_{\rho,l}^{(2)}) \right) \notag\\
=&   w^\star(x)+ \Big((\frac{\sqrt{\lambda_{2,\rho}}}{\sqrt{\lambda_{1,\rho}}}-1) x +\sqrt{\lambda_{2,\rho}}(\xi_{\rho,i}^{(1)}-\xi_{\rho,l}^{(2)})\Big)\cdot \nabla w^\star(x)+ \lambda_{1,\rho}^{-1} W_l^{\star}(x) \notag\\
   &\quad + \Big((\frac{\sqrt{\lambda_{2,\rho}}}{\sqrt{\lambda_{1,\rho}}}-1) x +\sqrt{\lambda_{2,\rho}}(\xi_{\rho,i}^{(1)}-\xi_{\rho,l}^{(2)})\Big)\cdot \nabla^2 w^\star(x) \Big((\frac{\sqrt{\lambda_{2,\rho}}}{\sqrt{\lambda_{1,\rho}}}-1) x +\sqrt{\lambda_{2,\rho}}(\xi_{\rho,i}^{(1)}-\xi_{\rho,l}^{(2)})\Big) +O(\lambda_{2,\rho}^{-\frac32}).
\end{align}
By \eqref{509} and \eqref{511}, we get
\begin{align*}
&\frac{\mu_1}{\lambda_{1,\rho}} \Big( \big(\bar{u}_{\lambda_{1,\rho}}^{(1)}( \frac{y}{\sqrt{\lambda_{1,\rho}}} + \xi_{\rho,l}^{(1)} ) \big)^2 + \bar{u}_{\lambda_{1,\rho}}^{(1)}\Big( \frac{y}{\sqrt{\lambda_{1,\rho}}} + \xi_{\rho,l}^{(1)} \Big) \bar{u}_{\lambda_{2,\rho}}^{(2)}\Big( \frac{y}{\sqrt{\lambda_{1,\rho}}} + \xi_{\rho,l}^{(1)} \Big) + \big(\bar{u}_{\lambda_{2,\rho}}^{(2)}( \frac{y}{\sqrt{\lambda_{1,\rho}}} + \xi_{\rho,l}^{(1)} )\big)^2 \Big)\\
=&\mu_1\Big( \big(\lambda_{1,\rho}^{-\frac12}\bar{u}_{\lambda_{1,\rho}}^{(1)}( \frac{y}{\sqrt{\lambda_{1,\rho}}} + \xi_{\rho,l}^{(1)} )\big)^2 + \sqrt{\frac{\lambda_{2,\rho}}{\lambda_{1,\rho}}}\lambda_{1,\rho}^{-\frac12}\bar{u}_{\lambda_{1,\rho}}^{(1)}( \frac{y}{\sqrt{\lambda_{1,\rho}}} + \xi_{\rho,l}^{(1)} ) \lambda_{2,\rho}^{-\frac12}\bar{u}_{\lambda_{2,\rho}}^{(2)}( \frac{y}{\sqrt{\lambda_{1,\rho}}} + \xi_{\rho,l}^{(1)} ) \\
&\quad + \big(\lambda_{1,\rho}^{-\frac12}\bar{u}_{\lambda_{2,\rho}}^{(2)}( \frac{y}{\sqrt{\lambda_{1,\rho}}} + \xi_{\rho,l}^{(1)} )\big)^2 \Big)\\
=&  3\mu_1(w^{*}(x))^2+O(\lambda_{1,\rho}^{-\frac{1}{2}}w(x)x\cdot \nabla w(x)).
\end{align*}
Similarly, there holds
\begin{align*}
&\frac{\mu_2}{\lambda_{1,\rho}} \Big( \big(\bar{v}_{\lambda_{1,\rho}}^{(1)}( \frac{x}{\sqrt{\lambda_{1,\rho}}} + \xi_{\rho,l}^{(1)} )\big)^2 + \bar{v}_{\lambda_{1,\rho}}^{(1)}( \frac{x}{\sqrt{\lambda_{1,\rho}}} + \xi_{\rho,l}^{(1)} ) \bar{v}_{\lambda_{2,\rho}}^{(2)}( \frac{x}{\sqrt{\lambda_{1,\rho}}} + \xi_{\rho,l}^{(1)} ) + \big(\bar{v}_{\lambda_{2,\rho}}^{(2)}( \frac{x}{\sqrt{\lambda_{1,\rho}}} + \xi_{\rho,l}^{(1)} )\big)^2 \Big)\\
=&   3\mu_2(w^{\star}(x))^2+O(\lambda_{1,\rho}^{-\frac{1}{2}}w(x)x\cdot \nabla w(x)).
\end{align*}
In addition, by \eqref{512}, we conclude that
\begin{align*}
\frac{\beta}{\lambda_{1,\rho}} \big(\bar{v}_{\lambda_{2,\rho}}^{(2)}( \frac{x}{\sqrt{\lambda_{1,\rho}}} + \xi_{\rho,l}^{(1)} )\big)^2=
    \beta (w^{\star}(x))^2+O(\lambda_{1,\rho}^{-\frac{1}{2}}w(x)x\cdot \nabla w(x)),
\end{align*}
similarly,
\begin{align*}
\frac{\beta}{\lambda_{1,\rho}} \big(\bar{u}_{\lambda_{2,\rho}}^{(2)}( \frac{x}{\sqrt{\lambda_{1,\rho}}} + \xi_{\rho,l}^{(1)})\big)^2=
    \beta (w^{*}(x))^2+O(\lambda_{1,\rho}^{-\frac{1}{2}}w(x)x\cdot \nabla w(x)).
\end{align*}
According to \eqref{509}, \eqref{510} and \eqref{512}, we have
\begin{align*}
&\frac{\beta}{\lambda_{1,\rho}}  \bar{u}_{\lambda_{1,\rho}}^{(1)} ( \frac{x}{\sqrt{\lambda_{1,\rho}}} + \xi_{\rho,l}^{(1)}  )  \big( \bar{v}_{\lambda_{1,\rho}}^{(1)} ( \frac{x}{\sqrt{\lambda_{1,\rho}}} + \xi_{\rho,l}^{(1)} ) + \bar{v}_{\lambda_{2,\rho}}^{(2)} ( \frac{x}{\sqrt{\lambda_{1,\rho}}} + \xi_{\rho,l}^{(1)} ) \big) \\
=& \beta \lambda_{1,\rho}^{-\frac12}\bar{u}_{\lambda_{1,\rho}}^{(1)} ( \frac{x}{\sqrt{\lambda_{1,\rho}}} + \xi_{\rho,l}^{(1)}  )  \big( \lambda_{1,\rho}^{-\frac12}\bar{v}_{\lambda_{1,\rho}}^{(1)} ( \frac{x}{\sqrt{\lambda_{1,\rho}}} + \xi_{\rho,l}^{(1)}  ) + \lambda_{1,\rho}^{-\frac12}\bar{v}_{\lambda_{2,\rho}}^{(2)} ( \frac{x}{\sqrt{\lambda_{1,\rho}}} + \xi_{\rho,l}^{(1)}  )  \big)\\
=&   2\beta w^{*}(x)w^{\star}(x)+O(\lambda_{1,\rho}^{-\frac{1}{2}}w(x)x\cdot \nabla w(x)).
\end{align*}
Similarly,
\begin{align*}
&\frac{\beta}{\lambda_{1,\rho}}  \bar{v}_{\lambda_{1,\rho}}^{(1)} ( \frac{x}{\sqrt{\lambda_{1,\rho}}} + \xi_{\rho,l}^{(1)}  )  \big( \bar{u}_{\lambda_{1,\rho}}^{(1)} ( \frac{x}{\sqrt{\lambda_{1,\rho}}} + \xi_{\rho,l}^{(1)} ) + \bar{u}_{\lambda_{2,\rho}}^{(2)} ( \frac{x}{\sqrt{\lambda_{1,\rho}}} + \xi_{\rho,l}^{(1)} ) \big) \\
=&  2\beta w^{*}(x)w^{\star}(x)+O(\lambda_{1,\rho}^{-\frac{1}{2}}w(x)x\cdot \nabla w(x)).
\end{align*}
\end{proof}

\begin{lemma}\label{lem53}
Let $\delta>0$ be a small constant. Then for any $x\in B_{\sqrt{\lambda_{1,\rho}}\delta}(0)$, it holds that for $l=1,2,\cdots,k$,
\begin{align*}
&\frac{(\frac{\lambda_{2,\rho}}{\lambda_{1,\rho}} - 1) \bar{u}_{\lambda_{2,\rho}}^{(2)}\left( \frac{x}{\sqrt{\lambda_{1,\rho}}} + \xi_{\rho,l}^{(1)} \right)}{\|\bar{u}_{\lambda_{1,\rho}}^{(1)} - \bar{u}_{\lambda_{2,\rho}}^{(2)}\|_{L^\infty(\mathbb{R}^N)} + \|\bar{v}_{\lambda_{1,\rho}}^{(1)} - \bar{v}_{\lambda_{2,\rho}}^{(2)}\|_{L^\infty(\mathbb{R}^N)}} \\
=&-\frac{2 w^{*}(x)}{k\int_{\mathbb{R}^N}\left(\left( w^{*}\right)^2+\left( w^{\star}\right)^2\right)} \sum_{i=1}^{k} \int_{ \mathbb{R}^N} (\mu_1 \tilde{\kappa}_{i,\rho}^{(1)} (w^{*}(x))^3+\mu_2 \tilde{\kappa}_{i,\rho}^{(2)} (w^{\star}(x))^3+\beta w^{*}(x)w^{\star}(x)\notag\\
&\times(w^{\star}(x) \tilde{\kappa}_{i,\rho}^{(1)}+w^{*}(x) \tilde{\kappa}_{i,\rho}^{(2)})) +O(\lambda_{1,\rho}^{-\frac12}w(x))
\end{align*}
and
\begin{align*}
&\frac{(\frac{\lambda_{2,\rho}}{\lambda_{1,\rho}} - 1) \bar{v}_{\lambda_{2,\rho}}^{(2)}\Big( \frac{x}{\sqrt{\lambda_{1,\rho}}} + \xi_{\rho,l}^{(1)} \Big)}{\|\bar{u}_{\lambda_{1,\rho}}^{(1)} - \bar{u}_{\lambda_{2,\rho}}^{(2)}\|_{L^\infty(\mathbb{R}^N)} + \|\bar{v}_{\lambda_{1,\rho}}^{(1)} - \bar{v}_{\lambda_{2,\rho}}^{(2)}\|_{L^\infty(\mathbb{R}^N)}} \\
=&- \frac{2 w^{\star}(x)}{k\int_{\mathbb{R}^N}\left(\left( w^{*}\right)^2+\left( w^{\star}\right)^2\right)} \sum_{i=1}^{k} \int_{ \mathbb{R}^N} (\mu_1 \tilde{\kappa}_{i,\rho}^{(1)} (w^{*}(x))^3+\mu_2 \tilde{\kappa}_{i,\rho}^{(2)} (w^{\star}(x))^3+\beta w^{*}(x)w^{\star}(x)\notag\\
&\times(w^{\star}(x) \tilde{\kappa}_{i,\rho}^{(1)}+w^{*}(x) \tilde{\kappa}_{i,\rho}^{(2)})) +O(\lambda_{1,\rho}^{-\frac12}w(x)).
\end{align*}

\end{lemma}

\begin{proof}
By \eqref{eq0}, we have
\begin{align*}
\int_{\mathbb{R}^N}(|\nabla\bar{u}_{\lambda_{i,\rho}}^{(i)}|^2+(P\left( x \right) + \lambda_{i,\rho} )(\bar{u}_{\lambda_{i,\rho}}^{(i)})^2)= \int_{\mathbb{R}^N}(\mu_1 ( \bar{u}_{\lambda_{i,\rho}}^{(i)})^4+\beta (\bar{u}_{\lambda_{i,\rho}}^{(i)}\bar{v}_{\lambda_{i,\rho}}^{(i)})^2),\quad i=1,2,
\end{align*}
which deduce that
\begin{align*}
&\int_{\mathbb{R}^N} ((\nabla \bar{u}_{\lambda_{1,\rho}}^{(1)}- \nabla \bar{u}_{\lambda_{2,\rho}}^{(2)}) (\nabla \bar{u}_{\lambda_{1,\rho}}^{(1)}+ \nabla \bar{u}_{\lambda_{2,\rho}}^{(2)})+ P(x)(\bar{u}_{\lambda_{1,\rho}}^{(1)}-\bar{u}_{\lambda_{2,\rho}}^{(2)}) (\bar{u}_{\lambda_{1,\rho}}^{(1)}+\bar{u}_{\lambda_{2,\rho}}^{(2)}) \\
&\quad\quad+\lambda_{1,\rho}(\bar{u}_{\lambda_{1,\rho}}^{(1)}- \bar{u}_{\lambda_{2,\rho}}^{(2)})(\bar{u}_{\lambda_{1,\rho}}^{(1)}+\bar{u}_{\lambda_{2,\rho}}^{(2)}) +(\lambda_{1,\rho}-\lambda_{2,\rho})(\bar{u}_{\lambda_{2,\rho}}^{(2)})^2) \\
=&\int_{\mathbb{R}^N}  (\mu_1 (\bar{u}_{\lambda_{1,\rho}}^{(1)}- \bar{u}_{\lambda_{2,\rho}}^{(2)})((\bar{u}_{\lambda_{1,\rho}}^{(1)})^3 + (\bar{u}_{\lambda_{1,\rho}}^{(1)})^2\bar{u}_{\lambda_{2,\rho}}^{(2)}+ \bar{u}_{\lambda_{1,\rho}}^{(1)}(\bar{u}_{\lambda_{2,\rho}}^{(2)})^2 +(\bar{u}_{\lambda_{2,\rho}}^{(2)})^3) \\
&\quad\quad+\beta ( (\bar{u}_{\lambda_{1,\rho}}^{(1)}-\bar{u}_{\lambda_{2,\rho}}^{(2)}) (\bar{u}_{\lambda_{1,\rho}}^{(1)}+\bar{u}_{\lambda_{2,\rho}}^{(2)})(\bar{v}_{\lambda_{1,\rho}}^{(1)})^2 +(\bar{v}_{\lambda_{1,\rho}}^{(1)}- \bar{v}_{\lambda_{2,\rho}}^{(2)})(\bar{v}_{\lambda_{1,\rho}}^{(1)}+ \bar{v}_{\lambda_{2,\rho}}^{(2)})(\bar{u}_{\lambda_{2,\rho}}^{(2)})^2)).
\end{align*}
Therefore, we conclude that
\begin{align}\label{5531}
\frac{(\frac{\lambda_{2,\rho}}{\lambda_{1,\rho}}-1 ) \int_{\mathbb{R}^N} (\bar{u}_{\lambda_{2,\rho}}^{(2)})^2}{\|\bar{u}_{\lambda_{1,\rho}}^{(1)} - \bar{u}_{\lambda_{2,\rho}}^{(2)}\|_{L^\infty(\mathbb{R}^N)} + \|\bar{v}_{\lambda_{1,\rho}}^{(1)} - \bar{v}_{\lambda_{2,\rho}}^{(2)}\|_{L^\infty(\mathbb{R}^N)}}
=&\frac{1}{\lambda_{1,\rho}}\int_{\mathbb{R}^N} (\nabla \kappa_{\rho}^{(1)}  (\nabla \bar{u}_{\lambda_{1,\rho}}^{(1)}+ \nabla \bar{u}_{\lambda_{2,\rho}}^{(2)})+ (P(x)+\lambda_{1,\rho})(\bar{u}_{\lambda_{1,\rho}}^{(1)}+\bar{u}_{\lambda_{2,\rho}}^{(2)})\kappa_{\rho}^{(1)} ) \notag\\
& - \frac{\mu_1}{\lambda_{1,\rho}}\int_{\mathbb{R}^N}(\kappa_{\rho}^{(1)} ((\bar{u}_{\lambda_{1,\rho}}^{(1)})^3 + (\bar{u}_{\lambda_{1,\rho}}^{(1)})^2\bar{u}_{\lambda_{2,\rho}}^{(2)}+ \bar{u}_{\lambda_{1,\rho}}^{(1)}(\bar{u}_{\lambda_{2,\rho}}^{(2)})^2 +(\bar{u}_{\lambda_{2,\rho}}^{(2)})^3)\notag\\
&- \frac{\beta}{\lambda_{1,\rho}} \int_{\mathbb{R}^N} ( \kappa_{\rho}^{(1)}  (\bar{u}_{\lambda_{1,\rho}}^{(1)}+\bar{u}_{\lambda_{2,\rho}}^{(2)})(\bar{v}_{\lambda_{1,\rho}}^{(1)})^2 +\kappa_{\rho}^{(2)} (\bar{v}_{\lambda_{1,\rho}}^{(1)}+ \bar{v}_{\lambda_{2,\rho}}^{(2)})(\bar{u}_{\lambda_{2,\rho}}^{(2)})^2).
\end{align}
Applying system \eqref{eq0} again, we obtain
\begin{align}\label{5532}
&\frac{1}{\lambda_{1,\rho}}\int_{\mathbb{R}^N} (\nabla \kappa_{\rho}^{(1)}  (\nabla \bar{u}_{\lambda_{1,\rho}}^{(1)}+ \nabla \bar{u}_{\lambda_{2,\rho}}^{(2)})+ (P(x)+\lambda_{1,\rho}) (\bar{u}_{\lambda_{1,\rho}}^{(1)}+\bar{u}_{\lambda_{2,\rho}}^{(2)})\kappa_{\rho}^{(1)}  )\notag \\
=&\frac{1}{\lambda_{1,\rho}}\int_{\mathbb{R}^N}  ((\lambda_{1,\rho}-\lambda_{2,\rho})\bar{u}_{\lambda_{2,\rho}}^{(2)}\kappa_{\rho}^{(1)}  +\mu_1 ((\bar{u}_{\lambda_{1,\rho}}^{(1)})^3 +(\bar{u}_{\lambda_{2,\rho}}^{(2)})^3)\kappa_{\rho}^{(1)} +\beta (\bar{u}_{\lambda_{1,\rho}}^{(1)}(\bar{v}_{\lambda_{1,\rho}}^{(1)})^2 + \bar{u}_{\lambda_{2,\rho}}^{(2)} (\bar{v}_{\lambda_{2,\rho}}^{(2)})^2)\kappa_{\rho}^{(1)} .
\end{align}
Combining \eqref{5531} and \eqref{5532}, we get
\begin{align}\label{5533}
&\frac{(\frac{\lambda_{2,\rho}}{\lambda_{1,\rho}}-1 ) \int_{\mathbb{R}^N} (\bar{u}_{\lambda_{2,\rho}}^{(2)})^2}{\|\bar{u}_{\lambda_{1,\rho}}^{(1)} - \bar{u}_{\lambda_{2,\rho}}^{(2)}\|_{L^\infty(\mathbb{R}^N)} + \|\bar{v}_{\lambda_{1,\rho}}^{(1)} - \bar{v}_{\lambda_{2,\rho}}^{(2)}\|_{L^\infty(\mathbb{R}^N)}}  \notag\\
=&\int_{\mathbb{R}^N}  ((1-\frac{\lambda_{2,\rho}}{\lambda_{1,\rho}}) \bar{u}_{\lambda_{2,\rho}}^{(2)} \kappa_{\rho}^{(1)} )- \frac{\mu_1}{\lambda_{1,\rho}}\int_{\mathbb{R}^N}(\kappa_{\rho}^{(1)} ( (\bar{u}_{\lambda_{1,\rho}}^{(1)})^2\bar{u}_{\lambda_{2,\rho}}^{(2)}+ \bar{u}_{\lambda_{1,\rho}}^{(1)}(\bar{u}_{\lambda_{2,\rho}}^{(2)})^2 ))\notag \\
&-\frac{\beta}{\lambda_{1,\rho}} \int_{\mathbb{R}^N} (\bar{u}_{\lambda_{2,\rho}}^{(2)}((\bar{v}_{\lambda_{1,\rho}}^{(1)})^2-(\bar{v}_{\lambda_{2,\rho}}^{(2)})^2)\kappa_{\rho}^{(1)}  +(\bar{u}_{\lambda_{2,\rho}}^{(2)})^2(\bar{v}_{\lambda_{1,\rho}}^{(1)}+\bar{v}_{\lambda_{2,\rho}}^{(2)})\kappa_{\rho}^{(2)}  ).
\end{align}
From \eqref{511}, we obtain
\begin{align}\label{55333}
\int_{\mathbb{R}^N}  ((1-\frac{\lambda_{2,\rho}}{\lambda_{1,\rho}}) \bar{u}_{\lambda_{2,\rho}}^{(2)} \kappa_{\rho}^{(1)} )
=&\int_{\bigcup_{l=1}^{k} B_{\delta}(\xi_{\rho,l}^{(1)})} ((1-\frac{\lambda_{2,\rho}}{\lambda_{1,\rho}}) \bar{u}_{\lambda_{2,\rho}}^{(2)} \kappa_{\rho}^{(1)} )   +O(e^{-\theta\sqrt{\lambda_{1,\rho}}}) \notag\\
=&\sum_{i=1}^{k}\int_{ B_{\delta\sqrt{\lambda_{1,\rho}}}(0)} (\lambda_{1,\rho}^{-\frac{N}{2}}\lambda_{2,\rho}^{\frac12}(1-\frac{\lambda_{2,\rho}}{\lambda_{1,\rho}}) \lambda_{2,\rho}^{-\frac12}\bar{u}_{\lambda_{2,\rho}}^{(2)}(\frac{x}{\sqrt{\lambda_{1,\rho}}}+\xi_{\rho,i}^{(1)}) \tilde{\kappa}_{i,\rho}^{(1)}) +O\big(e^{-\theta\sqrt{\lambda_{1,\rho}}}\big)\notag\\
=&O(\lambda_{1,\rho}^{-\frac{N}{2}}).
\end{align}
By \eqref{509} and \eqref{511}, we have
\begin{align}\label{5534}
&\int_{\mathbb{R}^N}(\kappa_{\rho}^{(1)} ( (\bar{u}_{\lambda_{1,\rho}}^{(1)})^2\bar{u}_{\lambda_{2,\rho}}^{(2)}+ \bar{u}_{\lambda_{1,\rho}}^{(1)}(\bar{u}_{\lambda_{2,\rho}}^{(2)})^2 )\notag\\
=&\int_{\bigcup_{l=1}^{k} B_{\delta}(\xi_{\rho,l}^{(1)})} (\kappa_{\rho}^{(1)} ( (\bar{u}_{\lambda_{1,\rho}}^{(1)})^2\bar{u}_{\lambda_{2,\rho}}^{(2)}+ \bar{u}_{\lambda_{1,\rho}}^{(1)}(\bar{u}_{\lambda_{2,\rho}}^{(2)})^2 )+O(e^{-\theta\sqrt{\lambda_{1,\rho}}})\notag\\
=&\lambda_{1,\rho}^{\frac{3-N}{2}}\sum_{i=1}^{k}\int_{ B_{\delta\sqrt{\lambda_{1,\rho}}}(0)} (\tilde{\kappa}_{i,\rho}^{(1)}( (\lambda_{1,\rho}^{-\frac12}\bar{u}_{\lambda_{1,\rho}}^{(1)}(\frac{x}{\sqrt{\lambda_{1,\rho}}}+\xi_{\rho,i}^{(1)}))^2 \sqrt{\frac{\lambda_{2,\rho}}{\lambda_{1,\rho}}} \lambda_{2,\rho}^{-\frac12}\bar{u}_{\lambda_{2,\rho}}^{(2)}(\frac{x}{\sqrt{\lambda_{1,\rho}}}+\xi_{\rho,i}^{(1)}) \notag\\
&\quad\quad\quad\quad\quad+ \lambda_{1,\rho}^{-\frac12}\bar{u}_{\lambda_{1,\rho}}^{(1)}(\frac{x}{\sqrt{\lambda_{1,\rho}}}+\xi_{\rho,i}^{(1)}) (\sqrt{\frac{\lambda_{2,\rho}}{\lambda_{1,\rho}}}\lambda_{2,\rho}^{-\frac12}\bar{u}_{\lambda_{2,\rho}}^{(2)}(\frac{x}{\sqrt{\lambda_{1,\rho}}}+\xi_{\rho,i}^{(1)}))^2 ))+O(e^{-\theta\sqrt{\lambda_{1,\rho}}})\notag\\
=&2\lambda_{1,\rho}^{\frac{3-N}{2}}\sum_{i=1}^{k}\int_{ B_{\delta\sqrt{\lambda_{1,\rho}}}(0)} (\tilde{\kappa}_{i,\rho}^{(1)} (w^{*}(x))^3) + O(\lambda_{1,\rho}^{-\frac{N-2}{2}}).
\end{align}
According to \eqref{510}-\eqref{512}, we conclude that
\begin{align}\label{5535}
&\int_{\mathbb{R}^N} (\bar{u}_{\lambda_{2,\rho}}^{(2)}((\bar{v}_{\lambda_{1,\rho}}^{(1)})^2-(\bar{v}_{\lambda_{2,\rho}}^{(2)})^2)\kappa_{\rho}^{(1)}  +(\bar{u}_{\lambda_{2,\rho}}^{(2)})^2(\bar{v}_{\lambda_{1,\rho}}^{(1)}+\bar{v}_{\lambda_{2,\rho}}^{(2)})\kappa_{\rho}^{(2)}  ) \notag\\
=&\int_{\bigcup_{l=1}^{k} B_{\delta}(\xi_{\rho,l}^{(1)})} (\bar{u}_{\lambda_{2,\rho}}^{(2)}((\bar{v}_{\lambda_{1,\rho}}^{(1)})^2-(\bar{v}_{\lambda_{2,\rho}}^{(2)})^2)\kappa_{\rho}^{(1)}  +(\bar{u}_{\lambda_{2,\rho}}^{(2)})^2(\bar{v}_{\lambda_{1,\rho}}^{(1)}+\bar{v}_{\lambda_{2,\rho}}^{(2)})\kappa_{\rho}^{(2)}  ) + O(e^{-\theta\sqrt{\lambda_{1,\rho}}}) \notag\\
=&\lambda_{1,\rho}^{\frac{3-N}{2}}\sum_{i=1}^{k}\int_{ B_{\delta\sqrt{\lambda_{1,\rho}}}(0)} (\sqrt{\frac{\lambda_{2,\rho}}{\lambda_{1,\rho}}}\lambda_{2,\rho}^{-\frac12}\bar{u}_{\lambda_{2,\rho}}^{(2)}(\frac{x}{\sqrt{\lambda_{1,\rho}}}+\xi_{\rho,i}^{(1)}) ((\lambda_{1,\rho}^{-\frac12}\bar{v}_{\lambda_{1,\rho}}^{(1)}(\frac{x}{\sqrt{\lambda_{1,\rho}}}+\xi_{\rho,i}^{(1)}))^2 \notag\\
&\quad\quad\quad\quad\quad\quad-(\sqrt{\frac{\lambda_{2,\rho}}{\lambda_{1,\rho}}}\lambda_{2,\rho}^{-\frac12}\bar{v}_{\lambda_{2,\rho}}^{(2)}(\frac{x}{\sqrt{\lambda_{1,\rho}}}+\xi_{\rho,i}^{(1)}))^2)\tilde{\kappa}_{i,\rho}^{(1)} \notag\\
&\quad\quad\quad\quad \quad\quad +(\sqrt{\frac{\lambda_{2,\rho}}{\lambda_{1,\rho}}}\lambda_{2,\rho}^{-\frac12}\bar{u}_{\lambda_{2,\rho}}^{(2)}(\frac{x}{\sqrt{\lambda_{1,\rho}}}+\xi_{\rho,i}^{(1)}))^2 (\lambda_{1,\rho}^{-\frac12}\bar{v}_{\lambda_{1,\rho}}^{(1)}(\frac{x}{\sqrt{\lambda_{1,\rho}}}+\xi_{\rho,i}^{(1)}) \notag\\
&\quad\quad\quad\quad\quad\quad+\sqrt{\frac{\lambda_{2,\rho}}{\lambda_{1,\rho}}}\lambda_{2,\rho}^{-\frac12}\bar{v}_{\lambda_{2,\rho}}^{(2)}(\frac{x}{\sqrt{\lambda_{1,\rho}}}+\xi_{\rho,i}^{(1)})) \tilde{\kappa}_{i,\rho}^{(2)} ) +O(e^{-\theta\sqrt{\lambda_{1,\rho}}}) \notag\\
=& \lambda_{1,\rho}^{\frac{3-N}{2}}\sum_{i=1}^{k} \int_{ B_{\delta\sqrt{\lambda_{1,\rho}}}(0)} ((w^{*}(x))^2 w^{\star}(x) \tilde{\kappa}_{i,\rho}^{(2)})
+O(\lambda_{1,\rho}^{-\frac{N-2}{2}}).
\end{align}
Combining \eqref{5533}-\eqref{5535}, we have
\begin{align}\label{5536}
&\frac{(\frac{\lambda_{2,\rho}}{\lambda_{1,\rho}}-1 ) \int_{\mathbb{R}^N} (\bar{u}_{\lambda_{2,\rho}}^{(2)})^2}{\|\bar{u}_{\lambda_{1,\rho}}^{(1)} - \bar{u}_{\lambda_{2,\rho}}^{(2)}\|_{L^\infty(\mathbb{R}^N)} + \|\bar{v}_{\lambda_{1,\rho}}^{(1)} - \bar{v}_{\lambda_{2,\rho}}^{(2)}\|_{L^\infty(\mathbb{R}^N)}} \notag\\
=&- 2 \lambda_{1,\rho}^{-\frac{N-1}{2}} \sum_{i=1}^{k} \int_{ B_{\delta\sqrt{\lambda_{1,\rho}}}(0)} (\mu_1 \tilde{\kappa}_{i,\rho}^{(1)} (w^{*}(x))^3+\beta(w^{*}(x))^2 w^{\star}(x) \tilde{\kappa}_{i,\rho}^{(2)})+ O(\lambda_{1,\rho}^{-\frac{N}{2}}).
\end{align}
Similarly,
\begin{align}\label{5537}
&\frac{(\frac{\lambda_{2,\rho}}{\lambda_{1,\rho}}-1 ) \int_{\mathbb{R}^N} (\bar{v}_{\lambda_{2,\rho}}^{(2)})^2}{\|\bar{v}_{\lambda_{1,\rho}}^{(1)} - \bar{u}_{\lambda_{2,\rho}}^{(2)}\|_{L^\infty(\mathbb{R}^N)} + \|\bar{v}_{\lambda_{1,\rho}}^{(1)} - \bar{v}_{\lambda_{2,\rho}}^{(2)}\|_{L^\infty(\mathbb{R}^N)}} \notag \\
=&- 2 \lambda_{1,\rho}^{-\frac{N-1}{2}}\sum_{i=1}^{k} \int_{ B_{\delta\sqrt{\lambda_{1,\rho}}}(0)} (\mu_2 \tilde{\kappa}_{i,\rho}^{(2)} (w^{\star}(x))^3+\beta(w^{\star}(x))^2 w^{*}(x) \tilde{\kappa}_{i,\rho}^{(1)})+ O(\lambda_{1,\rho}^{-\frac{N}{2}}).
\end{align}
By \eqref{rho2}, \eqref{5536} and \eqref{5537}, we have
\begin{align}\label{5538}
&\frac{(\frac{\lambda_{2,\rho}}{\lambda_{1,\rho}}-1 ) \rho^2 }{\|\bar{v}_{\lambda_{1,\rho}}^{(1)} - \bar{u}_{\lambda_{2,\rho}}^{(2)}\|_{L^\infty(\mathbb{R}^N)} + \|\bar{v}_{\lambda_{1,\rho}}^{(1)} - \bar{v}_{\lambda_{2,\rho}}^{(2)}\|_{L^\infty(\mathbb{R}^N)}} \notag \\
=&- 2 \lambda_{1,\rho}^{-\frac{N-1}{2}} \sum_{i=1}^{k} \int_{ B_{\delta\sqrt{\lambda_{1,\rho}}}(0)} (\mu_1 \tilde{\kappa}_{i,\rho}^{(1)} (w^{*}(x))^3+\mu_2 \tilde{\kappa}_{i,\rho}^{(2)} (w^{\star}(x))^3+\beta w^{*}(x)w^{\star}(x)(w^{\star}(x) \tilde{\kappa}_{i,\rho}^{(1)}+w^{*}(x) \tilde{\kappa}_{i,\rho}^{(2)}))\notag\\
&+ O(\lambda_{1,\rho}^{-\frac{N}{2}}).
\end{align}
Then by combining \eqref{511} and \eqref{5538}, we get
\begin{align*}
&\frac{(\frac{\lambda_{2,\rho}}{\lambda_{1,\rho}}-1 ) \bar{u}_{\lambda_{2,\rho}}^{(2)}(\frac{x}{\sqrt{\lambda_{1,\rho}}} +\xi_{\rho,i}^{(1)}) }{\|\bar{v}_{\lambda_{1,\rho}}^{(1)} - \bar{u}_{\lambda_{2,\rho}}^{(2)}\|_{L^\infty(\mathbb{R}^N)} + \|\bar{v}_{\lambda_{1,\rho}}^{(1)} - \bar{v}_{\lambda_{2,\rho}}^{(2)}\|_{L^\infty(\mathbb{R}^N)}} \notag \\
=&- \lambda_{2,\rho}^{-\frac12}\bar{u}_{\lambda_{2,\rho}}^{(2)}(\frac{x}{\sqrt{\lambda_{1,\rho}}} +\xi_{\rho,i}^{(1)})\frac{2  \lambda_{2,\rho}^{\frac12}}{\rho^{2}}\lambda_{1,\rho}^{-\frac{N-1}{2}} \sum_{i=1}^{k} \int_{ B_{\delta\sqrt{\lambda_{1,\rho}}}(0)} (\mu_1 \tilde{\kappa}_{i,\rho}^{(1)} (w^{*}(x))^3+\mu_2 \tilde{\kappa}_{i,\rho}^{(2)} (w^{\star}(x))^3+\beta w^{*}(x)w^{\star}(x)\notag\\
&\times(w^{\star}(x) \tilde{\kappa}_{i,\rho}^{(1)}+w^{*}(x) \tilde{\kappa}_{i,\rho}^{(2)}))
+ \rho^{-2}\lambda_{2,\rho}^{\frac12}\lambda_{2,\rho}^{-\frac12}\bar{u}_{\lambda_{2,\rho}}^{(2)}(\frac{x}{\sqrt{\lambda_{1,\rho}}} +\xi_{\rho,i}^{(1)})O(\lambda_{1,\rho}^{-\frac{N}{2}})\notag\\
=&- \frac{2w^{*}(x)}{k\int_{\mathbb{R}^N}\left(\left( w^{*}\right)^2+\left( w^{\star}\right)^2\right)} \sum_{i=1}^{k} \int_{ \mathbb{R}^N} (\mu_1 \tilde{\kappa}_{i,\rho}^{(1)} (w^{*}(x))^3+\mu_2 \tilde{\kappa}_{i,\rho}^{(2)} (w^{\star}(x))^3+\beta w^{*}(x)w^{\star}(x)\notag\\
&\times(w^{\star}(x) \tilde{\kappa}_{i,\rho}^{(1)}+w^{*}(x) \tilde{\kappa}_{i,\rho}^{(2)})) +O(\lambda_{1,\rho}^{-\frac12}w(x)).
\end{align*}
Similarly,
\begin{align*}
&\frac{(\frac{\lambda_{2,\rho}}{\lambda_{1,\rho}}-1 ) \bar{v}_{\lambda_{2,\rho}}^{(2)}(\frac{x}{\sqrt{\lambda_{1,\rho}}} +\xi_{\rho,i}^{(1)}) }{\|\bar{v}_{\lambda_{1,\rho}}^{(1)} - \bar{u}_{\lambda_{2,\rho}}^{(2)}\|_{L^\infty(\mathbb{R}^N)} + \|\bar{v}_{\lambda_{1,\rho}}^{(1)} - \bar{v}_{\lambda_{2,\rho}}^{(2)}\|_{L^\infty(\mathbb{R}^N)}} \notag \\
=&- \frac{2w^{\star}(x)}{k\int_{\mathbb{R}^N}\left(\left( w^{*}\right)^2+\left( w^{\star}\right)^2\right)} \sum_{i=1}^{k} \int_{ \mathbb{R}^N} (\mu_1 \tilde{\kappa}_{i,\rho}^{(1)} (w^{*}(x))^3+\mu_2 \tilde{\kappa}_{i,\rho}^{(2)} (w^{\star}(x))^3+\beta w^{*}(x)w^{\star}(x)\notag\\
&\times(w^{\star}(x) \tilde{\kappa}_{i,\rho}^{(1)}+w^{*}(x) \tilde{\kappa}_{i,\rho}^{(2)})) +O(\lambda_{1,\rho}^{-\frac12}w(x)).
\end{align*}
\end{proof}

Next, we give the estimate of $\kappa_{\rho}^{(1)} $ and $\kappa_{\rho}^{(2)} $ in $\mathbb{R}^3\backslash \cup^{k}_{l=1} B_{\frac{R}{\sqrt{\lambda_{1,\rho}}}}(\xi_{\rho,l}^{(1)})$.
\begin{lemma}\label{lem54}
There exist constants $C>0$ and $\tau>0$, such that
\begin{align}\label{531}
| \kappa_{\rho}^{(i)} | \leq C \sum _ { l = 1 } ^ { k } e ^ { -   \tau\sqrt{\lambda_{1,\rho}}| x - \xi^{(1)}_{\rho,l} |  },\ \ \ \forall x \in \mathbb{R}^N\backslash \cup^{k}_{l=1} B_{\frac{R}{\sqrt{\lambda_{1,\rho}}}}(\xi_{\rho,l}^{(1)}),\ \ \ i=1,2
\end{align}
and
\begin{align}\label{532}
| \nabla \kappa_{\rho}^{(i )}| \leq C e ^ { -  \tau \sqrt{\lambda_{1,\rho}}},\ \ \ \forall x \in \partial B_{\delta}(\xi^{(1)}_{\rho,l}),\ \ \ i=1,2,\ \ \ l=1,\cdots,k.
\end{align}
\end{lemma}

\begin{proof}
Choosing large enough $R>0$, for $x\in \mathbb{R}^N\backslash \cup^{k}_{l=1} B_{\frac{R}{\sqrt{\lambda_{1,\rho}}}}(\xi_{\rho,l}^{(1)})$, we have
\begin{align*}
1+\frac{1}{\lambda_{1,\rho}}P(x)-\frac{\mu_1}{\lambda_{1,\rho}}  ((\bar{u}^{(1)}_{\lambda_{1,\rho}})^2 +\bar{u}^{(1)}_{\lambda_{1,\rho}}\bar{u}^{(2)}_{\lambda_{2,\rho}} +(\bar{u}^{(2)}_{\lambda_{2,\rho}})^2)-\frac{\beta}{\lambda_{1,\rho}}(\bar{v}^{(2)}_{\lambda_{2,\rho}})^2\geq\frac12
\end{align*}
and
\begin{align*}
1+\frac{1}{\lambda_{1,\rho}}Q(x)-\frac{\mu_2}{\lambda_{1,\rho}}  ((\bar{v}^{(1)}_{\lambda_{1,\rho}})^2 +\bar{v}^{(1)}_{\lambda_{1,\rho}}\bar{v}^{(2)}_{\lambda_{2,\rho}} +(\bar{v}^{(2)}_{\lambda_{2,\rho}})^2)-\frac{\beta}{\lambda_{1,\rho}}(\bar{u}^{(2)}_{\lambda_{2,\rho}})^2\geq\frac12.
\end{align*}
According to Lemma \ref{lem33}, Lemma \ref{lem51} and Lemma \ref{lem53}, we can conclude that
\begin{align*}
\frac{\lambda_{1,\rho}^{-1}|\lambda_{2,\rho}-\lambda_{1,\rho}|\bar{u}_{\lambda_{2,\rho}}^{(2)}}{||\bar{u}^{(1)}_{\lambda_{1,\rho}} -\bar{u}^{(2)}_{\lambda_{2,\rho}}||_{L^{\infty}(\mathbb{R}^N)} +||\bar{v}^{(1)}_{\lambda_{1,\rho}}-\bar{v}^{(2)}_{\lambda_{2,\rho}}||_{L^{\infty}(\mathbb{R}^N)}}\leq
C  \sum _ { l = 1 } ^ { k } e ^ { - \tau \sqrt {\lambda_{1,\rho}} | x - \xi_{\rho,l} |  }
\end{align*}
and
\begin{align*}
\frac{\lambda_{1,\rho}^{-1}|\lambda_{2,\rho}-\lambda_{1,\rho}|\bar{v}_{\lambda_{2,\rho}}^{(2)}}{||\bar{u}^{(1)}_{\lambda_{1,\rho}} -\bar{u}^{(2)}_{\lambda_{2,\rho}}||_{L^{\infty}(\mathbb{R}^N)} +||\bar{v}^{(1)}_{\lambda_{1,\rho}}-\bar{v}^{(2)}_{\lambda_{2,\rho}}||_{L^{\infty}(\mathbb{R}^N)}}\leq
C  \sum _ { l = 1 } ^ { k } e ^ { - \tau \sqrt {\lambda_{1,\rho}} | x - \xi_{\rho,l} |  }.
\end{align*}
Moreover,
\begin{align*}
\Big|\frac{\beta}{\lambda_{1,\rho}} \bar{u}^{(1)}_{\lambda_{1,\rho}} (\bar{v}^{(1)}_{\lambda_{1,\rho}}+\bar{v}^{(2)}_{\lambda_{2,\rho}}) \kappa_{\rho}^{(2)}\Big |\leq C \sum _ { l = 1 } ^ { k } e ^ { - \tau \sqrt {\lambda_{1,\rho}} | x - \xi_{\rho,l} |  }
\end{align*}
and
\begin{align*}
\Big|\frac{\beta}{\lambda_{1,\rho}} \bar{v}^{(1)}_{\lambda_{1,\rho}} (\bar{u}^{(1)}_{\lambda_{1,\rho}}+\bar{u}^{(2)}_{\lambda_{2,\rho}}) \kappa_{\rho}^{(1)} \Big|\leq C \sum _ { l = 1 } ^ { k } e ^ { - \tau \sqrt {\lambda_{1,\rho}} | x - \xi_{\rho,l} |  }.
\end{align*}
Employing the comparison principle, we establish the result \eqref{531}. Subsequently, by $L^p$ estimate and the Sobolev embedding theorem, it follows that $\kappa_{\rho}^{(1)}$ and $\kappa_{\rho}^{(2)}$ satisfy \eqref{532}.
\end{proof}

We now estimate $\kappa_{\rho}^{(i)} $ in $\cup^{k}_{l=1} B_{\frac{R}{\sqrt{\lambda_{1,\rho}}}}(\xi_{\rho,l}^{(1)})$, $i=1,2$, $l=1,2,\cdots,k$.  Since $|\tilde{\kappa}_{l,\rho}^{(i)}|\leq 1$ for $i=1,2$ and $l=1,2,\cdots,k$, we suppose
\begin{align*}
\tilde{\kappa}_{l,\rho}^{(i)} \rightarrow \tilde{\kappa}_{l}^{(i)},\quad \text{in}\quad C_{loc}^{1}(\mathbb{R}^N).
\end{align*}
Define $L_{mn}$, $m,n=1,2$, and $\mathcal{Q}$ as follows
\begin{align*}
L_{11} \tilde{\kappa}_l^{(1)}= &-\triangle \tilde{\kappa}_l^{(1)}+ \tilde{\kappa}_l^{(1)}-(3\mu_1(w^{*})^2+\beta (w^{\star})^2 )\tilde{\kappa}_l^{(1)} \\
&\quad+\frac{2 w^{*} }{k\int_{\mathbb{R}^N}\left(\left( w^{*}\right)^2+\left( w^{\star}\right)^2\right)} \sum_{i=1}^{k} \int_{ \mathbb{R}^N} (\mu_1 (w^{*})^3+\beta w^{*}(w^{\star})^2 )\tilde{\kappa}_i^{(1)} ,\\
L_{22} \tilde{\kappa}_l^{(2)}=&-\triangle \tilde{\kappa}_l^{(2)}+ \tilde{\kappa}_l^{(2)} -(3\mu_2(w^{\star})^2+\beta (w^{*})^2 )\tilde{\kappa}_l^{(2)} \\
&\quad+ \frac{2 w^{\star}}{k\int_{\mathbb{R}^N}\left(\left( w^{*}\right)^2+\left( w^{\star}\right)^2\right)} \sum_{i=1}^{k} \int_{ \mathbb{R}^N} (\mu_2  (w^{\star})^3+\beta (w^{*})^2w^{\star} ) \tilde{\kappa}_i^{(2)} ,\\
L_{12} \tilde{\kappa}_l^{(2)}=& - 2\beta w^{*}w^{\star}\tilde{\kappa}_l^{(2)}+\frac{2 w^{*} }{k\int_{\mathbb{R}^N}\left(\left( w^{*}\right)^2+\left( w^{\star}\right)^2\right)} \sum_{i=1}^{k} \int_{ \mathbb{R}^N} (\mu_2 (w^{\star})^3+\beta (w^{*})^2 w^{\star} )\tilde{\kappa}_i^{(2)},\\
L_{21}\tilde{\kappa}_l^{(1)}=& -2\beta w^{*}w^{\star}\tilde{\kappa}_l^{(1)}+ \frac{2 w^{\star}}{k\int_{\mathbb{R}^N}\left(\left( w^{*}\right)^2+\left( w^{\star}\right)^2\right)} \sum_{i=1}^{k} \int_{ \mathbb{R}^N} (\mu_1 (w^{*})^3+\beta w^{*}(w^{\star})^2 )\tilde{\kappa}_i^{(1)}
\end{align*}
and
\[
\mathcal{Q} \begin{pmatrix}
\tilde{\kappa}_l^{(1)} \\
\tilde{\kappa}_l^{(2)}
\end{pmatrix} =
\begin{pmatrix}
L_{11} & L_{12} \\
L_{21} & L_{22}
\end{pmatrix}
\begin{pmatrix}
\tilde{\kappa}_l^{(1)} \\
\tilde{\kappa}_l^{(2)}
\end{pmatrix}.
\]

\begin{lemma}
If $\mathcal{Q} (\tilde{\kappa}_l^{(1)}, \tilde{\kappa}_l^{(2)})^T =0$, $l=1,2,\cdots,k$, then
\begin{align*}
\begin{pmatrix}
\tilde{\kappa}_l^{(1)} \\
\tilde{\kappa}_l^{(2)}
\end{pmatrix}
=
\begin{pmatrix}
\sum_{j=0}^{N} a_{l,j} \vartheta_j^{(1)}  \\
\sum_{j=0}^{N} a_{l,j} \vartheta_j^{(2)}
\end{pmatrix} ,\quad l=1,2,\cdots,k,
\end{align*}
where $a_{l,j}$ are some constants, $ \vartheta_j^{(1)}=\frac{\partial w^{*}}{\partial x_j}$,  $ \vartheta_j^{(2)}=\frac{\partial w^{\star}}{\partial x_j}$, $j=1,\cdots,N$. Moreover, $a_{i,0}=a_{j,0}$ for all $i,j=1,2,\cdots,k$.
\end{lemma}

\begin{proof}
 Since $w^{*}$ and $w^{\star}$ are radially symmetric functions, using the technique of the separation of variables, we can obtain
\begin{align*}
\begin{pmatrix}
\tilde{\kappa}_l^{(1)} \\
\tilde{\kappa}_l^{(2)}
\end{pmatrix}
=
\begin{pmatrix}
\sum_{j=1}^{N} a_{l,j} \vartheta_j^{(1)} + \tilde{\kappa}_{l,0}^{(1)} \\
\sum_{j=1}^{N} a_{l,j} \vartheta_j^{(2)} + \tilde{\kappa}_{l,0}^{(2)}
\end{pmatrix} ,\quad l=1,2,\cdots,k,
\end{align*}
where $\tilde{\kappa}_{l,0}^{(1)}$ and $\tilde{\kappa}_{l,0}^{(2)}$ are radial functions, satisfying
\begin{align}
-\triangle \tilde{\kappa}_{l,0}^{(1)}+ \tilde{\kappa}_{l,0}^{(1)}=&(3\mu_1(w^{*})^2+\beta (w^{\star})^2 )\tilde{\kappa}_{l,0}^{(1)} -\frac{2 w^{*} }{k\int_{\mathbb{R}^N}\left(\left( w^{*}\right)^2+\left( w^{\star}\right)^2\right)} \sum_{i=1}^{k} \int_{ \mathbb{R}^N} (\mu_1 (w^{*})^3+\beta w^{*}(w^{\star})^2 )\tilde{\kappa}_i^{(1)} \notag\\
&+2\beta w^{*}w^{\star}\tilde{\kappa}_{l,0}^{(2)}-\frac{2 w^{*} }{k\int_{\mathbb{R}^N}\left(\left( w^{*}\right)^2+\left( w^{\star}\right)^2\right)} \sum_{i=1}^{k} \int_{ \mathbb{R}^N} (\mu_2 (w^{\star})^3+\beta (w^{*})^2 w^{\star} )\tilde{\kappa}_i^{(2)},\label{k10}\\
-\triangle \tilde{\kappa}_{l,0}^{(2)}+ \tilde{\kappa}_{l,0}^{(2)} =&(3\mu_2(w^{\star})^2+\beta (w^{*})^2 )\tilde{\kappa}_{l,0}^{(2)} - \frac{2 w^{\star}}{k\int_{\mathbb{R}^N}\left(\left( w^{*}\right)^2+\left( w^{\star}\right)^2\right)} \sum_{i=1}^{k} \int_{ \mathbb{R}^N} (\mu_2  (w^{\star})^3+\beta (w^{*})^2w^{\star} ) \tilde{\kappa}_i^{(2)} \notag\\
&+2\beta w^{*}w^{\star}\tilde{\kappa}_{l,0}^{(1)}- \frac{2 w^{\star}}{k\int_{\mathbb{R}^N}\left(\left( w^{*}\right)^2+\left( w^{\star}\right)^2\right)} \sum_{i=1}^{k} \int_{ \mathbb{R}^N} (\mu_1 (w^{*})^3+\beta w^{*}(w^{\star})^2 )\tilde{\kappa}_i^{(1)}.\label{k20}
\end{align}
We set $\vartheta_0^{(1)}=w^{*}+x\cdot \nabla w^{*}$, $\vartheta_0^{(2)}=w^{\star}+x\cdot \nabla w^{\star}$, and
\begin{align*}
\bar{\mathcal{Q}} \begin{pmatrix}
u \\
v
\end{pmatrix}:=
\begin{pmatrix}
\bar{L}_{11} & \bar{L}_{12} \\
\bar{L}_{21} & \bar{L}_{22}
\end{pmatrix}
\begin{pmatrix}
u \\
v
\end{pmatrix},
\end{align*}
where
\begin{align*}
\bar{L}_{11}(u)&:=-\triangle u(x) + u(x) - (3\mu_1(w^{*})^2+\beta (w^{\star})^2)u(x), \\
\bar{L}_{12}(v)&:=- 2\beta w^{*}w^{\star} v(x),\\
\bar{L}_{22}(v)&:=-\triangle v(x) + v(x) - (3\mu_2(w^{\star})^2+\beta (w^{*})^2)v(x) ,\\
\bar{L}_{21}(u)&:=- 2\beta w^{*}w^{\star} u(x).
\end{align*}
Then
\begin{align}\label{QVV}
\bar{\mathcal{Q}} \begin{pmatrix}
\vartheta_0^{(1)} \\
\vartheta_0^{(2)}
\end{pmatrix}=
\begin{pmatrix}
- 2 w^{*} \\
- 2 w^{\star}
\end{pmatrix}.
\end{align}
Since $\bar{\mathcal{Q}}$ no non-trivial bounded radially symmetric kernel, it holds
\begin{align}\label{k1k2}
\begin{pmatrix}
\tilde{\kappa}_{l,0}^{(1)}\\
\tilde{\kappa}_{l,0}^{(2)}
\end{pmatrix}=
a_{l,0}
\begin{pmatrix}
\vartheta_0^{(1)} \\
\vartheta_0^{(2)}
\end{pmatrix},\quad l=1,2,\cdots,k.
\end{align}
Therefore, we get
\begin{align*}
\begin{pmatrix}
\tilde{\kappa}_l^{(1)} \\
\tilde{\kappa}_l^{(2)}
\end{pmatrix}
=
\begin{pmatrix}
\sum_{j=0}^{N} a_{l,j} \vartheta_j^{(1)}  \\
\sum_{j=0}^{N} a_{l,j} \vartheta_j^{(2)}
\end{pmatrix} ,\quad l=1,2,\cdots,k.
\end{align*}
Since $w^{*}$ and $w^{\star}$ are even functions,  $\frac{\partial w^{*}}{\partial x_j}$ and $\frac{\partial w^{\star}}{\partial x_j}$ are odd functions, we have
\begin{align*}
\int_{ \mathbb{R}^N} (\mu_1 (w^{*})^3+\beta w^{*}(w^{\star})^2 )\tilde{\kappa}_l^{(1)} &= \int_{ \mathbb{R}^N} (\mu_1 (w^{*})^3+\beta w^{*}(w^{\star})^2 )\tilde{\kappa}_{l,0}^{(1)},\\
\int_{ \mathbb{R}^N} (\mu_2  (w^{\star})^3+\beta (w^{*})^2w^{\star} ) \tilde{\kappa}_l^{(2)} &= \int_{ \mathbb{R}^N} (\mu_2  (w^{\star})^3+\beta (w^{*})^2w^{\star} ) \tilde{\kappa}_{l,0}^{(2)},
\quad l=1,2,\cdots,k.
\end{align*}
Substituting \eqref{k1k2} into \eqref{k10} and \eqref{k20}, then combining with \eqref{QVV}, we can obtain
\begin{align*}
a_{l,0}
=&\frac{\sum_{i=1}^{k} a_{i,0}   }{k\int_{\mathbb{R}^N}\left(\left( w^{*}\right)^2+\left( w^{\star}\right)^2\right)} \Big( \int_{ \mathbb{R}^N} ((\mu_1 (w^{*})^3+\beta w^{*}(w^{\star})^2 )\vartheta_0^{(1)}\\
&\quad\quad\quad\quad\quad\quad\quad\quad\quad\quad+\int_{ \mathbb{R}^N} (\mu_2 (w^{\star})^3+\beta (w^{*})^2 w^{\star} )\vartheta_0^{(2)}\Big).
\end{align*}
Hence
\begin{align}\label{abi0}
a_{i,0}=a_{j,0}\quad \text{for all} \quad i,j=1,2,\cdots,k.
\end{align}
\end{proof}

We are aimed to prove $\begin{pmatrix}
\tilde{\kappa}_l^{(1)} \\
\tilde{\kappa}_l^{(2)}
\end{pmatrix}=0$. For this purpose, we write
$\tilde{\kappa}_{l,\rho}^{(1)}$ and $\tilde{\kappa}_{l,\rho}^{(2)}$ into
\begin{align}\label{5-31}
\begin{pmatrix}
\tilde{\kappa}_{l,\rho}^{(1)} \\
\tilde{\kappa}_{l,\rho}^{(2)}
\end{pmatrix}
=
\begin{pmatrix}
\sum_{j=0}^{N} a_{l,j,\rho} \vartheta_j^{(1)} + \hat{\kappa}_{l,\rho}^{(1)} \\
\sum_{j=0}^{N} a_{l,j,\rho} \vartheta_j^{(2)} + \hat{\kappa}_{l,\rho}^{(2)}
\end{pmatrix},
\end{align}
where
\begin{align*}
(\hat{\kappa}_{l,\rho}^{(1)}, \hat{\kappa}_{l,\rho}^{(2)})\in (\hat{E}_1, \hat{E}_2):=\{(u, v)\in H^1\times H^1:\langle(u, v),(\vartheta_j^{(1)}, \vartheta_j^{(2)})\rangle=0,j=0,1,\cdots,N\}.
\end{align*}
In fact, it is standard to prove the following results.
\begin{align}\label{Quv}
\left|\left|\mathcal{Q} (u,v)^T \right|\right| \geq C
\left|\left|(u,v)^T \right|\right|,
\end{align}
where $||\cdot||$ is the regular norm for $H^1(\mathbb{R}^N)\times H^1(\mathbb{R}^N)$.

On the other hand,
\begin{align*}
\left|\left|\mathcal{Q} (\hat{\kappa}_{l,\rho}^{(1)}, \hat{\kappa}_{l,\rho}^{(2)})^T\right|\right|=\left|\left|\mathcal{Q} (\tilde{\kappa}_{l,\rho}^{(1)}, \tilde{\kappa}_{l,\rho}^{(2)})^T\right|\right|=O\big(\lambda_{1,\rho}^{-1}\big),
\end{align*}
which, combined with \eqref{Quv}, we have
\begin{align}\label{kk1}
\left|\left|(\hat{\kappa}_{i,\rho}^{(1)}, \hat{\kappa}_{i,\rho}^{(2)})^T\right|\right|= O\big(\lambda_{1,\rho}^{-1}\big).
\end{align}

Now we assume that
\begin{align*}
 a_{l,j,\rho} \rightarrow a_{l,j}
\end{align*}
as $\rho\rightarrow 0$ for $N=3$ or $\rho\rightarrow \rho_0$ for $N=2$, and $l=1,2,\cdots,k$, $j=0,1,\cdots,N.$ Then we have the following estimates.
\subsubsection{The non-critical case (N=3)}
\begin{lemma}\label{lem55}
It holds that $a_{l,0}=0$, $l=1,2,\cdots,k$.
\end{lemma}
\begin{proof}
By \eqref{509}-\eqref{512},  we have
\begin{align}\label{55-1}
&\int_{B_{\delta}(\xi_{\rho,l}^{(1)})} (\bar{u}_{\lambda_{1,\rho}}^{(1)}+\bar{u}_{\lambda_{2,\rho}}^{(2)})\kappa_{\rho}^{(1)} \notag\\
=&\lambda_{1,\rho}^{-\frac12}\int_{ B_{\delta\sqrt{\lambda_{1,\rho}}}(0)}  2w^{*}\left(\sum_{j=0}^{3} a_{l,j,\rho} \vartheta_j^{(1)} + \hat{\kappa}_{i,\rho}^{(1)}\right) + O(\lambda_{1,\rho}^{-1}) \notag\\
=&\lambda_{1,\rho}^{-\frac12}a_{l,0,\rho}\int_{ \mathbb{R}^3}  2w^{*}\vartheta_0^{(1)} + O(\lambda_{1,\rho}^{-1}) \notag\\
=&\lambda_{1,\rho}^{-\frac12}a_{l,0,\rho}\int_{ \mathbb{R}^3}  2w^{*} (w^{*}+x\cdot \nabla w^{*})+ O(\lambda_{1,\rho}^{-1}) \notag\\
=&\lambda_{1,\rho}^{-\frac12}a_{l,0,\rho}\left(2\int_{ \mathbb{R}^3}  (w^{*})^2+\sum^{3}_{i=1}\int_{ \mathbb{R}^3 } x_i\cdot \frac{\partial}{\partial x_i} (w^{*})^2\right)+ O(\lambda_{1,\rho}^{-1}) \notag\\
=&-\lambda_{1,\rho}^{-\frac12}a_{l,0,\rho} \int_{ \mathbb{R}^3}  (w^{*})^2 + O\big(\lambda_{1,\rho}^{-1}\big)
\end{align}
and
\begin{align}\label{55-2}
\int_{B_{\delta}(\xi_{\rho,l}^{(1)})} (\bar{v}_{\lambda_{1,\rho}}^{(1)}+\bar{v}_{\lambda_{2,\rho}}^{(2)})\kappa_{\rho}^{(2)}
=-\lambda_{1,\rho}^{-\frac12}a_{l,0,\rho} \int_{ \mathbb{R}^3}  (w^{\star})^2 + O\big(\lambda_{1,\rho}^{-1}\big).
\end{align}
Noting that
\begin{align*}
\int_{\mathbb{R}^3} \left((\bar{u}_{\lambda_{1,\rho}}^{(1)}+\bar{u}_{\lambda_{2,\rho}}^{(2)})\kappa_{\rho}^{(1)}+ (\bar{v}_{\lambda_{1,\rho}}^{(1)}+\bar{v}_{\lambda_{2,\rho}}^{(2)})\kappa_{\rho}^{(2)} \right)=\frac{\int_{\mathbb{R}^3} ((\bar{u}_{\lambda_{1,\rho}}^{(1)})^2 + (\bar{v}_{\lambda_{1,\rho}}^{(1)})^2) - \int_{\mathbb{R}^3} ((\bar{u}_{\lambda_{2,\rho}}^{(2)})^2 + (\bar{v}_{\lambda_{2,\rho}}^{(2)})^2)}{\|\bar{u}_{\lambda_{1,\rho}}^{(1)} - \bar{u}_{\lambda_{2,\rho}}^{(2)}\|_{L^\infty(\mathbb{R}^3)} + \|\bar{v}_{\lambda_{1,\rho}}^{(1)} - \bar{v}_{\lambda_{2,\rho}}^{(2)}\|_{L^\infty(\mathbb{R}^3)}} =0,
\end{align*}
we have
\begin{align}\label{55-4}
&\sum^{k}_{l=1} \int_{B_{\delta}(\xi_{\rho,l}^{(1)})} \left( (\bar{u}_{\lambda_{1,\rho}}^{(1)}+\bar{u}_{\lambda_{2,\rho}}^{(2)})\kappa_{\rho}^{(1)} + (\bar{v}_{\lambda_{1,\rho}}^{(1)}+\bar{v}_{\lambda_{2,\rho}}^{(2)})\kappa_{\rho}^{(2)} \right)\notag\\
=&\int_{\bigcup_{l=1}^k B^{c}_{\delta}(\xi_{\rho,l}^{(1)})} \left( (\bar{u}_{\lambda_{1,\rho}}^{(1)}+\bar{u}_{\lambda_{2,\rho}}^{(2)})\kappa_{\rho}^{(1)} + (\bar{v}_{\lambda_{1,\rho}}^{(1)}+\bar{v}_{\lambda_{2,\rho}}^{(2)})\kappa_{\rho}^{(2)} \right)\notag\\
=&O\Big(e^{-\theta \sqrt{\lambda_{1,\rho}}}\Big).
\end{align}
Combining \eqref{55-1}, \eqref{55-2} and \eqref{55-4}, we have
\begin{align*}
k\lambda_{1,\rho}^{-\frac12}\left(\int_{ \mathbb{R}^3} ((w^{*})^2+ (w^{\star})^2) \right) \sum^{k}_{l=1} a_{l,0,\rho} = O\big(\lambda_{1,\rho}^{-1}\big).
\end{align*}
 Then
\begin{align*}
 \sum^{k}_{l=1} a_{l,0,\rho} = o(1),
\end{align*}
and from \eqref{abi0}, we can obtain that
\begin{align*}
  a_{l,0} = 0,\quad l=1,2,\cdots,k.
\end{align*}
\end{proof}

\begin{remark}
From the computations in \eqref{55-1}, we note that for \( N = 2 \), the leading-order term in \eqref{55-1} vanishes. Accordingly, the sub-leading term includes additional terms beyond \( a_{l,0,\rho} \). This composition of the sub-leading term prevents us from establishing \( a_{l,0} = 0 \), so the proof strategy employed in Lemma \ref{lem55} is no longer applicable here.
\end{remark}

\begin{lemma}\label{lem56}
There holds that $a_{l,j}=0$, $l=1,2,\cdots,k$, $j=1,2,3$.
\end{lemma}

\begin{proof}
Applying Pohozaev identity \eqref{poho2}, by Lemma \ref{lem54}, we have
\begin{align*}
\int_{B_{\delta}(\xi^{(1)}_{\rho,l})}\left(\frac{\partial P(x)}{\partial x_j}(\bar{u}_{\lambda_{1,\rho}}^{(1)}+\bar{u}_{\lambda_{2,\rho}}^{(2)})\kappa_{\rho}^{(1)} +\frac{\partial Q(x)}{\partial x_j}(\bar{v}_{\lambda_{1,\rho}}^{(1)}+\bar{v}_{\lambda_{2,\rho}}^{(2)})\kappa_{\rho}^{(2)}  \right)=O(e ^ { -  \tau \sqrt{\lambda_{1,\rho}}}).
\end{align*}
Therefore, we have
\begin{align}\label{56-2}
&\int_{B_{\delta\sqrt{\lambda_{1,\rho}}}(0)} \left(\frac{\partial P(\frac{x}{\sqrt{\lambda_{1,\rho}}}+\xi_{\rho,l}^{(1)})}{\partial x_j} (\lambda_{1,\rho}^{-\frac12}\bar{u}_{\lambda_{1,\rho}}^{(1)}(\frac{x}{\sqrt{\lambda_{1,\rho}}} +\xi_{\rho,l}^{(1)}) +\lambda_{1,\rho}^{-\frac12}\bar{u}_{\lambda_{2,\rho}}^{(2)}(\frac{x}{\sqrt{\lambda_{1,\rho}}} +\xi_{\rho,l}^{(1)}))\tilde{\kappa}_{l,\rho}^{(1)}\right. \notag\\
&\quad\quad\quad\quad\quad +\left. \frac{\partial Q((\frac{x}{\sqrt{\lambda_{1,\rho}}}+\xi_{\rho,l}^{(1)}))}{\partial x_j} (\lambda_{1,\rho}^{-\frac12}\bar{v}_{\lambda_{1,\rho}}^{(1)}(\frac{x}{\sqrt{\lambda_{1,\rho}}} +\xi_{\rho,l}^{(1)}) +\lambda_{1,\rho}^{-\frac12}\bar{v}_{\lambda_{2,\rho}}^{(2)}(\frac{x}{\sqrt{\lambda_{1,\rho}}} +\xi_{\rho,l}^{(1)}))\tilde{\kappa}_{i,\rho}^{(2)}  \right)\notag\\
=&O\Big( e ^ { -  \tau \sqrt{\lambda_{1,\rho}}}\Big).
\end{align}
Since $|\xi_{\rho,i}^{(1)}-\xi_l|=O(\lambda_{1,\rho}^{-1})$, we can obtain that
\begin{align*}
\frac{\partial P(\xi_{\rho,l}^{(1)})}{\partial x_j}=\frac{\partial P(\xi_l)}{\partial x_j}+O(|\xi_{\rho,l}^{(1)}-\xi_l|)=O\big(\lambda_{1,\rho}^{-1}\big)
\end{align*}
and
\begin{align*}
\frac{\partial Q(\xi_{\rho,l}^{(1)})}{\partial x_j}=\frac{\partial Q(\xi_l)}{\partial x_j}+O(|\xi_{\rho,l}^{(1)}-\xi_l|)
=O\big(\lambda_{1,\rho}^{-1}\big),
\end{align*}
which implies that
\begin{align}\label{56-3}
\int_{B_{\delta\sqrt{\lambda_{1,\rho}}}(0)} &\left(\frac{\partial P(\xi_{\rho,l}^{(1)})}{\partial x_j} (\lambda_{1,\rho}^{-\frac12}\bar{u}_{\lambda_{1,\rho}}^{(1)}(\frac{x}{\sqrt{\lambda_{1,\rho}}}+\xi_{\rho,l}^{(1)}) +\lambda_{1,\rho}^{-\frac12}\bar{u}_{\lambda_{2,\rho}}^{(2)}(\frac{x}{\sqrt{\lambda_{1,\rho}}}+\xi_{\rho,l}^{(1)}))\tilde{\kappa}_{i,\rho}^{(1)}\right. \notag\\
&\quad +\left. \frac{\partial Q(\xi_{\rho,l}^{(1)})}{\partial x_j} (\lambda_{1,\rho}^{-\frac12}\bar{v}_{\lambda_{1,\rho}}^{(1)}(\frac{x}{\sqrt{\lambda_{1,\rho}}}+\xi_{\rho,l}^{(1)}) +\lambda_{1,\rho}^{-\frac12}\bar{v}_{\lambda_{2,\rho}}^{(2)}(\frac{x}{\sqrt{\lambda_{1,\rho}}}+\xi_{\rho,l}^{(1)}))\tilde{\kappa}_{i,\rho}^{(2)}  \right)=O\big(\lambda_{1,\rho}^{-1}\big).
\end{align}
Combining \eqref{56-2} and \eqref{56-3}, we get
\begin{align*}
&\int_{B_{\delta\sqrt{\lambda_{1,\rho}}}(0)} \left(\left(\frac{\partial P(\frac{x}{\sqrt{\lambda_{1,\rho}}}+\xi_{\rho,l}^{(1)})}{\partial x_j}- \frac{\partial P(\xi_{\rho,l}^{(1)})}{\partial x_j} \right) (\lambda_{1,\rho}^{-\frac12}\bar{u}_{\lambda_{1,\rho}}^{(1)}(\frac{x}{\sqrt{\lambda_{1,\rho}}}+\xi_{\rho,l}^{(1)}) +\lambda_{1,\rho}^{-\frac12}\bar{u}_{\lambda_{2,\rho}}^{(2)}(\frac{x}{\sqrt{\lambda_{1,\rho}}}+\xi_{\rho,l}^{(1)}))\tilde{\kappa}_{l,\rho}^{(1)}\right. \notag\\
&\quad\quad\quad\quad\quad +\left. \left(\frac{\partial Q((\frac{x}{\sqrt{\lambda_{1,\rho}}}+\xi_{\rho,l}^{(1)}))}{\partial x_j}- \frac{\partial Q(\xi_{\rho,l}^{(1)})}{\partial x_j} \right) (\lambda_{1,\rho}^{-\frac12}\bar{v}_{\lambda_{1,\rho}}^{(1)}(\frac{x}{\sqrt{\lambda_{1,\rho}}}+\xi_{\rho,l}^{(1)}) +\lambda_{1,\rho}^{-\frac12}\bar{v}_{\lambda_{2,\rho}}^{(2)}(\frac{x}{\sqrt{\lambda_{1,\rho}}}+\xi_{\rho,l}^{(1)}))\tilde{\kappa}_{l,\rho}^{(2)}  \right) \notag\\
=&O\big(\lambda_{1,\rho}^{-1}\big),
\end{align*}
which also implies that
\begin{align}\label{56-5}
\sum^{3}_{h=1}\int_{B_{\delta\sqrt{\lambda_{1,\rho}}}(0)} &\left(\frac{\partial^2 P(\xi_{\rho,l}^{(1)})}{\partial x_j \partial x_h} x_h (\lambda_{1,\rho}^{-\frac12}\bar{u}_{\lambda_{1,\rho}}^{(1)}(\frac{x}{\sqrt{\lambda_{1,\rho}}}+\xi_{\rho,l}^{(1)}) +\lambda_{1,\rho}^{-\frac12}\bar{u}_{\lambda_{2,\rho}}^{(2)}(\frac{x}{\sqrt{\lambda_{1,\rho}}}+\xi_{\rho,l}^{(1)}))\tilde{\kappa}_{l,\rho}^{(1)}\right. \notag\\
&\quad +\left. \frac{\partial^2 Q(\xi_{\rho,l}^{(1)})}{\partial x_j \partial x_h} x_h  (\lambda_{1,\rho}^{-\frac12}\bar{v}_{\lambda_{1,\rho}}^{(1)}(\frac{x}{\sqrt{\lambda_{1,\rho}}}+\xi_{\rho,l}^{(1)}) +\lambda_{1,\rho}^{-\frac12}\bar{v}_{\lambda_{2,\rho}}^{(2)}(\frac{x}{\sqrt{\lambda_{1,\rho}}}+\xi_{\rho,l}^{(1)}))\tilde{\kappa}_{l,\rho}^{(2)}  \right)=O(\lambda_{1,\rho}^{-1}).
\end{align}
Letting $\rho\rightarrow 0$ in \eqref{56-5}, we obtain that
\begin{align*}
 \int _ {  \mathbb{R}^3 } \Bigg(p_{lj} x _ { j } w^{*} \sum_{i=1}^3 a_{l,i}  \frac { \partial w^{*} } { \partial x _ { i } }+q_{lj}x _ { j } w^{\star} \sum_{i=1}^3 a_{l,i}  \frac { \partial w^{\star} } { \partial x _ { i } } \Bigg)=0.
\end{align*}

By the symmetry and Condition $(H_{2})$, we obtain that $a_{l,h}=0$, $h=1,2,3$, $l=1,\cdots,k$.
\end{proof}

\subsubsection{The critical case (N=2)}
\begin{lemma}\label{lem57}
It holds that $a_{l,j}=0$, $l=1,2,\cdots,k$, $j=1,2$.
\end{lemma}

\begin{proof}
Similar to the proof of Lemma \ref{lem56}, we can obtain that
\begin{align}\label{57-1}
\sum^{2}_{h=1}\int_{B_{\delta\sqrt{\lambda_{1,\rho}}}(0)} &\left(\frac{\partial^2 P(\xi_{\rho,l}^{(1)})}{\partial x_j \partial x_h} x_h (\lambda_{1,\rho}^{-\frac12}\bar{u}_{\lambda_{1,\rho}}^{(1)}(\frac{x}{\sqrt{\lambda_{1,\rho}}}+\xi_{\rho,l}^{(1)}) +\lambda_{1,\rho}^{-\frac12}\bar{u}_{\lambda_{2,\rho}}^{(2)}(\frac{x}{\sqrt{\lambda_{1,\rho}}}+\xi_{\rho,l}^{(1)}))\tilde{\kappa}_{l,\rho}^{(1)}\right. \notag\\
&\quad +\left. \frac{\partial^2 Q(\xi_{\rho,l}^{(1)})}{\partial x_j \partial x_h} x_h  (\lambda_{1,\rho}^{-\frac12}\bar{v}_{\lambda_{1,\rho}}^{(1)}(\frac{x}{\sqrt{\lambda_{1,\rho}}}+\xi_{\rho,l}^{(1)}) +\lambda_{1,\rho}^{-\frac12}\bar{v}_{\lambda_{2,\rho}}^{(2)}(\frac{x}{\sqrt{\lambda_{1,\rho}}}+\xi_{\rho,l}^{(1)}))\tilde{\kappa}_{l,\rho}^{(2)}  \right)dx=O\big(\lambda_{1,\rho}^{-1}\big).
\end{align}
Letting $\rho\rightarrow \rho_0$ in \eqref{57-1}, we obtain that
\begin{align*}
& \int _ {  \mathbb{R}^2 } \Bigg(p_{lj} x _ { j } w^{*} \sum_{i=1}^2 a_{l,i}  \frac { \partial w^{*} } { \partial x _ { i } }+q_{lj} x _ { j } w^{\star} \sum_{i=1}^2 a_{l,i}  \frac { \partial w^{\star} } { \partial x _ { i } } \Bigg)dx\\
+& \frac{1} { \beta ^ { 2 } - \mu_1 \mu_2 } \big[( \beta - \mu_2 ) p_{lj}  +( \beta - \mu_1 )q_{lj} \big] a_{l,0}\int _ {  \mathbb{R}^2 } \big[ x _ { j } w   (w+x\cdot \nabla w) \big]dx=0.
\end{align*}
Since
\begin{align*}
\int _ {  \mathbb{R}^2 } \big[ x _ { j } w   (w+x\cdot \nabla w) \big]dx=0,\quad h=1,2,
\end{align*}
we know
\begin{align}\label{57-2}
\frac{1} { \beta ^ { 2 } - \mu_1 \mu_2 } \big[( \beta - \mu_2 )p_{lj}  +( \beta - \mu_1 )q_{lj} \big] a_{l,j} \int _ {  \mathbb{R}^2 } x _ { j } w   \frac { \partial w } { \partial x _ { j } }=0.
\end{align}
Thus, by Condition $(H_2)$, we obtain that $a_{l,j}=0$, $j=1,2$, $l=1,\cdots,k$.
\end{proof}

\begin{lemma}\label{lem58}
There holds that $a_{l,0}=0$, $l=1,2,\cdots,k$.
\end{lemma}

\begin{proof}
For solutions \((\bar{u}_{\lambda_{\rho}}^{(i)}, \bar{v}_{\lambda_{\rho}}^{(i)})\) of \eqref{eq0}, multiplying the first equation by \(\langle x - \xi_{\rho,l}^{(1)}, \nabla\bar{u}_{\lambda_{1,\rho}}^{(1)} \rangle\) and the second by \(\langle x - \xi_{\rho,l}^{(1)}, \nabla\bar{v}_{\lambda_{1,\rho}}^{(1)} \rangle\), summing these results, and integrating on $B_{\delta}(\xi_{\rho,l}^{(1)})$, we have
\begin{align}\label{58-1}
&\int_{B_{\delta} (\xi_{\rho,l}^{(1)})} \Bigg[\left((P(x)+\lambda_{1,\rho})+\frac12\langle \nabla P(x),x - \xi_{\rho,l}^{(1)}\rangle \right)(\bar{u}_{\lambda_{1,\rho}}^{(1)})^2- \frac12 \mu_1 (\bar{u}_{\lambda_{1,\rho}}^{(1)})^4 - \beta (\bar{u}_{\lambda_{1,\rho}}^{(1)})^2 (\bar{v}_{\lambda_{1,\rho}}^{(1)})^2 \notag\\
&\quad\quad\quad+\left((Q(x)+\lambda_{1,\rho})+\frac12\langle \nabla Q(x),x - \xi_{\rho,l}^{(1)}\rangle \right)(\bar{v}_{\lambda_{1,\rho}}^{(1)})^2- \frac12 \mu_2 (\bar{v}_{\lambda_{1,\rho}}^{(1)})^4 \Bigg] \notag\\
=&\int_{\partial B_{\delta} (\xi_{\rho,l}^{(1)})} \Bigg[-\frac{\partial \bar{u}_{\lambda_{1,\rho}}^{(1)}}{\partial n}\langle x - \xi_{\rho,l}^{(1)}, \nabla \bar{u}_{\lambda_{1,\rho}}^{(1)} \rangle -\frac{\partial \bar{v}_{\lambda_{1,\rho}}^{(1)}}{\partial n}\langle x - \xi_{\rho,l}^{(1)}, \nabla \bar{v}_{\lambda_{1,\rho}}^{(1)} \rangle+\frac12\Bigg(|\nabla \bar{u}_{\lambda_{1,\rho}}^{(1)}|^2+ (P(x)+\lambda_{1,\rho})(\bar{u}_{\lambda_{1,\rho}}^{(1)})^2 \notag\\
&\quad\quad\quad-\frac{\mu_1}{2}( \bar{u}_{\lambda_{1,\rho}}^{(1)})^4 +|\nabla \bar{v}_{\lambda_{1,\rho}}^{(1)}|^2+ (Q(x)+\lambda_{1,\rho})(\bar{v}_{\lambda_{1,\rho}}^{(1)})^2-\frac{\mu_2}{2}( \bar{v}_{\lambda_{1,\rho}}^{(1)})^4 - \beta ( \bar{u}_{\lambda_{1,\rho}}^{(1)})^2 ( \bar{v}_{\lambda_{1,\rho}}^{(1)})^2\Bigg)\langle x - \xi_{\rho,l}^{(1)}, n\rangle\Bigg]dS,
\end{align}
similarly,
\begin{align}\label{58-2}
&\int_{B_{\delta} (\xi_{\rho,l}^{(1)})} \Bigg[\left((P(x)+\lambda_{2,\rho})+\frac12\langle \nabla P(x),x - \xi_{\rho,l}^{(1)}\rangle \right)(\bar{u}_{\lambda_{2,\rho}}^{(2)})^2- \frac12 \mu_1 (\bar{u}_{\lambda_{2,\rho}}^{(2)})^4 - \beta (\bar{u}_{\lambda_{2,\rho}}^{(2)})^2 (\bar{v}_{\lambda_{2,\rho}}^{(2)})^2 \notag\\
&\quad\quad\quad+\left((Q(x)+\lambda_{2,\rho})+\frac12\langle \nabla Q(x),x - \xi_{\rho,l}^{(1)}\rangle \right)(\bar{v}_{\lambda_{2,\rho}}^{(2)})^2- \frac12 \mu_2 (\bar{v}_{\lambda_{2,\rho}}^{(2)})^4 \Bigg] \notag\\
=&\int_{\partial B_{\delta} (\xi_{\rho,l}^{(1)})} \Bigg[-\frac{\partial \bar{u}_{\lambda_{2,\rho}}^{(2)}}{\partial n}\langle x - \xi_{\rho,l}^{(1)}, \nabla \bar{u}_{\lambda_{2,\rho}}^{(2)} \rangle -\frac{\partial \bar{v}_{\lambda_{2,\rho}}^{(2)}}{\partial n}\langle x - \xi_{\rho,l}^{(1)}, \nabla \bar{v}_{\lambda_{2,\rho}}^{(2)} \rangle+\frac12\Bigg(|\nabla \bar{u}_{\lambda_{2,\rho}}^{(2)}|^2+ (P(x)+\lambda_{2,\rho})(\bar{u}_{\lambda_{2,\rho}}^{(2)})^2 \notag\\
&\quad\quad\quad-\frac{\mu_1}{2}( \bar{u}_{\lambda_{2,\rho}}^{(2)})^4 +|\nabla \bar{v}_{\lambda_{2,\rho}}^{(2)}|^2+ (Q(x)+\lambda_{2,\rho})(\bar{v}_{\lambda_{2,\rho}}^{(2)})^2-\frac{\mu_2}{2}( \bar{v}_{\lambda_{2,\rho}}^{(2)})^4 - \beta ( \bar{u}_{\lambda_{2,\rho}}^{(2)})^2 ( \bar{v}_{\lambda_{2,\rho}}^{(2)})^2\Bigg)\langle x - \xi_{\rho,l}^{(1)}, n\rangle\Bigg]dS.
\end{align}
Then by \eqref{rho2}, \eqref{58-1} and \eqref{58-2}, we have
\begin{align}\label{58-3}
&\sum_{l=1}^k \int_{B_{\delta} (\xi_{\rho,l}^{(1)})} \Bigg[\lambda_{1,\rho}^{-1} (P(x)  +\frac12\langle \nabla P(x),x - \xi_{\rho,l}^{(1)}\rangle )   (\bar{u}_{\lambda_{1,\rho}}^{(1)}+\bar{u}_{\lambda_{2,\rho}}^{(2)})\kappa_{\rho}^{(1)} \notag\\
&\quad\quad\quad - \frac{1}{2}\lambda_{1,\rho}^{-1} \mu_1 (\bar{u}_{\lambda_{1,\rho}}^{(1)}+\bar{u}_{\lambda_{2,\rho}}^{(2)}) ((\bar{u}_{\lambda_{1,\rho}}^{(1)})^2+(\bar{u}_{\lambda_{2,\rho}}^{(2)})^2) \kappa_{\rho}^{(1)} \notag\\
&\quad\quad\quad- \beta \lambda_{1,\rho}^{-1} ((\bar{v}_{\lambda_{2,\rho}}^{(2)})^2 (\bar{u}_{\lambda_{1,\rho}}^{(1)}+\bar{u}_{\lambda_{2,\rho}}^{(2)}) \kappa_{\rho}^{(1)} + (\bar{u}_{\lambda_{1,\rho}}^{(1)})^2(\bar{v}_{\lambda_{1,\rho}}^{(1)} +\bar{v}_{\lambda_{2,\rho}}^{(2)})\kappa_{\rho}^{(2)}) \notag\\
&\quad\quad\quad+\lambda_{1,\rho}^{-1} (Q(x)+\frac12\langle \nabla Q(x),x - \xi_{\rho,l}^{(1)}\rangle )  (\bar{v}_{\lambda_{1,\rho}}^{(1)}+\bar{v}_{\lambda_{2,\rho}}^{(2)}) \kappa_{\rho}^{(2)} \notag\\
&\quad\quad\quad - \frac{1}{2}\lambda_{1,\rho}^{-1} \mu_2 (\bar{v}_{\lambda_{1,\rho}}^{(1)}+\bar{v}_{\lambda_{2,\rho}}^{(2)} ) ((\bar{v}_{\lambda_{1,\rho}}^{(1)})^2+(\bar{v}_{\lambda_{2,\rho}}^{(2)} )^2)\kappa_{\rho}^{(2)} \notag\\
&\quad\quad\quad +\frac{\lambda_{1,\rho}^{-1}  (\lambda_{1,\rho}-\lambda_{2,\rho}) \big((\bar{u}_{\lambda_{2,\rho}}^{(2)})^2+(\bar{v}_{\lambda_{2,\rho}}^{(2)})^2\big) }{\|\bar{u}_{\lambda_{1,\rho}}^{(1)} - \bar{u}_{\lambda_{2,\rho}}^{(2)}\|_{L^\infty(\mathbb{R}^2)} + \|\bar{v}_{\lambda_{1,\rho}}^{(1)} - \bar{v}_{\lambda_{2,\rho}}^{(2)}\|_{L^\infty(\mathbb{R}^2)}}+ \big((\bar{u}_{\lambda_{1,\rho}}^{(1)} +\bar{u}_{\lambda_{2,\rho}}^{(2)})\kappa_{\rho}^{(1)} + (\bar{v}_{\lambda_{1,\rho}}^{(1)} +\bar{v}_{\lambda_{2,\rho}}^{(2)})\kappa_{\rho}^{(2)} \big)\bigg]\notag\\
=& O( e^{-\theta\sqrt{\lambda_{1,\rho}}}).
\end{align}
By a similar argument as in the proof of Lemma \ref{lem53}, we have
\begin{align}\label{58-31}
&\frac{\lambda_{1,\rho}^{-1} (\lambda_{1,\rho}-\lambda_{2,\rho})  }{\|\bar{u}_{\lambda_{1,\rho}}^{(1)} - \bar{u}_{\lambda_{2,\rho}}^{(2)}\|_{L^\infty(\mathbb{R}^2)} + \|\bar{v}_{\lambda_{1,\rho}}^{(1)} - \bar{v}_{\lambda_{2,\rho}}^{(2)}\|_{L^\infty(\mathbb{R}^2)}}\notag\\
=&\rho^{-2}\sum_{l=1}^k \int_{B_{\delta} (\xi_{\rho,l}^{(1)})}  (\frac{\lambda_{2,\rho}}{\lambda_{1,\rho}}-1) (\bar{u}_{\lambda_{2,\rho}}^{(2)} \kappa_{\rho}^{(1)} +\bar{v}_{\lambda_{2,\rho}}^{(2)} \kappa_{\rho}^{(2)} ) \notag\\
&+ \rho^{-2}\lambda_{1,\rho}^{-1}\sum_{l=1}^k \int_{B_{\delta} (\xi_{\rho,l}^{(1)})}(\mu_1  ( (\bar{u}_{\lambda_{1,\rho}}^{(1)})^2\bar{u}_{\lambda_{2,\rho}}^{(2)}+ \bar{u}_{\lambda_{1,\rho}}^{(1)}(\bar{u}_{\lambda_{2,\rho}}^{(2)})^2 )\kappa_{\rho}^{(1)} + \mu_2  ( (\bar{v}_{\lambda_{1,\rho}}^{(1)})^2\bar{v}_{\lambda_{2,\rho}}^{(2)}+ \bar{v}_{\lambda_{1,\rho}}^{(1)}(\bar{v}_{\lambda_{2,\rho}}^{(2)})^2 )\kappa_{\rho}^{(2)})\notag \\
&+\rho^{-2} \beta \lambda_{1,\rho}^{-1}\sum_{l=1}^k \int_{B_{\delta} (\xi_{\rho,l}^{(1)})} \Big(\bar{u}_{\lambda_{2,\rho}}^{(2)}((\bar{v}_{\lambda_{1,\rho}}^{(1)})^2-(\bar{v}_{\lambda_{2,\rho}}^{(2)})^2)\kappa_{\rho}^{(1)}  +(\bar{u}_{\lambda_{2,\rho}}^{(2)})^2(\bar{v}_{\lambda_{1,\rho}}^{(1)}+\bar{v}_{\lambda_{2,\rho}}^{(2)})\kappa_{\rho}^{(2)}  \notag\\
&\qquad\qquad\qquad+\bar{v}_{\lambda_{2,\rho}}^{(2)}((\bar{u}_{\lambda_{1,\rho}}^{(1)})^2 -(\bar{u}_{\lambda_{2,\rho}}^{(2)})^2)\kappa_{\rho}^{(2)}  +(\bar{v}_{\lambda_{2,\rho}}^{(2)})^2(\bar{u}_{\lambda_{1,\rho}}^{(1)} +\bar{u}_{\lambda_{2,\rho}}^{(2)})\kappa_{\rho}^{(1)} \Big) +O\big( e^{-\theta\sqrt{\lambda_{1,\rho}}}\big)\notag.\\
\end{align}
By substituting \eqref{58-31} into \eqref{58-3}, it follows that
\begin{align}\label{58-32}
&\sum_{l=1}^k \int_{B_{\delta} (\xi_{\rho,l}^{(1)})} \Bigg[\lambda_{1,\rho}^{-1} (P(x)  +\frac12\langle \nabla P(x),x - \xi_{\rho,l}^{(1)}\rangle )   (\bar{u}_{\lambda_{1,\rho}}^{(1)}+\bar{u}_{\lambda_{2,\rho}}^{(2)})\kappa_{\rho}^{(1)} \notag\\
&\quad\quad\quad+\lambda_{1,\rho}^{-1} (Q(x)+\frac12\langle \nabla Q(x),x - \xi_{\rho,l}^{(1)}\rangle )  (\bar{v}_{\lambda_{1,\rho}}^{(1)}+\bar{v}_{\lambda_{2,\rho}}^{(2)}) \kappa_{\rho}^{(2)} \notag\\
&\quad\quad\quad - \frac{1}{2}\lambda_{1,\rho}^{-1} \mu_1 (\bar{u}_{\lambda_{1,\rho}}^{(1)}+\bar{u}_{\lambda_{2,\rho}}^{(2)}) ((\bar{u}_{\lambda_{1,\rho}}^{(1)}-\bar{u}_{\lambda_{2,\rho}}^{(2)})^2) \kappa_{\rho}^{(1)} \notag\\
&\quad\quad\quad - \frac{1}{2}\lambda_{1,\rho}^{-1} \mu_2 (\bar{v}_{\lambda_{1,\rho}}^{(1)}+\bar{v}_{\lambda_{2,\rho}}^{(2)} ) ((\bar{v}_{\lambda_{1,\rho}}^{(1)}-\bar{v}_{\lambda_{2,\rho}}^{(2)} )^2)\kappa_{\rho}^{(2)} \notag\\
&\quad\quad\quad+ \beta \lambda_{1,\rho}^{-1} (\bar{u}_{\lambda_{2,\rho}}^{(2)}((\bar{v}_{\lambda_{1,\rho}}^{(1)})^2-(\bar{v}_{\lambda_{2,\rho}}^{(2)})^2 ) \kappa_{\rho}^{(1)} + \bar{v}_{\lambda_{1,\rho}}^{(1)}((\bar{u}_{\lambda_{2,\rho}}^{(2)})^2-(\bar{u}_{\lambda_{1,\rho}}^{(1)})^2)\kappa_{\rho}^{(2)}) \notag\\
&\quad\quad\quad + \big((\bar{u}_{\lambda_{1,\rho}}^{(1)} +\frac{\lambda_{2,\rho}}{\lambda_{1,\rho}}\bar{u}_{\lambda_{2,\rho}}^{(2)})\kappa_{\rho}^{(1)} + (\bar{v}_{\lambda_{1,\rho}}^{(1)} +\frac{\lambda_{2,\rho}}{\lambda_{1,\rho}}\bar{v}_{\lambda_{2,\rho}}^{(2)})\kappa_{\rho}^{(2)} \big)\bigg]\notag\\
=& O\big( e^{-\theta\sqrt{\lambda_{1,\rho}}}\big).
\end{align}
By Lemma \ref{lem51}, we have

\begin{align*}
 \lambda_{1,\rho}-\lambda_{2,\rho}=O(\lambda_{1,\rho}^{-3}).
\end{align*}
Note that
\begin{align}\label{58-4}
&\lambda_{1,\rho}^{-1}\sum_{l=1}^k \int_{B_{\delta} (\xi_{\rho,l}^{(1)})}\Bigg[  (P(x)  +\frac12\langle \nabla P(x),x - \xi_{\rho,l}^{(1)}\rangle )   (\bar{u}_{\lambda_{1,\rho}}^{(1)}+\bar{u}_{\lambda_{2,\rho}}^{(2)})\kappa_{\rho}^{(1)}\notag\\
&\quad\qquad\qquad\qquad+ (Q(x)+\frac12\langle \nabla Q(x),x - \xi_{\rho,l}^{(1)}\rangle ) (\bar{v}_{\lambda_{1,\rho}}^{(1)}+\bar{v}_{\lambda_{2,\rho}}^{(2)}) \kappa_{\rho}^{(2)}\Bigg]\notag\\
=&\lambda_{1,\rho}^{-1} \sum_{l=1}^k \int_{B_{\delta} (\xi_{\rho,l}^{(1)})}\Bigg[ \big( (P(x)-P(\xi_{l}) ) +\frac12\langle \nabla P(x),x - \xi_{\rho,l}^{(1)}\rangle \big)   (\bar{u}_{\lambda_{1,\rho}}^{(1)}+\bar{u}_{\lambda_{2,\rho}}^{(2)})\kappa_{\rho}^{(1)}\notag\\
&\quad\qquad\qquad\qquad+\big( (Q(x)-Q(\xi_l))+\frac12\langle \nabla Q(x),x - \xi_{\rho,l}^{(1)}\rangle \big) (\bar{v}_{\lambda_{1,\rho}}^{(1)}+\bar{v}_{\lambda_{2,\rho}}^{(2)}) \kappa_{\rho}^{(2)}\notag\\
&\quad\qquad\qquad\qquad+\Big(P(\xi_{l})(\bar{u}_{\lambda_{1,\rho}}^{(1)}+\bar{u}_{\lambda_{2,\rho}}^{(2)}) \kappa_{\rho}^{(1)} +Q(\xi_l) (\bar{v}_{\lambda_{1,\rho}}^{(1)}+\bar{v}_{\lambda_{2,\rho}}^{(2)}) \kappa_{\rho}^{(2)}\Big)\Bigg].
\end{align}
By Lemma \ref{lem52} and \eqref{abi0}, we obtain
\begin{align}\label{58-41}
&\lambda_{1,\rho}^{-1} \sum_{l=1}^k \int_{B_{\delta} (\xi_{\rho,l}^{(1)})}\Bigg[ (P(x)-P(\xi_{l}) ) (\bar{u}_{\lambda_{1,\rho}}^{(1)}+\bar{u}_{\lambda_{2,\rho}}^{(2)})\kappa_{\rho}^{(1)} +(Q(x)-Q(\xi_l)) (\bar{v}_{\lambda_{1,\rho}}^{(1)}+\bar{v}_{\lambda_{2,\rho}}^{(2)}) \kappa_{\rho}^{(2)}\Bigg]\notag\\
=&2\lambda_{1,\rho}^{-\frac32} a_{l,0,\rho} \sum_{l=1}^k \int_{\mathbb{R}^2}\Bigg[ \Big(P(\frac{x}{\sqrt{\lambda_{1,\rho}}}+\xi_{l,\varepsilon}^{(1)})-P(\xi_{l}) \Big) w^*(w^{*}+x\cdot \nabla w^{*})  \notag\\
&\quad\qquad\qquad\qquad+(Q(\frac{x}{\sqrt{\lambda_{1,\rho}}}+\xi_{l,\varepsilon}^{(1)})-Q(\xi_l)) w^{\star}(w^{\star}+x\cdot \nabla w^{\star})\Bigg]+O\big(\lambda_{1,\rho}^{-\frac72}\big)\notag\\
=&2\lambda_{1,\rho}^{-\frac52} a_{l,0,\rho} \Big(\sum_{l=1}^k \sum_{i=1}^{2} (\sigma_1^2 p_{li}+ \sigma_2^2 q_{li})\Big) \int_{\mathbb{R}^2} |x|^2 w(w+ x\cdot \nabla w)  +O\big(\lambda_{1,\rho}^{-\frac72}\big)
\end{align}
and
\begin{align}\label{58-42}
&\frac12\lambda_{1,\rho}^{-1} \int_{B_{\delta}  (\xi_{\rho,l}^{(1)})} \Big(\langle \nabla P(x),x - \xi_{\rho,l}^{(1)}\rangle    (\bar{u}_{\lambda_{1,\rho}}^{(1)}+\bar{u}_{\lambda_{2,\rho}}^{(2)})\kappa_{\rho}^{(1)}+\langle \nabla Q(x),x - \xi_{\rho,l}^{(1)}\rangle  (\bar{v}_{\lambda_{1,\rho}}^{(1)}+\bar{v}_{\lambda_{2,\rho}}^{(2)}) \kappa_{\rho}^{(2)}\Big)\notag\\
=&2\lambda_{1,\rho}^{-\frac52} a_{l,0,\rho} \Big(\sum_{l=1}^k\sum_{i=1}^{2} (\sigma_1^2 p_{li}+ \sigma_2^2 q_{li})\Big)\int_{\mathbb{R}^2} |x|^2 w(w+x\cdot \nabla w) +O\big(\lambda_{1,\rho}^{-\frac72}\big),
\end{align}
in addition, by the assumption $P(\xi_{l}) =Q(\xi_{l})=0$, we have
\begin{align}\label{58-43}
&\lambda_{1,\rho}^{-1} \sum_{l=1}^k \int_{B_{\delta} (\xi_{\rho,l}^{(1)})} \Big(P(\xi_{l})(\bar{u}_{\lambda_{1,\rho}}^{(1)}+\bar{u}_{\lambda_{2,\rho}}^{(2)}) \kappa_{\rho}^{(1)} +Q(\xi_l) (\bar{v}_{\lambda_{1,\rho}}^{(1)}+\bar{v}_{\lambda_{2,\rho}}^{(2)}) \kappa_{\rho}^{(2)}\Big)=0.
\end{align}
Combining \eqref{58-4}-\eqref{58-43}, we obtain
\begin{align}\label{58-44}
&\lambda_{1,\rho}^{-1}\sum_{l=1}^k \int_{B_{\delta} (\xi_{\rho,l}^{(1)})}\Bigg[  (P(x)  +\frac12\langle \nabla P(x),x - \xi_{\rho,l}^{(1)}\rangle )   (\bar{u}_{\lambda_{1,\rho}}^{(1)}+\bar{u}_{\lambda_{2,\rho}}^{(2)})\kappa_{\rho}^{(1)}\notag\\
&\quad\qquad\qquad\qquad+ (Q(x)+\frac12\langle \nabla Q(x),x - \xi_{\rho,l}^{(1)}\rangle ) (\bar{v}_{\lambda_{1,\rho}}^{(1)}+\bar{v}_{\lambda_{2,\rho}}^{(2)}) \kappa_{\rho}^{(2)}\Bigg]\notag\\
=&4\lambda_{1,\rho}^{-\frac52} a_{l,0,\rho} \Big(\sum_{l=1}^k\sum_{i=1}^{2} (\sigma_1^2 p_{li}+ \sigma_2^2 q_{li})\Big)\int_{\mathbb{R}^2} |x|^2 w(w+x\cdot \nabla w) +O(\lambda_{1,\rho}^{-\frac72}).
\end{align}
Moreover,
\begin{align}\label{58-5}
&-\frac{1}{2}\lambda_{1,\rho}^{-1} \mu_1  \sum_{l=1}^k \int_{B_{\delta} (\xi_{\rho,l}^{(1)})} (\bar{u}_{\lambda_{1,\rho}}^{(1)}+\bar{u}_{\lambda_{2,\rho}}^{(2)}) \big((\bar{u}_{\lambda_{1,\rho}}^{(1)}-\bar{u}_{\lambda_{2,\rho}}^{(2)})^2\big) \kappa_{\rho}^{(1)} \notag\\
=&-\frac{1}{2}\lambda_{1,\rho}^{-\frac32} \mu_1  \sum_{l=1}^k \int_{\mathbb{R}^2} 2w^{*}(x)\Bigg(\frac{\lambda_{1,\rho}-\lambda_{2,\rho}}{ \sqrt{\lambda_{1,\rho}}+\sqrt{\lambda_{2,\rho}}} w^*(x) \notag\\
&\qquad\qquad\qquad+\sqrt{\lambda_{2,\rho}}\Big(w^*(x)-w^*\big(\sqrt{\frac{\lambda_{2,\rho}}{\lambda_{1,\rho}}} x + \sqrt{\lambda_{2,\rho}}(\xi_{\rho,l}^{(1)}-\xi_{\rho,l}^{(2)})\big)\Big)\Bigg)^2(w^*+ x\cdot \nabla w^*) +O(\lambda_{1,\rho}^{-4})\notag\\
=&O\big(\lambda_{1,\rho}^{-\frac72}\big),
\end{align}
similarly,
\begin{align}\label{58-6}
- \frac{1}{2}\lambda_{1,\rho}^{-1} \mu_2 \sum_{l=1}^k \int_{B_{\delta} (\xi_{\rho,l}^{(1)})} (\bar{v}_{\lambda_{1,\rho}}^{(1)}+\bar{v}_{\lambda_{2,\rho}}^{(2)} ) \big((\bar{v}_{\lambda_{1,\rho}}^{(1)}-\bar{v}_{\lambda_{2,\rho}}^{(2)} )^2\big)\kappa_{\rho}^{(2)} =O(\lambda_{1,\rho}^{-\frac72})
\end{align}
and
\begin{align}\label{58-7}
&\beta \lambda_{1,\rho}^{-1} \sum_{l=1}^k \int_{B_{\delta} (\xi_{\rho,l}^{(1)})}  \Big(\bar{u}_{\lambda_{2,\rho}}^{(2)}((\bar{v}_{\lambda_{1,\rho}}^{(1)})^2 -(\bar{v}_{\lambda_{2,\rho}}^{(2)})^2 ) \kappa_{\rho}^{(1)} + \bar{v}_{\lambda_{1,\rho}}^{(1)}((\bar{u}_{\lambda_{2,\rho}}^{(2)})^2 - (\bar{u}_{\lambda_{1,\rho}}^{(1)})^2)\kappa_{\rho}^{(2)}\Big) \notag\\
=&\beta \lambda_{1,\rho}^{-1} \sum_{l=1}^k \int_{B_{\delta}(\xi_{\rho,l}^{(1)})} (\bar{v}_{\lambda_{2,\rho}}^{(2)}-\bar{v}_{\lambda_{1,\rho}}^{(1)}) (\bar{u}_{\lambda_{1,\rho}}^{(1)} \bar{v}_{\lambda_{1,\rho}}^{(1)}- \bar{u}_{\lambda_{2,\rho}}^{(2)} \bar{v}_{\lambda_{2,\rho}}^{(2)})\kappa_{\rho}^{(1)}\notag\\
=&\beta \lambda_{1,\rho}^{-2} k \int_{\mathbb{R}^2} \Bigg(\frac{\lambda_{2,\rho}-\lambda_{1,\rho}}{ \sqrt{\lambda_{2,\rho}}+\sqrt{\lambda_{1,\rho}}} w^\star(x)+\sqrt{\lambda_{2,\rho}}\Big(w^\star\big(\sqrt{\frac{\lambda_{2,\rho}}{\lambda_{1,\rho}}} x + \sqrt{\lambda_{2,\rho}}(\xi_{\rho,l}^{(1)}-\xi_{\rho,l}^{(2)})\big)-w^\star(x)\Big)\Bigg) \notag\\
&\qquad\qquad\times\Bigg((\lambda_{1,\rho}-\lambda_{2,\rho})w^*(x)w^\star(x)+\lambda_{2,\rho} w^*(x)w^\star(x)\notag\\
&\qquad\qquad\qquad - \lambda_{2,\rho}w^*\big(\sqrt{\frac{\lambda_{2,\rho}}{\lambda_{1,\rho}}} x + \sqrt{\lambda_{2,\rho}}(\xi_{\rho,l}^{(1)}-\xi_{\rho,l}^{(2)})\big)w^\star\big(\sqrt{\frac{\lambda_{2,\rho}}{\lambda_{1,\rho}}} x + \sqrt{\lambda_{2,\rho}}(\xi_{\rho,l}^{(1)}-\xi_{\rho,l}^{(2)})\big)  \Bigg)     \notag\\
&\qquad\qquad\times (w^*+ x\cdot \nabla w^*) \notag\\
=&O\big(\lambda_{1,\rho}^{-\frac72}\big).
\end{align}
By \eqref{rho2} and \eqref{511}, we have
\begin{align}\label{58-8}
&\sum_{l=1}^k \int_{B_{\delta} (\xi_{\rho,l}^{(1)})}
\Big((\bar{u}_{\lambda_{1,\rho}}^{(1)} +\frac{\lambda_{2,\rho}}{\lambda_{1,\rho}}\bar{u}_{\lambda_{2,\rho}}^{(2)})\kappa_{\rho}^{(1)} + (\bar{v}_{\lambda_{1,\rho}}^{(1)} +\frac{\lambda_{2,\rho}}{\lambda_{1,\rho}}\bar{v}_{\lambda_{2,\rho}}^{(2)})\kappa_{\rho}^{(2)} \Big)\notag\\
=&\sum_{l=1}^k \int_{B_{\delta} (\xi_{\rho,l}^{(1)})}
\Big(\frac{\lambda_{2,\rho}-\lambda_{1,\rho}}{\lambda_{1,\rho}}\bar{u}_{\lambda_{2,\rho}}^{(2)}\kappa_{\rho}^{(1)} + \frac{\lambda_{2,\rho}-\lambda_{1,\rho}}{\lambda_{1,\rho}}\bar{v}_{\lambda_{2,\rho}}^{(2)}\kappa_{\rho}^{(2)} \Big) +O( e^{-\theta\sqrt{\lambda_{1,\rho}}})\notag\\
=& \lambda_{1,\rho} ^{-\frac12}\sum_{l=1}^k \int_{\mathbb{R}^2}
\Bigg(\frac{\lambda_{2,\rho}-\lambda_{1,\rho}}{\lambda_{1,\rho}} \Big(w^{*}(x)+ \big(\frac{\sqrt{\lambda_{2,\rho}}-\sqrt{\lambda_{1,\rho}}}{\sqrt{\lambda_{1,\rho}}} x +\sqrt{\lambda_{2,\rho}}(\xi_{\rho,i}^{(1)}-\xi_{\rho,l}^{(2)})\big)\cdot \nabla w^{*}(x)\Big) (w^*+ x\cdot \nabla w^*) \notag\\
&\qquad+ \frac{\lambda_{2,\rho}-\lambda_{1,\rho}}{\lambda_{1,\rho}}\sqrt{\frac{\lambda_{2,\rho}}{\lambda_{1,\rho}}} \Big(w^{\star}(x)+ \big(\frac{\sqrt{\lambda_{2,\rho}}-\sqrt{\lambda_{1,\rho}}}{\sqrt{\lambda_{1,\rho}}} x +\sqrt{\lambda_{2,\rho}}(\xi_{\rho,i}^{(1)}-\xi_{\rho,l}^{(2)})\big)\cdot \nabla w^{\star}(x)\Big) (w^{\star}+ x\cdot \nabla w^{\star}) \Bigg) \notag\\
&+O\big(\lambda_{1,\rho}^{-\frac72}\big)\notag\\
=&O\big(\lambda_{1,\rho}^{-\frac72}\big)
\end{align}
By combining \eqref{58-44}-\eqref{58-8} and then letting \( \rho \to \rho_0 \), it follows that
\begin{align*}
a_{l,0} \Big(\sum_{l=1}^k\sum_{i=1}^{2} (\sigma_1^2p_{li}+ \sigma_2^2 q_{li})\Big)\int_{\mathbb{R}^2} |x|^2 w(w+ x\cdot \nabla w)=0.
\end{align*}
This means that $a_{l,0}=0$.
\end{proof}

Finally, we prove Theorem \ref{main1}.
\begin{proof}[Proof of Theorem \ref{main1}]
On the one hand, Lemma \ref{lem54} shows that
\begin{align*}
\kappa_{\rho}^{(i)}(x)=o(1), \quad x\in \mathbb{R}^N \backslash \cup^{k}_{l=1} B_{\frac{R}{\sqrt{\lambda_{1,\rho}}}}(\xi_{\rho,l}^{(1)}) ,\quad i=1,2.
\end{align*}
On the other hand, by \eqref{5-31}, \eqref{kk1}, Lemmas \ref{lem55} to \ref{lem58}, we have
\begin{align*}
\kappa_{\rho}^{(i)}(x)=o(1), \quad x\in \cup^{k}_{l=1} B_{\frac{R}{\sqrt{\lambda_{1,\rho}}}}(\xi_{\rho,l}^{(1)}),\quad i=1,2.
\end{align*}
This contradicts $||\kappa_{\rho}^{(1)}||_{L^{\infty}(\mathbb{R}^N)}+||\kappa_{\rho}^{(2)}||_{L^{\infty}(\mathbb{R}^N)}=1$. The proof of Theorem \ref{main1} is completed.
\end{proof}

\end{document}